\documentclass[12pt,reqno]{amsart}
\usepackage{fullpage}
\usepackage{times}
\usepackage[active]{srcltx}
\usepackage{graphicx}

\newtheorem{theorem}{Theorem}[section]
\newtheorem{lemma}[theorem]{Lemma}
\newtheorem{prop}[theorem]{Proposition}
\newtheorem{cor}[theorem]{Corollary}
\theoremstyle{remark}
\newtheorem{definition}[theorem]{Definition}
\newcommand{\qedef}{\hfill$\diamondsuit$}
\newtheorem*{remark}{Remark}
\numberwithin{equation}{section}
\renewcommand{\theenumi}{\alph{enumi}}

\renewcommand{\mod}[1]{{\ifmmode\text{\rm\ (mod~$#1$)}\else\discretionary{}{}{\hbox{ }}\rm(mod~$#1$)\fi}}
\newcommand{\ep}{\varepsilon}
\newsymbol\nmid 232D
\newcommand{\exdiv}{\mathrel{\|}}
\newcommand{\stirling}[2]{\genfrac{\{}{\}}{0pt}{}{#1}{#2}}

\newcommand{\C}{{\mathbb C}}
\newcommand{\E}{{\mathbb{E}}}
\newcommand{\F}{{\mathcal F}}
\renewcommand{\L}{{\mathcal L}}
\newcommand{\M}{{\mathcal M}}
\newcommand{\N}{{\mathbb N}}
\newcommand{\Q}{{\mathbb Q}}
\newcommand{\R}{{\mathbb R}}
\newcommand{\Rcal}{{\mathcal R}}
\newcommand{\Z}{{\mathbb Z}}
\renewcommand{\Re}{\mathop{\rm Re}}
\renewcommand{\Im}{\mathop{\rm Im}}
\newcommand{\Erf}{\mathop{\rm Erf}}
\newcommand{\Obar}{\overline{O}}

\newcommand{\Lsum}[1]{\sum_{\substack{#1 \\ L(1/2+i\gamma,\chi)=0}}}
\newcommand{\Lbarsum}[1]{\sum_{\substack{#1 \\ L(1/2+i\gamma,\bar\chi)=0}}}
\newcommand{\mostchisum}{\sum_{\substack{\chi\mod q \\ \chi\ne\chi_0}}}
\newcommand{\allchisum}{\sum_{\chi\mod q}}
\newcommand{\Lprod}[1]{\prod_{\substack{#1 \\ L(1/2+i\gamma,\chi)=0}}}

\newcommand{\allchiprod}{\prod_{\chi\mod q}}

\begin{document}

\title[Inequities in the Shanks--R\'enyi Prime Number~Race]{Inequities in the Shanks--R\'enyi Prime Number~Race: An asymptotic formula for the densities}
\author[Daniel Fiorilli and Greg Martin]{Daniel Fiorilli}
\address{D\'epartement de math\'ematiques et de statistique \\ Universit\'e de Montr\'eal \\ CP 6128, succ.\ Centre-ville \\ Montr\'eal, QC \\ Canada H3C 3J7}
\email{fiorilli@dms.umontreal.ca}
\author{Greg Martin}
\address{Department of Mathematics \\ University of British Columbia \\ Room
121, 1984 Mathematics Road \\ Canada V6T 1Z2}
\email{gerg@math.ubc.ca}
\maketitle

\begin{abstract}
Chebyshev was the first to observe a bias in the distribution of primes in residue classes. The general phenomenon is that if $a$ is a nonsquare\mod q and $b$ is a square\mod q, then there tend to be more primes congruent to $a\mod q$ than $b\mod q$ in initial intervals of the positive integers; more succinctly, there is a tendency for $\pi(x;q,a)$ to exceed $\pi(x;q,b)$. Rubinstein and Sarnak defined $\delta(q;a,b)$ to be the logarithmic density of the set of positive real numbers $x$ for which this inequality holds; intuitively, $\delta(q;a,b)$ is the ``probability'' that $\pi(x;q,a) > \pi(x;q,b)$ when $x$ is ``chosen randomly''. In this paper, we establish an asymptotic series for $\delta(q;a,b)$ that can be instantiated with an error term smaller than any negative power of $q$. This asymptotic formula is written in terms of a variance $V(q;a,b)$ that is originally defined as an infinite sum over all nontrivial zeros of Dirichlet $L$-functions corresponding to characters\mod q; we show how $V(q;a,b)$ can be evaluated exactly as a finite expression. In addition to providing the exact rate at which $\delta(q;a,b)$ converges to $\frac12$ as $q$ grows, these evaluations allow us to compare the various density values $\delta(q;a,b)$ as $a$ and $b$ vary modulo $q$; by analyzing the resulting formulas, we can explain and predict which of these densities will be larger or smaller, based on arithmetic properties of the residue classes $a$ and $b\mod q$. For example, we show that if $a$ is a prime power and $a'$ is not, then $\delta(q;a,1) < \delta(q;a',1)$ for all but finitely many moduli $q$ for which both $a$ and $a'$ are nonsquares. Finally, we establish rigorous numerical bounds for these densities $\delta(q;a,b)$ and report on extensive calculations of them, including for example the determination of all 117 density values that exceed $\frac9{10}$.
\end{abstract}

\tableofcontents

\section{Introduction}
\label{intro}

We have known for over a century now that the prime numbers are asymptotically evenly distributed among the reduced residue classes modulo any fixed positive integer $q$. In other words, if $\pi(x;q,a)$ denotes the number of primes not exceeding $x$ that are congruent to $a\mod q$, then $\lim_{x\to\infty} \pi(x;q,a)/\pi(x;q,b) = 1$ for any integers $a$ and $b$ that are relatively prime to $q$. However, this information by itself is not enough to tell us about the distribution of values of the difference $\pi(x;q,a)-\pi(x;q,b)$, in particular whether this difference must necessarily take both positive and negative values. Several authors---notably Chebyshev in 1853 and Shanks \cite{shanks} in 1959---observed that $\pi(x;4,3)$ has an extremely strong tendency to be greater than $\pi(x;4,1)$, and similar biases exist for other moduli as well. The general phenomenon is that $\pi(x;q,a)$ tends to exceed $\pi(x;q,b)$ when $a$ is a nonsquare modulo $q$ and $b$ is a square modulo $q$.

In 1994, Rubinstein and Sarnak \cite{RS} developed a framework for studying these questions that has proven to be quite fruitful. Define $\delta(q;a,b)$ to be the logarithmic density of the set of real numbers $x\ge1$ satisfying $\pi(x;q,a)>\pi(x;q,b)$. (Recall that the logarithmic density of a set $S$ of positive real numbers is
\[
\lim_{X\to\infty} \bigg( \frac1{\log X} \int\limits_{\substack{1\le x\le X \\ x\in S}} \frac{dx}x \bigg),
\]
or equivalently the natural density of the set $\{\log x\colon x\in S\}$.)
Rubinstein and Sarnak investigated these densities under the following two hypotheses:
\begin{itemize}
\item The Generalized Riemann Hypothesis (GRH): all nontrivial zeros of Dirichlet $L$-functions have real part equal to $\frac12$
\item A linear independence hypothesis (LI): the nonnegative imaginary parts of these nontrivial zeros are linearly independent over the rationals
\end{itemize}
Under these hypotheses, they proved that the limit defining $\delta(q;a,b)$ always exists and is strictly between 0 and 1. Among other things, they also proved that $\delta(q;a,b)$ tends to $\frac12$ as $q$ tends to infinity, uniformly for all pairs $a,b$ of distinct reduced residues\mod q.

In the present paper, we examine these densities $\delta(q;a,b)$ more closely. We are particularly interested in a quantitative statement of the rate at which $\delta(q;a,b)$ approaches $\frac12$. In addition, computations show that for a fixed modulus $q$, the densities $\delta(q;a,b)$ vary as $a$ and $b$ range over nonsquares and squares modulo $q$, respectively. We are also interested in determining which pairs $a,b\mod q$ give rise to larger or smaller values of $\delta(q;a,b)$, and especially in giving criteria that depend as directly as possible on $a$ and $b$ rather than on analytic data such as the zeros of Dirichlet $L$-functions.

Our first theorem, which is proved in Section~\ref {derivation section}, exhibits an asymptotic series for $\delta(q;a,b)$:

\begin{theorem}
\label{delta series theorem}
Assume GRH and LI. Let $q$ be a positive integer, and let $\rho(q)$ be the function defined in Definition~\ref {rho def}. Let $a$ and $b$ be reduced residues\mod q such that $a$ is a nonsquare\mod q and $b$ is a square\mod q, and let $V(q;a,b)$ be the variance defined in Definition~\ref {bchi and Vqab def}. Then for any nonnegative integer $K$,
\begin{equation}
\label{asymptotic formula general case}
\delta(q;a,b) = \frac12 + \frac{\rho(q)}{\sqrt{2\pi V(q;a,b)}} \sum_{\ell=0}^K \frac1{V(q;a,b)^\ell} \sum_{j=0}^\ell \rho(q)^{2j} s_{q;a,b}(\ell,j) + O_K \bigg( \frac{\rho(q)^{2K+3}}{V(q;a,b)^{K+3/2}} \bigg),
\end{equation}
where the real numbers $s_{q;a,b}(\ell,j)$, which are bounded in absolute value by a function of $\ell$ uniformly in $q$, $a$, $b$, and $j$, are defined in Definition~\ref {s-coeffs def}. In particular, $s_{q;a,b}(0,0) = 1$, so that
\begin{equation}
\label{asymptotic formula K=0 case}
\delta(q;a,b) = \frac12 + \frac{\rho(q)}{\sqrt{2\pi V(q;a,b)}} + O \bigg( \frac{\rho(q)^3}{V(q;a,b)^{3/2}} \bigg).
\end{equation}
\end{theorem}

We will see in Proposition~\ref{Vqab asymptotic prop} that $V(q;a,b)\sim 2\phi(q) \log q$, and so the error term in equation~\eqref{asymptotic formula general case} is $\ll_{K,\ep} 1/q^{K+3/2-\ep} $.

The assumption that $a$ is a nonsquare\mod q and $b$ is a square\mod q is natural in this context, reflecting the bias observed by Chebyshev. Rubinstein and Sarnak showed (assuming GRH and LI) that $\delta(q;b,a) + \delta(q;a,b) = 1$; therefore if $a$ is a square\mod q and $b$ is a nonsquare\mod q, the right-hand sides of the asymptotic formulas~\eqref{asymptotic formula general case} and~\eqref{asymptotic formula K=0 case} become $\frac12 - \cdots$ instead of $\frac12 + \cdots$. Rubinstein and Sarnak also showed that $\delta(q;b,a) = \delta(q;a,b) = \frac12$ if $a$ and $b$ are both squares or both nonsquares\mod q.

The definitions of $\rho(q)$ and of $V(q;a,b)$ are as follows:

\begin{definition}
As usual, $\omega(q)$ denotes the number of distinct prime factors of $q$. Define $\rho(q)$ to be the number of real characters\mod q, or equivalently the index of the subgroup of squares in the full multiplicative group\mod q, or equivalently still the number of solutions of $x^2\equiv1\mod q$. An exercise in elementary number theory shows that
\begin{equation*}
\rho(q) = \begin{cases}
2^{\omega(q)}, &\text{if } 2\nmid q, \\
2^{\omega(q)-1}, &\text{if } 2\mid q \text{ but } 4\nmid q, \\
2^{\omega(q)}, &\text{if } 4\mid q \text{ but } 8\nmid q, \\
2^{\omega(q)+1}, &\text{if } 8\mid q, \\
\end{cases}
\end{equation*}
which implies that $\rho(q) \ll_\ep q^\ep$ for every $\ep>0$.
\qedef
\label{rho def}
\end{definition}

\begin{definition}
For any Dirichlet character $\chi\mod q$, define
\[
b(\chi) = \Lsum{\gamma\in\R} \frac1{\frac14+\gamma^2}.
\]
We adopt the convention throughout this paper that the zeros are listed with multiplicity in all such sums (though note that the hypothesis LI, when in force, implies that all such zeros are simple).
For any reduced residues $a$ and $b\mod q$, define
\begin{equation*}
V(q;a,b) = \allchisum |\chi(b)-\chi(a)|^2 b(\chi).
\end{equation*}
We will see in Proposition~\ref {variance prop} that $V(q;a,b)$ is the variance of a particular distribution associated with the difference $\pi(x;q,a) - \pi(x;q,b)$.
\qedef
\label{bchi and Vqab def}
\end{definition}

As the asymptotic series in Theorem~\ref {delta series theorem} depends crucially on the variance $V(q;a,b)$, we next give a formula for it (established in Section~\ref {variance formula section}) that involves only a finite number of easily computed quantities:

\begin{theorem}
\label{variance evaluation theorem}
Assume GRH. For any pair $a,b$ of distinct reduced residues modulo $q$,
\begin{equation*}
V(q;a,b) = 2\phi(q) \big( \L(q) + K_q(a-b) + \iota_q(-ab^{-1})\log2 \big) + 2M^*(q;a,b),
\end{equation*}
where the functions $\L$, $K_q$, and $\iota_q$ are defined in Definition~\ref{iota and Lq and Rqn def} and the quantity $M^*(q;a,b)$ is defined in Definition~\ref{Mqab def}.
\end{theorem}

The definitions of these three arithmetic functions and of the analytic quantity $M^*$ are as follows:

\begin{definition}
As usual, $\phi(q)$ denotes Euler's totient function, and $\Lambda(q)$ denotes the von Mangoldt function, which takes the value $\log p$ if $q$ is a power of the prime $p$ and 0 otherwise. For any positive integer $q$, define
\begin{equation*}
\L(q) = \log q - \sum_{p\mid q} \frac{\log p}{p-1} + \frac{\Lambda(q)}{\phi(q)} - (\gamma_0+\log2\pi),
\end{equation*}
where $\gamma_0 = \lim_{x\to\infty} \big( \sum_{n\le x} \frac1n - \log x \big)$ is Euler's constant; it can be easily shown that $\L(q)$ is positive when $q\ge43$. Note that $\L(q) = \log(q/2\pi e^{\gamma_0})$ when $q$ is prime and that $\L(q) = \log q + O(\log\log q)$ for any integer $q\ge3$. Also let
\begin{equation*}
\iota_q(n) = \begin{cases}
1, &\text{if }n\equiv1\mod q, \\
0, &\text{if }n\not\equiv1\mod q
\end{cases}
\end{equation*}
denote the characteristic function of the integers that are congruent to 1\mod q. Finally, define
\begin{equation*}
K_q(n) = \frac{\Lambda(q/(q,n))}{\phi(q/(q,n))} - \frac{\Lambda(q)}{\phi(q)}.
\end{equation*}
Note that these last two functions depend only on the residue class of $n$ modulo $q$. For this reason, in expressions such as $\iota_q(n^{-1})$ or $K_q(n^{-1})$, the argument $n^{-1}$ is to be interpreted as an integer that is the multiplicative inverse of $n\mod q$. In addition, note that $K_q(n) \ge0$, since the only way that the second term can contribute is if $q$ is a prime power, in which case the first term contributes at least as much. On the other hand $K_q$ is bounded above, since if $q/(q,n)$ is a power of the prime $p$ then $K_q(n) \le (\log p)/(p-1) \le \log2$. Note also that $K_q(n)=0$ when $(n,q)=1$.
\qedef
\label{iota and Lq and Rqn def}
\end{definition}

\begin{definition}
As usual, $L(s,\chi) = \sum_{n=1}^\infty \chi(n)n^{-s}$ denotes the $L$-function associated to the Dirichlet character $\chi$. Given such a character $\chi\mod q$, let $q^*$ denote its conductor (that is, the smallest integer $d$ such that $\chi$ is induced by a character modulo $d$), and let $\chi^*$ be the unique character modulo $q^*$ that induces $\chi$. Now define
\begin{equation*}
M^*(q;a,b) = \mostchisum |\chi(a)-\chi(b)|^2 \frac{L'(1,\chi^*)}{L(1,\chi^*)}
\end{equation*}
and
\begin{equation*}
M(q;a,b) = \mostchisum |\chi(a)-\chi(b)|^2 \frac{L'(1,\chi)}{L(1,\chi)}.
\end{equation*}
\qedef
\label{Mqab def}
\end{definition}

The formula for $V(q;a,b)$ in Theorem~\ref{variance evaluation theorem}  is exact and hence well suited for computations. For theoretical purposes, however, we need a better understanding of $M^*(q;a,b)$, which our next theorem (proved in Section~\ref {M evaluation section}) provides:

\begin{theorem}
Assume GRH. For any pair $a,b$ of distinct reduced residues modulo $q$, let $r_1$ and $r_2$ denote the least positive residues of $ab^{-1}$ and $ba^{-1}\mod q$, and let the quantity $H(q;a,b)$ be defined in Definition~\ref {hqpr def}. Then
\begin{equation*}
M^*(q;a,b) = \phi(q) \bigg( \frac{\Lambda(r_1)}{r_1} +\frac{\Lambda(r_2)}{r_2} + H(q;a,b) + O \bigg( \frac{\log^2 q}{q} \bigg) \bigg),
\end{equation*}
where the implied constant is absolute.
\label{M evaluation theorem}
\end{theorem}

(The unexpected appearance of the specific integers $r_1$ and $r_2$, in a formula for a quantity depending upon entire residue classes\mod q, is due to the approximation of infinite series by their first terms---see Proposition~\ref{M exact evaluation prop}.)
The quantity $H(q;a,b)$ is usually quite small, unless there is an extreme coincidence in the locations of $a$ and $b$ relative to the prime divisors of $q$, which would be reflected in a small value of the quantity $e(q;p,r)$ defined as follows:

\begin{definition}
Given an integer $q$ and a prime $p$, let $\nu \ge 0$ be the integer such that $ p^{\nu} \parallel q$ (that is, $p^\nu \mid q$ but $p^{\nu+1} \nmid q$). For any reduced residue $r\mod q$, define
$
e(q;p,r) = \min\{ e \ge 1 \colon p^{e} \equiv r^{-1} \mod{q/p^\nu} \},
$
and define
$$
h(q;p,r) = \frac{1}{\phi(p^{\nu})}\frac{\log p}{p^{e(q;p,r)}}.
$$
When $r$ is not in the multiplicative subgroup generated by $p\mod{q/p^\nu}$, we make the convention that $e(q;p,r) = \infty$ and $h(q;p,r) = 0$. Finally, for any integers $a$ and $b$, define
\[
H(q; a, b) = \sum_{p\mid q} \big( h(q;p,ab^{-1}) + h(q;p,ba^{-1}) \big).
\]
Note that if $q=p^\nu$ is a prime power, then $h(q;p,r) = (\log p)/p^\nu(p-1)$ is independent of~$r$, which implies that $H(q;a,b)\ll (\log q)/q$ when $q$ is a prime power.
\qedef
\label{hqpr def}
\end{definition}

The extremely small relative error in Theorem~\ref{delta series theorem} implies that the formula given therein is useful even for moderate values of $q$. The following corollary of the above theorems, the proof of which is given in Section~\ref {impact section}, is useful only for large $q$ due to a worse error term. It has the advantage, however, of isolating the fine-scale dependence of $\delta(q;a,b)$ on the residue classes $a$ and $b$ from its primary dependence on the modulus~$q$:

\begin{cor}
Assume GRH and LI. Let $q\ge43$ be an integer. Let $a$ and $b$ be reduced residues\mod q such that $a$ is a nonsquare\mod q and $b$ is a square\mod q, and let $r_1$ and $r_2$ denote the least positive residues of $ab^{-1}$ and $ba^{-1}\mod q$. Then
\begin{equation}
\delta(q;a,b) = \frac12 + \frac{\rho(q)}{2\sqrt{\pi\phi(q)\L(q)}} \bigg( 1 - \frac{\Delta(q;a,b)}{2\L(q)} + O\bigg( \frac1{\log^2q} \bigg) \bigg),
\label{delta first order term}
\end{equation}
where
\begin{equation}
\Delta(q;a,b) = K_q(a-b) + \iota_q(-ab^{-1})\log2 + \frac{\Lambda(r_1)}{r_1} +\frac{\Lambda(r_2)}{r_2} + H(q;a,b)
\label{Delta def}
\end{equation}
(here, the functions $\L$, $K_q$, and $\iota_q$ are defined in Definition~\ref {iota and Lq and Rqn def}, and $H$ is defined in Definition~\ref {hqpr def}). Moreover, $\Delta(q;a,b)$ is nonnegative and bounded above by an absolute constant.
\label{isolate ab contribution cor}
\end{cor}

Armed with this knowledge of the delicate dependence of $\delta(q;a,b)$ on the residue classes $a$ and $b$, we are actually able to ``race races'', that is, investigate inequalities between various values of $\delta(q;a,b)$ as $q$ increases. We remark that Feuerverger and Martin~\cite[Theorem 2(b)]{biases} showed that $\delta(q;a,b) = \delta(q;ab^{-1},1)$ for any square $b\mod q$, and so it often suffices to consider only the densities $\delta(q;a,1)$. Some surprising inequalities come to light when we fix the residue class $a$ and allow the modulus $q$ to vary (among moduli relatively prime to $a$ for which $a$ is a nonsquare). Our next theorem, which is a special case of Corollary~\ref {partial order cor} derived in Section~\ref {predictability section}, demonstrates some of these inequalities:

\begin{theorem}
\label{fix a and b theorem}
Assume GRH and LI.
\begin{itemize}
\item For any integer $a\ne-1$, we have $\delta(q;-1,1) < \delta(q;a,1)$ for all but finitely many integers $q$ with $(q,a)=1$ such that both $-1$ and $a$ are nonsquares$\mod q$.
\item If $a$ is a prime power and $a'\ne-1$ is an integer that is not a prime power, then $\delta(q;a,1) < \delta(q;a',1)$ for all but finitely many integers $q$ with $(q,aa')=1$ such that both $a$ and $a'$ are nonsquares$\mod q$.
\item If $a$ and $a'$ are prime powers with $\Lambda(a)/a > \Lambda(a')/a'$, then $\delta(q;a,1) < \delta(q;a',1)$ for all but finitely many integers $q$ with $(q,aa')=1$ such that both $a$ and $a'$ are nonsquares$\mod q$.
\end{itemize}
\end{theorem}

Finally, these results have computational utility as well. A formula~\cite[equation~(2-57)]{biases} for calculating the value of $\delta(q;a,b)$ is known. However, this formula requires knowledge of a large number of zeros of all Dirichlet $L$-functions associated to characters\mod q even to estimate via numerical integration; therefore it becomes unwieldy to use the formula when $q$ becomes large. On the other hand, the asymptotic series in Theorem~\ref{delta series theorem} can be made completely effective, and the calculation of $V(q;a,b)$ is painless thanks to Theorem~\ref {variance evaluation theorem}. Therefore the densities $\delta(q;a,b)$ can be individually calculated, and collectively bounded, for large~$q$.

\begin{figure}[b]
\caption{All densities $\delta(q;a,b)$ with $q\le1000$}
\includegraphics[width=6in]{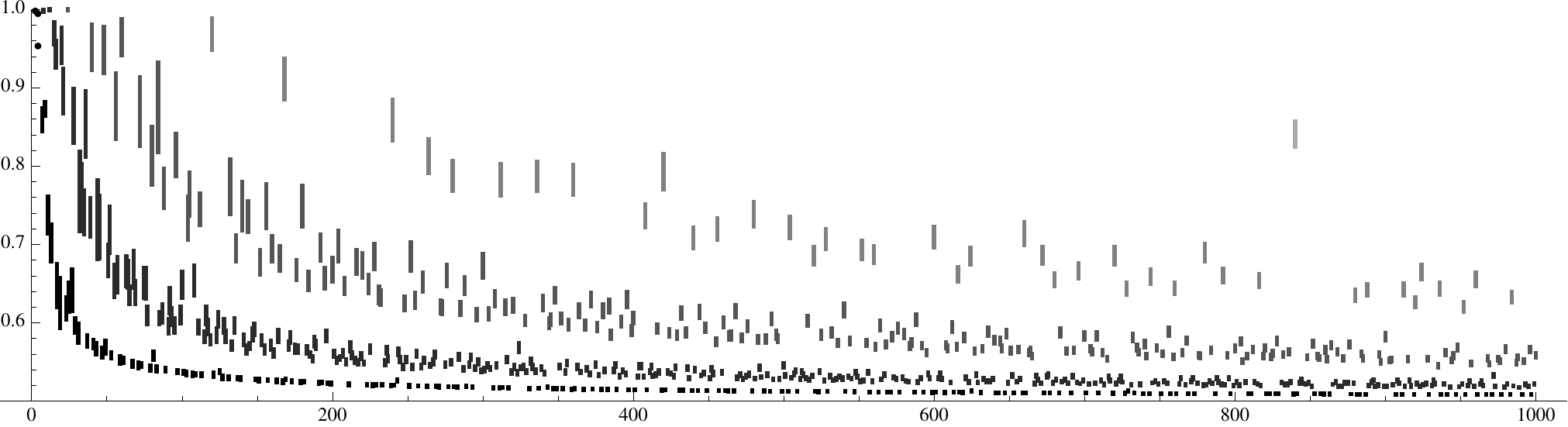}
\label{alldeltas1000}
\end{figure}

For example, the values of $\delta(q;a,b)$ for all moduli up to $1000$ are plotted in Figure~\ref {alldeltas1000}. The modulus $q$ is given on the horizontal axis; the vertical line segment plotted for each $q$ extends between the maximal and minimal values of $\delta(q;a,b)$, as $a$ runs over all nonsquares\mod q and $b$ runs over all squares\mod q. (Of course both $a$ and $b$ should be relatively prime to~$q$. We also omit moduli of the form $q\equiv2\mod4$, since the distribution of primes into residue classes modulo such $q$ is the same as their distribution into residue classes modulo $q/2$.)

The values shown in Figure~\ref {alldeltas1000} organize themselves into several bands; each band corresponds to a constant value of $\rho(q)$, the effect of which on the density $\delta(q;a,b)$ can be clearly seen in the second term on the right-hand side of equation~\eqref{delta first order term}. For example, the lowest (and darkest) band corresponds to moduli $q$ for which $\rho(q)=2$, meaning odd primes and their powers (as well as $q=4$); the second-lowest band corresponds to those moduli for which $\rho(q)=4$, consisting essentially of numbers with two distinct prime factors; and so on, with the first modulus $q=840$ for which $\rho(q)=32$ (the segment closest to the upper right-hand corner of the graph) hinting at the beginning of a fifth such band. Each band decays roughly at a rate of $1/\sqrt{q\log q}$, as is also evident from the aforementioned term of equation~\eqref{delta first order term}.

To give one further example of these computations, which we describe in Section~\ref {explicit computation of the delta}, we are able to find the largest values of $\delta(q;a,b)$ that ever occur. (All decimals listed in this paper are rounded off in the last decimal place.)

\begin{theorem}
\label{top ten thm}
Assume GRH and LI. The ten largest values of $\delta(q;a,b)$ are given in Table~\ref{10 most unfair}.
\end{theorem}

\begin{table}[h]
\label{10 most unfair}
\caption{The top 10 most unfair prime number races}
\begin{tabular}{|c|c|c|c|}
\hline
$q$ & $a$ & $b$ & $\delta(q;a,b)$ \\
\hline
24  &   5   &   1   &   0.999988   \\
24  &   11  &   1  &   0.999983   \\
12  &   11  &   1  &   0.999977   \\
24  &   23  &   1  &   0.999889   \\
24  &   7   &   1   &   0.999834   \\
24  &   19  &   1  &   0.999719   \\
8   &   3   &   1   &   0.999569   \\
12  &   5   &   1   &   0.999206   \\
24  &   17  &   1  &   0.999125   \\
3   &   2   &   1   &   0.999063   \\
\hline
\end{tabular}
\end{table}

Our approach expands upon the seminal work of Rubinstein and Sarnak~\cite{RS}, who introduced a random variable whose distribution encapsulates the information needed to understand $\pi(x;q,a) - \pi(x;q,b)$. We discuss these random variables, formulas and estimates for their characteristic functions (that is, Fourier transforms), and the subsequent derivation of the asymptotic series from Theorem~\ref{delta series theorem} in Section~\ref {asymptotic series section}. In Section~\ref {variance analysis section} we demonstrate how to transform the variance $V(q;a,b)$ from an infinite sum into a finite expression; we can even calculate it extremely precisely using only arithmetic (rather than analytic) information. We also show how the same techniques can be used to establish a central limit theorem for the aforementioned distributions, and we outline how modifications of our arguments can address the two-way race between all nonresidues and all residues\mod q. We investigate the fine-scale effect of the particular residue classes $a$ and $b$ upon the density $\delta(q;a,b)$ in Section~\ref {fine-scale section}; we also show how a similar analysis can explain a ``mirror image'' phenomenon noticed by Bays and Hudson~\cite{cyclic}. Finally, Section~\ref {computations section} is devoted to explicit estimates and a description of our computations of the densities and the resulting conclusions, including Theorem~\ref{top ten thm}.

\subsection*{Acknowledgments} The authors thank Brian Conrey and K.~Soundararajan for suggesting proofs of Lemma~\ref {log bessel lemma}(c) and Proposition~\ref{Sound method prop}, respectively, that were superior to our original proofs. We also thank Andrew Granville for indicating how to improve the error term in Proposition~\ref{only the first term survives prop}, as well as Colin Myerscough for correcting a numerical error in Proposition~\ref {Phi basically decreasing prop} that affected our computations in Sections~\ref {density bounds section}--\ref {explicit computation of the delta}. Robert Rumely and Michael Rubinstein provided lists of zeros of Dirichlet $L$-functions and the appropriate software to compute these zeros, which are needed for the calculations of the densities in Section~\ref{computations section}, and we thank them as well. Finally, we express our gratitude to our advisors past and present, Andrew Granville, Hugh Montgomery, and Trevor Wooley, both for their advice about this paper and for their guidance in general. Le premier auteur est titulaire d'une bourse doctorale du Conseil de recherches en sciences naturelles et en g\'enie du Canada. The second author was supported in part by grants from the Natural Sciences and Engineering Research Council of Canada.

\section{The asymptotic series for the density $\delta(q;a,b)$}
\label{asymptotic series section}

The ultimate goal of this section is to prove Theorem~\ref{delta series theorem}. We begin in Section~\ref {RV section} by describing a random variable whose distribution is the same as the limiting logarithmic distribution of a suitably normalized version of $\pi(x;q,a)-\pi(x;q,b)$, as well as calculating its variance. This approach is the direct descendant of that of Rubinstein and Sarnak~\cite{RS}; one of our main innovations is the exact evaluation of the variance $V(q;a,b)$ in a form that does not involve the zeros of Dirichlet $L$-functions. In Section~\ref {cumulant section} we derive the formula for the characteristic function (Fourier transform) of that random variable; this formula is already known, but our derivation is slightly different and allows us to write the characteristic function in a convenient form (see Proposition~\ref {expressions for Phi prop}). We then use our knowledge of the characteristic function to write the density $\delta(q;a,b)$ as the truncation of an infinite integral in Section~\ref {bounds section}, where the error terms are explicitly bounded using knowledge of the counting function $N(T,\chi)$ of zeros of Dirichlet $L$-functions. Finally, we derive the asymptotic series from Theorem~\ref{delta series theorem} from this truncated integral formula in Section~\ref {derivation section}.

\subsection{Distributions and random variables}
\label{RV section}

We begin by describing random variables related to the counting functions of primes in arithmetic progressions. As is typical when considering primes in arithmetic progressions, we first consider expressions built out of Dirichlet characters.

\begin{definition}
For any Dirichlet character $\chi$ such that GRH holds for $L(s,\chi)$, define
\begin{equation*}
E(x,\chi) = \Lsum{\gamma\in\R} \frac{x^{i\gamma}}{\frac12+i\gamma}.
\end{equation*}
This sum does not converge absolutely, but (thanks to GRH and the functional equation for Dirichlet $L$-functions) it does converge conditionally when interpreted as the limit of $\sum_{|\gamma|<T}$ as $T$ tends to infinity. All untruncated sums over zeros of Dirichlet $L$-functions in this paper should be similarly interpreted.
\qedef
\label{Exchi def}
\end{definition}

\begin{definition}
For any real number $\gamma$, let $Z_\gamma$ denote a random variable that is uniformly distributed on the unit circle, and let $X_\gamma$ denote the random variable that is the real part of $Z_\gamma$. We stipulate that the collection $\{Z_\gamma\}_{\gamma\ge0}$ is independent and that $Z_{-\gamma} = \overline{Z_\gamma}$; this implies that the collection $\{X_\gamma\}_{\gamma\ge0}$ is also independent and that $X_{-\gamma} = X_\gamma$.
\qedef
\label{Z and X def}
\end{definition}

By the limiting logarithmic distribution of a real-valued function $f(t)$, we mean the measure $d\nu$ having the property that the limiting logarithmic density of the set of positive real numbers such that $f(t)$ lies between $\alpha$ and $\beta$ is $\int_\alpha^\beta d\nu$ for any interval $(\alpha,\beta)$.

\begin{prop}
Assume LI. Let $\{c_\chi\colon \chi\mod q\}$ be a collection of complex numbers, indexed by the Dirichlet characters\mod q, satisfying $c_{\bar\chi} = \overline{c_\chi}$. The limiting logarithmic distribution of the function
\begin{equation*}
\allchisum c_\chi E(x,\chi)
\end{equation*}
is the same as the distribution of the random variable
\begin{equation*}
2 \allchisum |c_\chi| \Lsum{\gamma>0} \frac{X_\gamma}{\sqrt{\frac14+\gamma^2}}.
\end{equation*}
\label{to random variable prop}
\end{prop}

\begin{proof}
We have
\begin{align*}
\allchisum c_\chi E(x,\chi) &= \lim_{T\to\infty} \allchisum c_\chi \Lsum{|\gamma|<T} \frac{x^{i\gamma}}{\frac12+i\gamma} \\
&= \lim_{T\to\infty} \allchisum c_\chi \bigg( \Lsum{0<\gamma<T} \frac{x^{i\gamma}}{\frac12+i\gamma} + \Lsum{-T<\gamma<0} \frac{x^{i\gamma}}{\frac12+i\gamma} \bigg).
\end{align*}
(The assumption of LI precludes the possibility that $\gamma=0$.) By the functional equation, the zeros of $L(s,\chi)$ below the real axis correspond to those of $L(s,\bar\chi)$ above the real axis. Therefore
\begin{align}
\allchisum c_\chi E(x,\chi) &= \lim_{T\to\infty} \allchisum c_\chi \bigg( \Lsum{0<\gamma<T} \frac{x^{i\gamma}}{\frac12+i\gamma} + \Lbarsum{0<\gamma<T} \frac{x^{-i\gamma}}{\frac12-i\gamma} \bigg) \label{functional} \\
&= \lim_{T\to\infty} \bigg( \allchisum c_\chi \Lsum{0<\gamma<T} \frac{x^{i\gamma}}{\frac12+i\gamma} + \allchisum \overline{c_{\bar\chi} \Lbarsum{0<\gamma<T} \frac{x^{i\gamma}}{\frac12+i\gamma}} \bigg). \notag
\end{align}
Reindexing this last sum by replacing $\bar\chi$ by $\chi$, we obtain
\begin{align}
\allchisum c_\chi E(x,\chi) &= \lim_{T\to\infty} \bigg( \allchisum c_\chi \Lsum{0<\gamma<T} \frac{x^{i\gamma}}{\frac12+i\gamma} + \allchisum \overline{c_\chi \Lsum{0<\gamma<T} \frac{x^{i\gamma}}{\frac12+i\gamma}} \bigg) \notag \\
&= \lim_{T\to\infty} 2\Re \bigg( \allchisum c_\chi \Lsum{0<\gamma<T} \frac{x^{i\gamma}}{\frac12+i\gamma} \bigg) \label{equation} \\
&= 2\lim_{T\to\infty} \allchisum |c_\chi| \Re \bigg( \Lsum{0<\gamma<T} \frac{e^{i\gamma\log x}\theta_{\chi,\gamma}}{\sqrt{\frac14+\gamma^2}} \bigg), \notag
\end{align}
where $\theta_{\chi,\gamma} = c_\chi|\frac12+i\gamma|/|c_\chi|(\frac12+i\gamma)$ is a complex number of modulus 1. The quantity $e^{i\gamma\log x}\theta_{\chi,\gamma}$ is uniformly distributed (as a function of $\log x$) on the unit circle as $x$ tends to infinity, and hence its limiting logarithmic distribution is the same as the distribution of $Z_\gamma$. Since the various $\gamma$ in each inner sum are linearly independent over the rationals by LI, the tuple $(e^{i\gamma\log x}\theta_{\chi,\gamma})_{0<\gamma<T}$ is uniformly distributed in the $N(T,\chi)$-dimensional torus by Kronecker's theorem. Therefore the limiting logarithmic distribution of the sum
\[
\Lsum{0<\gamma<T} \frac{e^{i\gamma\log x}\theta_{\chi,\gamma}}{\sqrt{\frac14+\gamma^2}}\]
is the same as the distribution of the random variable
\[
\Lsum{0<\gamma<T} \frac{Z_\gamma}{\sqrt{\frac14+\gamma^2}}.
\]
Finally, the work of Rubinstein and Sarnak \cite[Section 3.1]{RS} shows that the limiting logarithmic distribution of
\begin{equation*}
\allchisum c_\chi E(x,\chi) = 2\lim_{T\to\infty} \allchisum |c_\chi| \Re \bigg( \Lsum{0<\gamma<T} \frac{e^{i\gamma\log x}\theta_{\chi,\gamma}}{\sqrt{\frac14+\gamma^2}} \bigg)
\end{equation*}
is the same as the distribution of the random variable
\begin{align*}
\allchisum c_\chi E(x,\chi) &= 2\lim_{T\to\infty} \allchisum |c_\chi| \Lsum{0<\gamma<T} \frac{X_\gamma}{\sqrt{\frac14+\gamma^2}} \\
&= 2 \allchisum |c_\chi| \Lsum{\gamma>0} \frac{X_\gamma}{\sqrt{\frac14+\gamma^2}},
\end{align*}
the convergence of this last limit being ensured by the fact that the $X_\gamma$ are bounded and that each of the sums
\begin{equation*}
\Lsum{\gamma>0} \bigg( \frac1{\sqrt{\frac14+\gamma^2}} \bigg)^2 \le b(\chi)
\end{equation*}
is finite. This establishes the lemma.
\end{proof}

We shall have further occasion to change the indexing of sums, between over all $\gamma$ and over only positive $\gamma$, in the same manner as in equations \eqref{functional} and \eqref{equation}; henceforth we shall justify such changes ``by the functional equation for Dirichlet $L$-functions'' and omit the intermediate steps.

\begin{definition}
For any relative prime integers $q$ and $a$, define
\[
c(q;a) = -1 + \#\{x\mod q\colon x^2\equiv a\mod q\}.
\]
Note that $c(q;a)$ takes only the values $-1$ and $\rho(q)-1$. Now, with $X_\gamma$ as defined in Definition~\ref{Z and X def}, define the random variable
\begin{equation*}
X_{q;a,b} = c(q,b)-c(q,a) + 2 \allchisum |\chi(b)-\chi(a)| \Lsum{\gamma>0} \frac{X_\gamma}{\sqrt{\frac14+\gamma^2}}.
\label{X(q;a,b) definition}
\end{equation*}
Note that the expectation of the random variable $X_{q;a,b}$ is either $\pm\rho(q)$ or~0, depending on the values of $c(q,a)$ and $c(q,b)$.
\qedef
\label{Xqab def}
\end{definition}

\begin{definition}
With $\pi(x;q,a) = \#\{p\le x\colon p\text{ prime, }p\equiv a\mod q\}$ denoting the counting function of primes in the arithmetic progression $a\mod q$, define the normalized error term
\[
E(x;q,a) = \frac{\log x}{\sqrt x} \big( \phi(q)\pi(x;q,a)-\pi(x) \big).
\]
\qedef
\label{Exqa def}
\end{definition}

The next proposition characterizes the limiting logarithmic distribution of the difference of two of these normalized counting functions.

\begin{prop}
Assume GRH and LI. Let $a$ and $b$ be reduced residues modulo $q$. The limiting logarithmic distribution of $E(x;q,a)-E(x;q,b)$ is the same as the distribution of the random variable $X_{q;a,b}$ defined in Definition~\ref{Xqab def}.
\label{same limiting distributions prop}
\end{prop}

\begin{remark}
Since $\delta(q;a,b)$ is defined to be the logarithmic density of those real numbers $x$ for which $\pi(x;q,a) > \pi(x;q,b)$, or equivalently for which $E(x;q,a) > E(x;q,b)$, we see that $\delta(q;a,b)$ equals the probability that $X_{q;a,b}$ is greater than~0. However, we never use this fact directly in the present paper, instead quoting from~\cite{biases} a consequence of that fact in equation~\eqref{deus ex machina} below.
\end{remark}

\begin{proof}
As is customary, define
\begin{equation*}
\psi(x,\chi) = \sum_{n\le x} \chi(n)\Lambda(n)
\end{equation*}
A consequence of the explicit formula for $\psi(x,\chi)$ that arises from the analytic proof of the prime number theorem for arithmetic progressions (\cite[Corollary 12.11]{magicbook} combined with \cite[(12.12)]{magicbook}) is that for $\chi\ne\chi_0$,
\begin{equation*}
\psi(x,\chi) = -\Lsum{\gamma\in\R} \frac{x^{1/2+i\gamma}}{\frac12+i\gamma} + O(\log q \cdot \log x)
\end{equation*}
under the assumption of GRH. We also know \cite[Lemma 2.1]{RS} that
\begin{equation}
E(x;q,a) = -c(q,a) + \mostchisum \bar\chi(a) \frac{\psi(x,\chi)}{\sqrt x} + O_q\bigg( \frac1{\log x} \bigg).
\label{Exqa in terms of Exchi}
\end{equation}
Combining these last two equations with Definition~\ref{Exchi def} for $E(x,\chi)$, we obtain
\begin{equation*}
E(x;q,a) = -c(q,a) - \mostchisum \bar\chi(a) E(x,\chi) + O_q\bigg( \frac1{\log x} \bigg).
\end{equation*}
We therefore see that
\[
E(x;q,a)-E(x;q,b) = c(q,b)-c(q,a) + \allchisum (\bar\chi(b)-\bar\chi(a)) E(x,\chi) + O_q\bigg( \frac1{\log x} \bigg)
\]
(where we have added in the $\chi=\chi_0$ term for convenience). The error term tends to zero as $x$ grows and thus doesn't affect the limiting distribution, and the constant $c(q,b)-c(q,a)$ is independent of~$x$. Therefore, by Proposition~\ref {to random variable prop}, the limiting logarithmic distribution of $E(x;q,a)-E(x;q,b)$ is the same as the distribution of the random variable
\[
c(q,b)-c(q,a) + 2 \allchisum |\bar\chi(b)-\bar\chi(a)| \Lsum{\gamma>0} \frac{X_\gamma}{\sqrt{\frac14+\gamma^2}}.
\]
Since $|\bar\chi(b)-\bar\chi(a)| = |\chi(b)-\chi(a)|$, this last expression is exactly the random variable $X_{q;a,b}$ as claimed.
\end{proof}

To conclude this section, we calculate the variance of the random variable $X_{q;a,b}$.

\begin{prop}
Assume LI. Let $\{c_\chi\colon \chi\mod q\}$ be a collection of complex numbers satisfying $c_{\bar\chi} = \overline{c_\chi}$. For any constant $\mu$, the variance of the random variable
\begin{equation}
\mu + 2 \allchisum c_\chi \Lsum{\gamma>0} \frac{X_\gamma}{\sqrt{\frac14+\gamma^2}}
\label{need.LI}
\end{equation}
equals
$
\allchisum |c_\chi|^2 b(\chi),
$
where $b(\chi)$ was defined in Definition~\ref {bchi and Vqab def}. In particular, the variance of the random variable $X_{q;a,b}$ defined in Definition~\ref {Xqab def} is equal to the quantity $V(q;a,b)$ defined in Definition~\ref {bchi and Vqab def}.
\label{variance prop}
\end{prop}

\begin{proof}
The random variables $\{X_\gamma\colon \gamma>0\}$ form an independent collection by definition; it is important to note that no single variable $X_\gamma$ can correspond to multiple characters $\chi$, due to the assumption of LI. The variance of the sum~\eqref{need.LI} is therefore simply the sum of the individual variances, that is,
\[
\sigma^2\bigg( 2 \allchisum |c_\chi| \Lsum{\gamma>0} \frac{X_\gamma}{\sqrt{\frac14+\gamma^2}} \bigg) = 4 \allchisum |c_\chi|^2 \Lsum{\gamma>0} \frac{\sigma^2(X_\gamma)}{\frac14+\gamma^2}.
\]
The variance of any $X_\gamma$ is $\frac12$, and so this last expression equals
\begin{align*}
&2 \allchisum |c_\chi|^2 \Lsum{\gamma>0} \frac1{\frac14+\gamma^2} \\
&\qquad{}= \allchisum |c_\chi|^2 \Lsum{\gamma>0} \frac1{\frac14+\gamma^2} + \allchisum |c_\chi|^2 \Lbarsum{\gamma<0} \frac1{\frac14+\gamma^2} \\
&\qquad{}= \allchisum |c_\chi|^2 \Lsum{\gamma\in\R} \frac1{\frac14+\gamma^2} = \allchisum |c_\chi|^2 b(\chi)
\end{align*}
by the functional equation for Dirichlet $L$-functions. The fact that $V(q;a,b)$ is the variance of $X_{q;a,b}$ now follows directly from their definitions.
\end{proof}

\subsection{Calculating the characteristic function}
\label{cumulant section}

The characteristic function $\hat X_{q;a,b}(z)$ of the random variable $X_{q;a,b}$ will be extremely important to our analysis of the density $\delta(q;a,b)$. To derive the formula for this characteristic function, we begin by setting down some relevant facts about the standard Bessel function $J_0$ of order zero. Specifically, we collect in the following lemma some useful information about the power series coefficients $\lambda_n$ for
\begin{equation}
\log J_0(z) = \sum_{n=0}^\infty \lambda_n z^n,
\label{lambda def}
\end{equation}
which is valid for $|z| \le \frac{12}5$ since $J_0$ has no zeros in a disk of radius slightly larger than $\frac{12}5$ centered at the origin.

\begin{lemma}
\label{log bessel lemma}
Let the coefficients $\lambda_n$ be defined in equation~\eqref{lambda def}. Then:
\begin{enumerate}
  \item $\lambda_n \ll \big(\frac5{12}\big)^n$ uniformly for $n\ge0$;
  \item $\lambda_0=0$ and $\lambda_{2m-1}=0$ for every $m\ge1$;
  \item$\lambda_{2m}<0$ for every $m\ge1$;
  \item$\lambda_n$ is a rational number for every $n\ge0$.
\end{enumerate}
\end{lemma}

\begin{proof}
The fact that $\log J_0$ is analytic in a disk of radius slightly larger than $\frac{12}5$ centered at the origin immediately implies part~(a). Part~(b) follows from the fact that $J_0$ is an even function with $J_0(0)=1$. Next, $J_0$ has the product expansion~\cite[Section 15.41, equation (3)]{watson} 
\[
J_0(z) = \prod_{k=1}^\infty \bigg( 1 - \frac{z^2}{z_k^2} \bigg),
\]
where the $z_k$ are the positive zeros of $J_0$. Taking logarithms of both sides and expanding each summand in a power series (valid for $|z| \le \frac{12}5$ as before) gives
\[
\log J_0(z) = \sum_{k=1}^\infty \log \bigg( 1 - \frac{z^2}{z_k^2} \bigg) = - \sum_{n=1}^\infty \frac{z^{2n}}n \sum_{k=1}^\infty \frac1{z_k^{2n}},
\]
which shows that $\lambda_{2n} = -n^{-1} \sum_{k=1}^\infty z_k^{-2n}$ is negative, establishing part~(c). Finally, the Bessel function $J_0(z) = \sum_{m=0}^\infty \big({-}\frac14)^m z^{2m}/(m!)^2$ itself has a power series with rational coefficients, as does $\log(1+z)$; therefore the composition $\log(1+(J_0(z)-1))$ also has rational coefficients, establishing part~(d).
\end{proof}

\begin{definition}
Let $\lambda_n$ be defined in equation~\eqref{lambda def}. For any distinct reduced residues $a$ and $b\mod q$, define
\begin{equation}
W_n(q;a,b) = \frac{2^{2n}|\lambda_{2n}|}{V(q;a,b)} \allchisum |\chi(a)-\chi(b)|^{2n} \Lsum{\gamma>0} \frac{1} {(1/4+\gamma^2)^n},
\end{equation}
where $V(q;a,b)$ was defined in Definition~\ref {bchi and Vqab def}, so that $W_1(q;a,b) = \frac12$ for example.
\qedef
\label{Wkqab def}
\end{definition}

In fact, $W_n(q;a,b) V(q;a,b)$ is (up to a constant factor depending on $n$) the $2n$th cumulant of $X(q;a,b)$, which explains why it will appear in the lower terms of the asymptotic formula. We have normalized by $V(q;a,b)$ so that the $W_n(q;a,b)$ depend upon $q$, $a$, and $b$ in a bounded way:

\begin{prop}
We have $W_n(q;a,b) \ll \big( \frac{10}3 \big)^{2n}$ uniformly for all integers~$q$ and all reduced residues $a$ and $b\mod q$.
\label{Wk bound prop}
\end{prop}

\begin{proof}
From Definition~\ref{Wkqab def} and Lemma~\ref {log bessel lemma}(a), we see that
\begin{align*}
W_n (q;a,b) &\ll \frac{2^ {2n}} {V(q;a,b)} \big( \tfrac5{12} \big)^{2n} \allchisum |\chi(a)-\chi(b)|^{2n} \Lsum{\gamma>0} \frac{1} {(1/4+\gamma^2)^n} \notag \\
&\ll \frac{(5/6)^{2n}} {V(q;a,b)} \allchisum 2^{2n-2} |\chi(b)-\chi(a)|^2 \Lsum{\gamma>0} \frac{4^{n-1}}{1/4+\gamma^2} \notag \\
&= \big( \tfrac5{6} \big)^{2n} 2^{2n-2} 4^{n-1} \ll \big( \tfrac{10}3 \big)^{2n},
\end{align*}
as claimed.
\end{proof}

The following functions are necessary to write down the formula for the characteristic function $\hat X_{q;a,b}$.

\begin{definition}
For any Dirichlet character $\chi$, define
\[
F(z,\chi) = \prod_{\substack{\gamma>0 \\ L(\frac12+i\gamma,\chi)=0}} J_0\bigg( \frac{2z}{\sqrt{\frac14+\gamma^2}}\bigg).
\]
Then define
\begin{equation*}
\Phi_{q;a,b}(z) = \allchiprod F\big( |\chi(a)-\chi(b)|z, \chi \big)
\end{equation*}
for any reduced residues $a$ and $b\mod q$. Note that $|F(x,\chi)| \le 1$ for all real numbers $x$, since the same is true of~$J_0$.
\qedef
\label{Fzchi and Phiqabz def}
\end{definition}

The quantity $W_n(q;a,b)$ owes its existence to the following convenient expansion:

\begin{prop}
For any reduced residue classes $a$ and $b\mod q$,
\begin{equation*}
\Phi_{q;a,b}(z) = \exp\bigg( {-V(q;a,b)} \sum_{m=1}^{\infty} W_m(q;a,b) z^{2m} \bigg)
\end{equation*}
for $|z| < \frac3{10}$. In particular,
\begin{equation*}
\Phi_{q;a,b}(z) = e^{-V(q;a,b)z^2/2} \big(1 + O(V(q;a,b)z^4) \big)
\end{equation*}
for $|z| \le \min\{V(q;a,b)^{-1/4}, \frac14\}$.
\label{expressions for Phi prop}
\end{prop}

\begin{proof}
Taking logarithms of both sides of the definition of $\Phi_{q;a,b}(z)$ in Definition~\ref{Fzchi and Phiqabz def} yields
\begin{equation*}
\log \Phi_{q;a,b}(z) = \allchisum \Lsum{\gamma>0} \log J_0\bigg( \frac{2|\chi(a)-\chi(b)|z}{\sqrt{\frac14+\gamma^2}} \bigg).
\end{equation*}
Since $|z| < \frac3{10}$, the argument of the logarithm of $J_0$ is at most $2\cdot2\cdot\frac3{10}/\frac12 = \frac{12}5$, and so the power series expansion~\eqref{lambda def} converges absolutely, giving
\begin{equation*}
\log \Phi_{q;a,b}(z) = \allchisum \Lsum{\gamma>0} \sum_{n=0}^\infty \lambda_n \bigg( \frac{2|\chi(a)-\chi(b)|z}{\sqrt{\frac14+\gamma^2}} \bigg)^n.
\end{equation*}
By Lemma~\ref{log bessel lemma}(b) only the terms $n=2m$ with $m\ge1$ survive, and by Lemma~\ref{log bessel lemma}(c) we may replace $\lambda_{2m}$ by $-|\lambda_{2m}|$. We thus obtain
\begin{align*}
\log \Phi_{q;a,b}(z) &= {-} \sum_{m=0}^\infty z^{2m}\cdot |\lambda_{2m}|2^{2m} \allchisum |\chi(a)-\chi(b)|^{2m} \Lsum{\gamma>0} \frac{1}{(\frac14+\gamma^2)^m} \\
&= {-} \sum_{m=1}^{\infty} V(q;a,b)W_m(q;a,b) z^{2m}
\end{align*}
for $|z| < \frac3{10}$, by Definition~\ref{Wkqab def} for $W_m(q;a,b)$. This establishes the first assertion of the proposition.

By Proposition~\ref{Wk bound prop}, we also have
\begin{equation}
\sum_{m=2}^{\infty} W_m(q;a,b) z^{2m} \ll \sum_{m=2}^{\infty} \big( \tfrac{10}3 \big)^{2m} z^{2m} = \frac{(10/3)^4 z^4}{1-100z^2/9} \ll z^4
\label{rest of W terms}
\end{equation}
uniformly for $|z|\le\frac14$, say. Therefore by the first assertion of the proposition,
\begin{multline*}
\Phi_{q;a,b}(z) = \exp\big( {-V(q;a,b)} W_1(q;a,b) z^{2} \big) \exp\bigg( {-V(q;a,b)} \sum_{m=2}^{\infty} W_m(q;a,b) z^{2m} \bigg) \\
{}= e^{-V(q;a,b)z^2/2} \exp\big( O( V(q;a,b) z^4) \big) = e^{-V(q;a,b)z^2/2} \big( 1+O( V(q;a,b) z^4) \big)
\end{multline*}
as long as $V(q;a,b) z^4 \le 1$. This establishes the second assertion of the proposition.
\end{proof}

All the tools are now in place to calculate the characteristic function $\hat X_{q;a,b}(z) = \E\big( e^{izX_{q;a,b}} \big)$.

\begin{prop}
For any reduced residue classes $a$ and $b\mod q$,
\[
\hat X_{q;a,b} (z) = e^{iz (c(q,b)-c(q,a))} \Phi_{q;a,b}(z).
\]
In particular,
\begin{equation*}
\log \hat X_{q;a,b}(z) = i\big( c(q,a)-c(q,b) \big)z - \tfrac12{V(q;a,b)}z^2 + O\big( V(q;a,b)z^4 \big)
\end{equation*}
for $|z|\le\frac14$.
\label{used in central limit theorem prop}
\end{prop}

\begin{remark}
The first assertion of the proposition was shown by Feuerverger and Martin~\cite{biases} by a slightly different method. Unfortunately an $i$ in the exponential factor of \cite[equation (2-21)]{biases} is missing, an omission that is repeated in the statement of \cite[Theorem 4]{biases}.
\end{remark}

\begin{proof}
For a random variable $X$, define the cumulant-generating function
\[
g_X(t)= \log \hat X(t) = \log \E(e^{i t X})
\]
to be the logarithm of the characteristic function of~$X$. It is easy to see that $g_{\alpha X}(t)=g_{X}(\alpha t)$ for any constant $\alpha$. Moreover, if $X$ and $Y$ are independent random variables, then $\E(e^{i t X}e^{i t Y}) = \E(e^{i t X}) \E(e^{i t Y})$ and so $g_{X+Y}(t)=g_X(t)+g_Y(t)$. Note that if the random variable $C$ is constant with value $c$, then $g_C(t) = itc$.

We can also calculate $g_{X_\gamma}(t)$ where $X_\gamma$ was defined in Definition~\ref {Z and X def}. Indeed, if $\Theta$ is a random variable uniformly distributed on the interval $[-\pi,\pi]$, then $Z_{\gamma} = e^{i\Theta}$ and thus $X_{\gamma} = \cos \Theta$, whence
\[
g_{X_{\gamma}}(t) =\log \E\big( e^{it\cos \Theta} \big) = \log\left( \int_{-\pi}^{\pi} e^{it \cos \theta} \,\frac{d\theta}{2\pi} \right) = \log J_0(t),
\]
where $J_0$ is the Bessel function of order zero \cite[9.1.21]{handbook}.

From Definition~\ref {Xqab def}, the above observations yield
\begin{equation*}
g_{X_{q;a,b}}(t) = it (c(q,b)-c(q,a)) + \allchisum \Lsum{\gamma>0} g_{X_{\gamma}} \bigg( \frac{2 |\chi(a)-\chi(b)|}{\sqrt{1/4+\gamma^2}}t \bigg);
\end{equation*}
in other words,
\begin{align}
\log \hat X_{q;a,b} (t)  &= it (c(q,b)-c(q,a)) + \allchisum \Lsum{\gamma>0} \log J_0 \bigg( \frac{2 |\chi(a)-\chi(b)|}{\sqrt{1/4+\gamma^2}}t \bigg) \notag \\
&= it (c(q,b)-c(q,a)) + \log \Phi_{q;a,b}(x)
\label{back to the log step}
\end{align}
according to Definition~\ref {Fzchi and Phiqabz def}. Exponentiating both sides establishes the first assertion of the proposition. To establish the second assertion, we combine equation~\eqref {back to the log step} with Proposition~\ref {expressions for Phi prop} to see that for $|z|\le\frac14$,
\begin{align*}
\log \hat X_{q;a,b} (t) &= it (c(q,b)-c(q,a)) - V(q;a,b) \sum_{m=1}^{\infty} W_m(q;a,b) z^{2m} \\
&= it (c(q,b)-c(q,a)) - \tfrac12V(q;a,b)z^2 + O(V(q;a,b)z^4)
\end{align*}
by the estimate~\eqref{rest of W terms} and the fact that $W_1(q;a,b) = \frac12$.
\end{proof}

\subsection{Bounds for the characteristic function}
\label{bounds section}

A formula (namely equation~\eqref {deus ex machina} below) is known that relates $\delta(q;a,b)$ to an integral involving $\Phi_{q;a,b}$. Using this formula to obtain explicit estimates for $\delta(q;a,b)$ requires explicit estimates upon $\Phi_{q;a,b}$; our first estimate shows that this function takes its largest values near~$0$.

\begin{prop}
Let $0\le\kappa\le\frac5{24}$. For any reduced residue classes $a$ and $b\mod q$, we have $|\Phi_{q;a,b}(t)| \le |\Phi_{q;a,b}(\kappa)|$ for all $t\ge \kappa$.
\label{Phi basically decreasing prop}
\end{prop}

\begin{proof}
From Definition~\ref{Fzchi and Phiqabz def}, it suffices to show that for any real number $\gamma>0$,
\begin{equation}
\bigg| J_0\bigg( \frac{2|\chi(a)-\chi(b)|t}{\sqrt{1/4+\gamma^2}} \bigg) \bigg| \le \bigg| J_0\bigg( \frac{2|\chi(a)-\chi(b)|\kappa}{\sqrt{1/4+\gamma^2}} \bigg) \bigg|
\label{bessel.need}
\end{equation}
for all $t\ge \kappa$. We use the facts that $J_0$ is a positive, decreasing function on the interval $[0,\frac{5}{3}]$ and that $J_0\big(\frac{5}{3}\big) \ge |J_0(x)|$ for all $x\ge\frac{5}{3}$. Since
\[
0 \le \frac{2|\chi(a)-\chi(b)|\kappa}{\sqrt{1/4+\gamma^2}} \le \frac{2\cdot2\cdot5/24}{\sqrt{1/4}} = \frac{5}{3},
\]
we see that $J_0$ is positive and decreasing on the interval
\[
\bigg[ \frac{2|\chi(a)-\chi(b)|\kappa}{\sqrt{1/4+\gamma^2}}, \frac{5}{3} \bigg].
\]
Together with $J_0\big(\frac{5}{3}\big) \ge |J_0(x)|$ for all $x\ge\frac{5}{3}$, this establishes equation~\eqref{bessel.need} and hence the lemma.
\end{proof}

Let $N(T,\chi)$ denote, as usual, the number of nontrivial zeros of $L(s,\chi)$ having imaginary part at most $T$ in absolute value. Since the function $\Phi_{q;a,b}$ is a product indexed by these nontrivial zeros, we need to establish the following explicit estimates for $N(T,\chi)$. Although exact values for the constants in the results of this section are not needed for proving Theorem~\ref{delta series theorem}, they will become necessary in Section~\ref {computations section} when we explicitly calculate values and bounds for $\delta(q;a,b)$.

\begin{prop}
Let the nonprincipal character $\chi \mod q$ be induced by $\chi^* \mod {q^*}$. For any real number $T\ge 1$,
$$
N(T,\chi) \le \frac T\pi \log \frac{q^*T}{2\pi e} + 0.68884 \log \frac{q^*T}{2\pi e} + 10.6035.
$$
For $T\ge 100$,
$$
N(T,\chi) \ge \frac{44T}{45\pi} \log \frac{q^*T}{2\pi e} - 10.551.
$$
\label{McCurley bound}
\end{prop}

\begin{proof}
We cite the following result of McCurley~\cite[Theorem 2.1]{mcc}: for $T \ge 1$ and $\eta \in (0,0.5]$,
$$
\bigg| N(T,\chi) - \frac{T}{\pi} \log \frac{q^*T}{2\pi e} \bigg| < C_1 \log q^*T +C_2,
$$
with $C_1=\frac{1+2\eta}{\pi \log 2}$ and $C_2 =.3058-.268 \eta +4\frac{\log\zeta(1+\eta)}{\log 2}-2\frac{\log\zeta(2+2\eta)}{\log 2}+\frac{2}{\pi}\frac{\log\zeta(\frac{3}{2}+2\eta)}{\log 2}$. (McCurley states his result for primitive nonprincipal characters, but since $L(s,\chi)$ and $L(s,\chi^*)$ have the same zeros inside the critical strip, the result holds for any nonprincipal character.) Taking $\eta = 0.25$, we obtain
\begin{equation*}
\bigg| N(T,\chi) - \frac{T}{\pi} \log \frac{q^*T}{2\pi e} \bigg| < 0.68884 \log q^*T + 8.64865 < 0.68884 \log \frac{q^*T}{2\pi e} + 10.6035.
\end{equation*}
This inequality establishes the first assertion of the proposition. The inequality also implies that
\begin{equation*}
N(T,\chi) > \frac{44T}{45\pi} \log \frac{q^*T}{2\pi e} + \bigg( \bigg( \frac{T}{45\pi} - .68884 \bigg) \log \frac{q^*T}{2\pi e} - 10.6035 \bigg);
\end{equation*}
the second assertion of the proposition follows upon calculating that the expression in parentheses is at least $-10.551$ when $T\ge100$ (we know that $q^*\ge3$ as there are no nonprincipal primitive characters modulo~1 or~2).
\end{proof}

The next two results establish an exponentially decreasing upper bound for $\Phi_{q;a,b}(t)$ when $t$ is large.

\begin{lemma}
For any nonprincipal character $\chi\mod q$, we have $|F(x,\chi)F(x,\bar\chi)| \le e^{-0.2725x}$ for $x\ge200$.
\label{F.upper.bound.lemma}
\end{lemma}

\begin{proof}
First note that
\[
F(x,\bar\chi) = \prod_{\substack{\gamma>0 \\ L(1/2+i\gamma,\bar\chi)=0}} J_0\bigg( \frac{2x}{\sqrt{1/4+\gamma^2}} \bigg) = \Lprod{\gamma<0} J_0\bigg( \frac{2x}{\sqrt{1/4+(-\gamma)^2}} \bigg)
\]
by the identity $L(s,\bar\chi) = \overline{L(\bar s,\chi)}$, and therefore
\[
F(x,\chi)F(x,\bar\chi) = \Lprod{\gamma\in\R} J_0\bigg( \frac{2x}{\sqrt{1/4+\gamma^2}} \bigg).
\]
Using the bound \cite[equation (4.5)]{RS}
\[
|J_0(z)| \le \min\bigg\{ 1, \sqrt{\tfrac2{\pi|x|}} \bigg\},
\]
we see that for $x\ge1$,
\[
|F(x,\chi)F(x,\bar\chi)| \le \Lprod{-x/2<\gamma<x/2} \bigg| J_0\bigg( \frac{2x}{\sqrt{1/4+\gamma^2}} \bigg) \bigg| \le \Lprod{|\gamma|<x/2} \frac{(1/4+\gamma^2)^{1/4}}{\sqrt{\pi x}}.
\]
When $x\ge1$ and $|\gamma|<x/2$, the factor $(1/4+\gamma^2)^{1/4}(\pi x)^{-1/2}$ never exceeds $1/2$. Therefore
\[
|F(x,\chi)F(x,\bar\chi)| \le 2^{-N(x/2,\chi)} = \exp\big( {-}(\log 2) N(x/2,\chi) \big).
\]
By Proposition~\ref {McCurley bound}, we thus have for $x\ge200$
\begin{align*}
|F(x,\chi)F(x,\bar\chi)| &\le 2^{10.558} \exp\bigg( {-}\frac{22\log 2}{45\pi} x \log \frac{q^*x}{4\pi e} \bigg) \\
&\le \exp\bigg( {-}0.107866 x \log \frac{3 x}{4\pi e}+7.3183 \bigg) \le e^{-0.2725 x},
\end{align*}
as claimed.
\end{proof}

\begin{prop}
For any distinct reduced residue classes $a$ and $b\mod q$ such that $(ab,q)=1$, we have
$|\Phi_{q;a,b}(t)| \le e^{-0.0454\phi(q)t}$ for $t\ge200$.
\label{Phi upper bound prop}
\end{prop}

\begin{proof}
We begin by noting that the orthogonality relations for Dirichlet characters imply that $\sum_{\chi\mod q} |\chi(a)-\chi(b)|^2 = 2\phi(q)$ (as we show in Proposition~\ref {just orthogonality prop} below). On the other hand, if $S$ is the set of characters $\chi\mod q$ such that $|\chi(a)-\chi(b)|\ge1$, then
\[
\sum_{\chi\mod q} |\chi(a)-\chi(b)|^2 \le \sum_{\substack{\chi\mod q \\ \chi\notin S}} 1 + \sum_{\chi\in S} 4 = \phi(q)-\#S + 4\#S.
\]
Combining these two inequalities shows that $2\phi(q) \le \phi(q) + 3\#S$, or equivalently $\#S \ge \frac13\phi(q)$. Note that clearly $\chi_0 \notin S$.

From Definition~\ref {Fzchi and Phiqabz def}, we have
\[
|\Phi_{q;a,b}(t)|^2 = \prod_{\chi\mod q} |F(|\chi(a)-\chi(b)|t,\chi)|^2 = \prod_{\chi\mod q} \big| F(|\chi(a)-\chi(b)|t,\chi) F(|\chi(a)-\chi(b)|t,\bar\chi) \big|,
\]
since every character appears once as $\chi$ and once as $\bar\chi$ in the product on the right-hand side. Since $|F(x,\chi)| \le 1$ for all real numbers $x$, we can restrict the product on the right-hand side to those characters $\chi\in S$ and still have a valid upper bound. For any $\chi\in S$, Lemma~\ref{F.upper.bound.lemma} gives us $|F(|\chi(a)-\chi(b)|t,\chi)F(|\chi(a)-\chi(b)|t,\bar\chi)| \le e^{{-0.2725}|\chi(a)-\chi(b)|t} \le e^{{-0.2725}t}$ for $t\ge200$, whence

\[
|\Phi_{q;a,b}(t)|^2 \le \prod_ {\chi\in S} \big| F(|\chi(a)-\chi(b)|t,\chi) F(|\chi(a)-\chi(b)|t,\bar\chi) \big| \le ( e^{{-0.2725} t} )^{\#S} \le (e^{-0.0454\phi(q)t})^2,
\]
which is equivalent to the assertion of the proposition.
\end{proof}

At this point we can establish the required formula for $\delta(q;a,b)$, in terms of a truncated integral involving $\Phi_{q;a,b}$, with an explicit error term. To more easily record the explicit bounds for error terms, we employ a variant of the $O$-notation: we write $A = \Obar(B)$ if $|A| \le B$ (as opposed to a constant times $B$) for all values of the parameters under consideration.

\begin{prop}
\label{delta comes down to small interval prop}
Assume GRH and LI. Let $a$ and $b$ be reduced residues\mod q such that $a$ is a nonsquare\mod q and $b$ is a square\mod q. If $V(q;a,b) \ge 531$, then
\begin{multline*}
\delta(q;a,b) = \frac12 + \frac1{2\pi} \int_{-V(q;a,b)^{-1/4}}^{V(q;a,b)^{-1/4}} \frac{\sin \rho(q) x}x \Phi_{q;a,b}(x) \,dx \\
{}+ \Obar\bigg( 0.03506 \frac{e^{-9.08\phi(q)}}{\phi(q)} + 63.67 \rho(q) e^ {-V(q;a,b)^{1/2}/2}\bigg),
\end{multline*}
\end{prop}

\begin{proof}
Our starting point is the formula of Feuerverger and Martin \cite[equation (2.57)]{biases}, which is valid under the assumptions of GRH and LI:
\begin{equation}
\delta(q;a,b) = \frac12 - \frac1{2\pi} \int_{-\infty}^\infty \frac{\sin((c(q,a)-c(q,b))x)}x \Phi_{q;a,b}(x) \,dx.
\label{deus ex machina}
\end{equation}
In the case where $a$ is a nonsquare modulo $q$ and $b$ is a square modulo $q$, the constant $c(q,a)-c(q,b)$ equals $-\rho(q)$, so that
\begin{equation*}
\delta(q;a,b) = \frac12 + \frac1{2\pi} \int_{-\infty}^\infty \frac{\sin\rho(q)x}x \Phi_{q;a,b}(x) \,dx.
\end{equation*}
The part of the integral where $x\ge200$ can be bounded using Proposition~\ref {Phi upper bound prop}:
\[
\bigg| \frac1{2\pi} \int_{200}^\infty \frac{\sin\rho(q)x}x \Phi_{q;a,b}(x) \,dx \bigg| \le \frac1{400\pi} \int_{200}^\infty e^{-0.0454\phi(q)x} \,dx < \frac{0.01753e^{-9.08\phi(q)}}{\phi(q)}.
\]
The part where $x\le-200$ is bounded by the same amount, and so
\begin{equation}
\delta(q;a,b) = \frac12 + \frac1{2\pi} \int_{-200}^{200} \frac{\sin\rho(q)x}x \Phi_{q;a,b}(x) \,dx + \Obar\bigg( 0.03506 \frac{e^{-9.08\phi(q)}}{\phi(q)} \bigg).
\label{tails gone}
\end{equation}

We now consider the part of the integral where $V(q;a,b)^{-1/4} \le x \le 200$. The hypothesis that $V(q;a,b) \ge 531$ implies that $V(q;a,b)^{-1/4} < \frac5{24}$, which allows us to make two simplifications. First, by Proposition~\ref {Phi basically decreasing prop}, we know that $|\Phi_{q;a,b}(x)| \le \Phi_{q;a,b}(V(q;a,b)^{-1/4})$ for all $x$ in the range under consideration. Second, by Proposition~\ref {expressions for Phi prop} we have
\begin{equation*}
\Phi_{q;a,b}(x) = \exp\bigg( {-V(q;a,b)} \sum_{m=1}^{\infty} W_m(q;a,b) x^{2m} \bigg) \le e^{-V(q;a,b)x^2/2}
\end{equation*}
for all real numbers $|x| < \frac3{10}$, since $W_1(q;a,b) = \frac12$ and all the $W_m(q;a,b)$ are nonnegative by Definition~\ref{Wkqab def}. Since $\frac5{24}<\frac3{10}$, we see that $|\Phi_{q;a,b}(x)| \le e^{-V(q;a,b)^{1/2}/2}$ for all $x$ in the range under consideration. Noting also that $\big| \sin(\rho(q)x)/x \big| \le \rho(q)$ for all real numbers $x$, we conclude that
\[
\bigg| \int_{V(q;a,b)^{-1/4}}^{200} \frac{\sin\rho(q) x}{x} \Phi(x) dx \bigg| \le \rho(q) \int_{V(q;a,b)^{-1/4}}^{200} e^{-V(q;a,b)^{1/2}/2} dx\le 200 \rho(q) e^ {-V(q;a,b)^{1/2}/2}.
\]
The part of the integral where $-200 \le x \le -V(q;a,b)^{-1/4}$ is bounded by the same amount, and thus equation~\eqref{tails gone} becomes
\begin{multline*}
\delta(q;a,b) = \frac12 + \frac1{2\pi} \int_{-V(q;a,b)^{-1/4}}^{V(q;a,b)^{-1/4}} \frac{\sin\rho(q)x}x \Phi_{q;a,b}(x) \,dx \\
{}+ \Obar\bigg( 0.03506 \frac{e^{-9.08\phi(q)}}{\phi(q)} + \frac{200}\pi \rho(q) e^ {-V(q;a,b)^{1/2}/2}\bigg),
\end{multline*}
which establishes the proposition.
\end{proof}

\subsection{Derivation of the asymptotic series}
\label{derivation section}

In this section we give the proof of Theorem~\ref{delta series theorem}. Our first step is to transform the conclusion of Proposition~\ref {delta comes down to small interval prop}, which was phrased with a mind towards the explicit calculations in Section~\ref{computations section}, into a form more convenient for our present purposes:

\begin{lemma}
Assume GRH and LI. For any reduced residues $a$ and $b\mod q$ such that $a$ is a nonsquare\mod q and $b$ is a square\mod q, and for any fixed $J>0$,
\begin{multline*}
\delta(q;a,b) = \frac12 + \frac{\rho(q)}{2\pi\sqrt{V(q;a,b)}} \int_{-V(q;a,b)^{1/4}}^{V(q;a,b)^{1/4}} \frac{\sin\big( \rho(q) y/\sqrt{V(q;a,b)} \big)}{\rho(q) y/\sqrt{V(q;a,b)}} \Phi_{q;a,b}\bigg( \frac y{\sqrt{V(q;a,b)}} \bigg) \,dy \\
+ O_J\big( V(q;a,b)^{-J} \big).
\end{multline*}
\label{small interval, inexplicit constant lemma}
\end{lemma}

\begin{proof}
We make the change of variables $x = y/\sqrt{V(q;a,b)}$ in Proposition~\ref{delta comes down to small interval prop}, obtaining
\begin{multline*}
\delta(q;a,b) = \frac12 + \frac1{2\pi} \int_{-V(q;a,b)^{1/4}}^{V(q;a,b)^{1/4}} \frac{\sin\big( \rho(q) y/\sqrt{V(q;a,b)} \big)}{y/\sqrt{V(q;a,b)}} \Phi_{q;a,b}\bigg( \frac y{\sqrt{V(q;a,b)}} \bigg) \frac{dy}{\sqrt{V(q;a,b)}} \\
{}+ \Obar\bigg( 0.06217 \frac{e^{-5.12\phi(q)}}{\phi(q)} + 63.67 \rho(q) e^ {-V(q;a,b)^{1/2}/2}\bigg),
\end{multline*}
the main terms of which are exactly what we want. The lemma then follows from the estimates
\[
e^{-5.12\phi(q)} \ll_J V(q;a,b)^{-J} \quad\text{and}\quad \rho(q)e^ {-V(q;a,b)^{1/2}/2} \ll_J V(q;a,b)^{-J}
\]
for any fixed constant $J$: these estimates hold because $V(q;a,b) \sim 2\phi(q)\log q$ by Proposition~\ref{Vqab asymptotic prop}, while the standard lower bound $\phi(q) \gg q/\log\log q$ follows from equation~\eqref {phi lower bound}.
\end{proof}

We will soon be expanding most of the integrand in Lemma~\ref {small interval, inexplicit constant lemma} into a power series; the following definition and lemma treat the integrals that so arise.

\begin{definition}
For any nonnegative integer $k$, define $(2k-1)!! = (2k-1)(2k-3)\cdots 3\cdot1$, where we make the convention that $(-1)!!=1$. Also, for any nonnegative integer $k$ and any positive real number $B$, define
\[
M_k(B) = \int_{-B}^B y^{2k}e^{-y^2/2}\,dy.
\]
\qedef
\label{double factorial and normal moment def}
\end{definition}

\begin{lemma}
Let $J$ and $B$ be positive real numbers. For any nonnegative integer $k$, we have $M_k(B) = (2k-1)!!\sqrt{2\pi} + O_{k,J}\big( B^{-J} \big)$.
\label{normal moments lemma}
\end{lemma}

\begin{proof}
We proceed by induction on $k$. In the case $k=0$, we have
\begin{align*}
M_0(B) = \int_{-B}^B e^{-y^2/2}\,dy &= \int_{-\infty}^\infty e^{-y^2/2}\,dy - 2 \int_B^\infty e^{-y^2/2}\,dy \\
&= \sqrt{2\pi} + O\bigg( \int_B^\infty e^{-By/2}\,dy \bigg) \\
&= \sqrt{2\pi} + O\bigg( \frac2B e^{-B^2/2} \bigg) = \sqrt{2\pi} + O_J( B^{-J} )
\end{align*}
as required. On the other hand, for $k\ge1$ we can use integration by parts to obtain
\begin{align*}
M_k(B) = \int_{-B}^B y^{2k-1}\cdot ye^{-y^2/2}\,dy &= -y^{2k-1}e^{-y^2/2} \bigg|_{-B}^B + (2k-1) \int_{-B}^B y^{2k-2}e^{-y^2/2}\, dy \\
&= O\big( B^{2k-1}e^{-B^2/2} \big) + (2k-1)M_{k-1}(B).
\end{align*}
Since the error term $B^{2k-1}e^{-B^2/2}$ is indeed $O_{k,J}\big( B^{-J} \big)$, the lemma follows from the inductive hypothesis for $M_{k-1}(B)$.
\end{proof}

The following familiar power series expansions can be truncated with reasonable error terms:

\begin{lemma}
\label{taylor series}
Let $K$ be a nonnegative integer and $C>1$ a real number. Uniformly for $|z|\le C$, we have the series expansions
\begin{align*}
e^z &= \sum_{j=0}^{K} \frac{z^{j}}{j!} + O_{C,K}(|z|^{K+1}); \\
\frac{\sin z}{z} &= \sum_{j=0}^{K} (-1)^j \frac{z^{2j}}{(2j+1)!} + O_{C,K}(|z|^{2(K+1)}).
\end{align*}
\end{lemma}

\begin{proof}
The Taylor series for $e^z$, valid for all complex numbers $z$, can be written as
$$
e^z=\sum_{j=0}^K \frac{z^j}{j!} + z^{K+1}\sum_{j=0}^{\infty} \frac{z^j}{(j+K+1)!}.
$$
The function $\sum_{j=0}^{\infty} {z^j/(j+K+1)!}$ converges for all complex numbers $z$ and hence represents an entire function; in particular, it is continuous and hence bounded in the disc $|z|\le C$. This establishes the first assertion of the lemma, and the second assertion is proved in a similar fashion.
\end{proof}

Everything we need to prove Theorem~\ref{delta series theorem} is now in place, once we give the definition of the constants $s_{q;a,b}(\ell,j)$ that appear in its statement:

\begin{definition}
For any reduced residues $a$ and $b\mod q$, and any positive integers $j\le\ell$, define
\begin{equation*}
s_{q;a,b}(\ell,j) = \frac{(-1)^j}{(2j+1)!} \mathop{\sum \cdots \sum}_{i_2+2i_3+\dots+ \ell i_{\ell +1}=\ell-j} \big( 2(\ell+i_2+\dots+i_{\ell+1})-1 \big)!! \prod_{k=2}^{\ell+1} \frac{(-W_k(q;a,b))^{i_k}}{i_k!},
\end{equation*}
where the indices $i_2,\dots,i_{\ell+1}$ take all nonnegative integer values that satisfy the constraint $i_2+2i_3+\dots+ \ell i_{\ell +1}=\ell-j$. Note that $s_{q;a,b}(0,0)=1$ always. Since $W_k(q;a,b) \ll \big(\frac{10}3 \big)^k$ by Proposition~\ref{Wk bound prop}, we see that $s_{q;a,b}(\ell,j)$ is bounded in absolute value by some (combinatorially complicated) function of $\ell$ uniformly in $q$, $a$, and $b$ (and uniformly in $j$ as well, since there are only finitely many possibilities $\{0, 1, \dots, \ell\}$ for $j$).
\qedef
\label{s-coeffs def}
\end{definition}

\begin{proof}[Proof of Theorem~\ref{delta series theorem}]
To lighten the notation in this proof, we temporarily write $\rho$ for $\rho(q)$, $\delta$ for $\delta(q;a,b)$, $V$ for $V(q;a,b)$, and $W_k$ for $W_k(q;a,b)$. We also allow all $O$-constants to depend on~$K$. Since $\delta$ is bounded, the theorem is trivially true when $V$ is bounded, since the error term is at least as large as any other term in that case; therefore we may assume that $V$ is sufficiently large. For later usage in this proof, we note that $\rho \ll V^{1/4}$, which follows amply from the bound $\rho \ll_\ep q^\ep$ mentioned in Definition~\ref{rho def} and the asymptotic formula $V \sim 2\phi(q)\log q$ proved in Proposition~\ref{Vqab asymptotic prop}.

We begin by noting that from Proposition~\ref {expressions for Phi prop},
\begin{align}
  \Phi_{q;a,b}(x) = \exp \bigg( {-}V\sum_{k=1}^\infty W_k x^{2k} \bigg) &= \exp \bigg( {-}V\sum_{k=1}^{K+1} W_k x^{2k} + O( V x^{2(K+2)} ) \bigg) \notag \\
  &= \exp \bigg( {-}V\sum_{k=1}^{K+1} W_k x^{2k} \bigg) \big( 1 + O( V x^{2(K+2)} ) \big)
\label{only finitely many W}
\end{align}
uniformly for all $|x|\le\min(\frac14,V^{-1/4})$, where the second equality follows from the upper bound given in Proposition~\ref{Wk bound prop}.
Inserting this formula into the expression for $\delta(q;a,b)$ from Lemma~\ref {small interval, inexplicit constant lemma}, applied with $J=K+2$, gives
\begin{multline*}
  \delta = \frac12 + \frac\rho{2\pi\sqrt V} \int_{-V^{1/4}}^{V^{1/4}} \frac{\sin (\rho y/\sqrt V)}{\rho y/\sqrt V} \exp \bigg( {-}\sum_{k=1}^ {K+1} \frac{W_k y^{2k}}{V^{k-1}} \bigg) \bigg( 1 + O \bigg( \frac{y^{2(K+2)}}{V^{K+1}} \bigg) \bigg) \,dy \\
  {}+ O \big( V^{-K-2} \big).
\end{multline*}
This use of equation~\eqref{only finitely many W} is justified because the argument $y/\sqrt V$ of $\Phi_{q;a,b}$ in the integral in Lemma~\ref {small interval, inexplicit constant lemma} is at most $V^{1/4}/\sqrt V \le \frac14$, by the assumption that $V$ is sufficiently large. To simplify the error term in the integral, we ignore all of the factors in the integrand (which are bounded by~1 in absolute value) except for the $k=1$ term, in which $W_1=\frac12$, to derive the upper bound
\[
\int_{-V^{1/4}}^{V^{1/4}} \frac{\sin (\rho y/\sqrt V)}{\rho y/\sqrt V} \exp \bigg( {-}\sum_{k=1}^ {K+1} \frac{W_k y^{2k}}{V^{k-1}} \bigg) \frac{y^{2(K+2)}}{V^{K+1}} \,dy \ll \frac1{V^{K+1}} \int_{-\infty}^\infty e^{-y^2/2} y^{2K+4}\,dy \ll_K \frac1{V^{K+1}}.
\]
Therefore
\begin{equation}
  \delta = \frac12 + \frac\rho{2\pi\sqrt V} \int_{-V^{1/4}}^{V^{1/4}} \frac{\sin (\rho y/\sqrt V)}{\rho y/\sqrt V} \exp \bigg( {-}\sum_{k=1}^ {K+1} \frac{W_k y^{2k}}{V^{k-1}} \bigg) \,dy + O \bigg( \frac\rho{V^{K+3/2}} \bigg).
\label{after COV}
\end{equation}

The integrand in equation~\eqref{after COV} is the product of $K+2$ functions, namely $K+1$ exponential factors and a factor involving the function $(\sin z)/z$. Our plan is to keep the first exponential function as it is and expand the other factors into their power series at the origin. Note that the argument of the $k$th exponential factor is at most $W_k V^{1-k/2}$ in absolute value, which is bounded (by a constant depending on $K$) for all $k\ge2$ by Proposition~\ref{Wk bound prop}. Similarly, the argument of the function $(\sin z)/z$ is bounded by $\rho V^{1/4}/\sqrt V \ll 1$. Therefore the expansion of all of these factors, excepting the exponential factor corresponding to $k=1$, into their power series is legitimate in the range of integration.

Specifically, we have the two identities
\begin{align*}
\sum_{j=0}^K \frac{(-1)^j}{(2j+1)!} \frac{(\rho y)^{2j}}{V^j} &= \frac{\sin (\rho y/\sqrt V)}{\rho y/\sqrt V} + O\bigg( \frac{(\rho y)^{2(K+1)}}{V^{K+1}} \bigg); \\
\sum_{i_k=0}^K \frac{(-1)^{i_k}}{i_k!} \bigg( \frac{W_k y^{2k}}{V^{k-1}} \bigg)^{i_k} &= \exp \bigg( {-} \frac{W_k y^{2k}}{V^{k-1}} \bigg) + O\bigg( \bigg( \frac{W_k y^{2k}}{V^{k-1}} \bigg)^{\!K+1} \bigg) \\
&= \exp \bigg( {-} \frac{W_k y^{2k}}{V^{k-1}} \bigg) + O\bigg( \frac{y^{2k(K+1)}}{V^{K+1}} \bigg),
\end{align*}
where the error terms are justified by Lemma~\ref{taylor series}; in
the last equality we have used Proposition~\ref{Wk bound prop} to
ignore the contribution of the factor $W_k$ to the error term (since
the $O$-constant may depend on $K$). From these identities, we
deduce that
\begin{align*}
\bigg( \sum_{j=0}^K \frac{(-1)^j}{(2j+1)!} \frac{(\rho y)^{2j}}{V^j} \bigg) & e^{-y^2/2} \prod_{k=2}^{K+1} \bigg( \sum_{i_k=0}^K \frac{(-1)^{i_k}}{i_k!} \bigg( \frac{W_k y^{2k}}{V^{k-1}} \bigg)^{i_k} \bigg) \\
&= \bigg( \frac{\sin (\rho y/\sqrt V)}{\rho y/\sqrt V} + O\bigg( \frac{(\rho y)^{2(K+1)}}{V^{K+1}} \bigg) \bigg) \\
&\qquad{}\times e^{-y^2/2} \prod_{k=2}^{K+1} \bigg( \exp \bigg( {-} \frac{W_k y^{2k}}{V^{k-1}} \bigg) + O\bigg( \bigg( \frac{W_k y^{2k}}{V^{k-1}} \bigg)^{\!K+1} \bigg) \bigg) \\
&= \frac{\sin (\rho y/\sqrt V)}{\rho y/\sqrt V} \prod_{k=1}^{K+1} \exp \bigg( {-} \frac{W_k y^{2k}}{V^{k-1}} \bigg) + O\bigg( y^{(K+2)(K+1)^2} e^{-y^2/2} \frac{\rho^{2K+2}}{V^{K+1}} \bigg).
\end{align*}
(The computation of the error term is simplified by the fact that all the main terms on the right-hand side are at most~1 in absolute value, so that we need only figure out the largest powers of $y$ and $\rho$, and the smallest power of $V$, that can be obtained by the cross terms.)

Substituting this identity into equation~\eqref{after COV} yields
\begin{align*}
\delta &= \frac12 + \frac\rho{2\pi\sqrt V} \int_{-V^{1/4}}^{V^{1/4}} \bigg( \sum_{j=0}^K \frac{(-1)^j}{(2j+1)!} \frac{(\rho y)^{2j}}{V^j} \bigg) e^{-y^2/2} \prod_{k=2}^{K+1} \bigg( \sum_{i_k=0}^K \frac{(-1)^{i_k}}{i_k!} \bigg( \frac{W_k y^{2k}}{V^{k-1}} \bigg)^{i_k} \bigg) \,dy \\
&\hskip1.5in{}+ O \bigg( \frac\rho{\sqrt V} \int_{-\infty}^\infty y^{(K+2)(K+1)^2} e^{-y^2/2} \frac{\rho^{2K+2}}{V^{K+1}} \,dy + \frac\rho{V^{K+3/2}} \bigg) \\
&= \frac12 + \frac\rho{2\pi\sqrt V} \sum_{j=0}^K \sum_{i_2=0}^K \cdots \sum_{i_{K+1}=0}^K \bigg( \frac{(-1)^j}{(2j+1)!} \frac{\rho^{2j}}{V^j} \\
&\hskip1.5in{}\times \prod_{k=2}^{K+1} \frac1{i_k!} \bigg( \frac{-W_k}{V^{k-1}} \bigg)^{i_k} M_{j+2i_2+\dots+(K+1)i_{K+1}}\big( V^{1/4} \big) \bigg) + O \bigg( \frac{\rho^{2K+3}}{V^{K+3/2}} \bigg),
\end{align*}
where $M$ was defined in Definition~\ref {double factorial and normal moment def}. Invoking Lemma~\ref{normal moments lemma} and then collecting the summands according to the power $\ell = j + i_1 + 2i_2 + \dots + Ki_{K+1}$ of $V$ in the denominator, we obtain
\begin{align}
\delta &= \frac12 + \frac\rho{\sqrt{2\pi V}} \sum_{j=0}^K \sum_{i_2=0}^K \cdots \sum_{i_{K+1}=0}^K \bigg( \frac{(-1)^j}{(2j+1)!} \frac{\rho^{2j}}{V^j} \prod_{k=2}^{K+1} \frac1{i_k!} \bigg( \frac{-W_k}{V^{k-1}} \bigg)^{i_k} \notag \\
&\qquad{}\times \Big( \big( 2(j+2i_2+\dots+(K+1)i_{K+1})-1 \big)!! + O\big( V^{-(K+1)} \big) \Big) \bigg) + O \bigg( \frac {\rho^{2K+3}} {V^{K+3/2}} \bigg) \notag \\
&= \frac12 + \frac\rho {\sqrt{2\pi V}} \sum_{\ell=0}^{K(1+K(K+1)/2)} \frac1{V^\ell} \sum_{j=0}^K \frac{(-1)^j \rho^{2j}}{(2j+1)!} \mathop{\sum_{i_2=0}^K \cdots \sum_{i_{K+1}=0}^K}_{i_2+2i_3+\dots+Ki_{K+1}=\ell-j} \bigg( \prod_{k=2}^{K+1} \frac{(-W_k)^{i_k}}{i_k!} \notag \\
&\qquad{}\times \big( 2(\ell+i_2+\dots+i_{K+1})-1 \big)!! \bigg) + O \bigg( \frac{\rho^{2K+3}}{V^{K+3/2}} \bigg),
\label{hey we're done}
\end{align}
where we have subsumed the first error term into the second with the help of Proposition~\ref{Wk bound prop}.

The proof of Theorem~\ref{delta series theorem} is actually now complete, although it takes a moment to recognize it. For $0\le\ell\le K$, the values of $j$ that contribute to the sum are $0\le j\le\ell$, since $\ell-j$ must be a sum of nonnegative numbers due to the condition of summation of the inner sum. In particular, all possible values of $j$ and the $i_k$ are represented in the sum, and the upper bound of $K$ for these variables is unnecessary. We therefore see that the coefficient of $\rho^{2j}V^{-\ell}$ on the right-hand side of equation~\eqref{hey we're done} matches Definition~\ref{s-coeffs def} for $s_{q;a,b}(\ell,j)$. On the other hand, for each of the finitely many larger values of $\ell$, the $\ell$th summand is bounded above by $\rho^{2K}V^{-K-1}$ times some constant depending only on $K$ (again we have used Proposition~\ref{Wk bound prop} to bound the quantities $W_k$ uniformly), which is smaller than the indicated error term once the leading factor $\rho/\sqrt{2\pi V}$ is taken into account.
\end{proof}

\section{Analysis of the variance $V(q;a,b)$}
\label{variance analysis section}

In this section we prove Theorems~\ref {variance evaluation theorem} and~\ref {M evaluation theorem}, as well as discussing related results to which our methods apply. We begin by establishing some arithmetic identities involving Dirichlet characters and their conductors in Section~\ref {characters section}. Using these identities and a classical formula for $b(\chi)$, we complete the proof of Theorem~\ref {variance evaluation theorem} in Section~\ref {variance formula section}. The linear combination of values $\frac{L'}L(1,\chi)$ that defines $M^*(q;a,b)$ can be converted into an asymptotic formula involving the von Mangoldt $\Lambda$-function, as we show in Section~\ref {M evaluation section}, and in this way we establish Theorem~\ref {M evaluation theorem}.

Our analysis to this point has the interesting consequence that the densities $\delta(q;a,b)$ can be evaluated extremely precisely using only arithmetic content, that is, arithmetic on rational numbers (including multiplicative functions of integers) and logarithms of integers; we explain this consequence in Section~\ref {arithmetic only}. Next, we show in Section~\ref {CLT section} that the limiting logarithmic distributions of the differences $E(x;q,a) - E(x;q,b)$ obey a central limit theorem as $q$ tends to infinity. Finally, we explain in Section~\ref {quadratic race section} how our analysis can be modified to apply to the race between the aggregate counting functions $\pi(x;q,N)=\#\{ p\le x\colon p$ is a quadratic nonresidue\mod q\} and $\pi(x;q,R)=\#\{ p\le x\colon p$ is a quadratic residue\mod q\}.

\subsection{Arithmetic sums over characters}
\label{characters section}

We begin by establishing some preliminary arithmetic identities that will be needed in later proofs.

\begin{prop}
Let $a$ and $b$ be distinct reduced residue classes\mod q. Then
\[
\allchisum |\chi(a)-\chi(b)|^2 = 2\phi(q),
\]
while for any reduced residue $c\not\equiv1\mod q$ we have
\[
\allchisum |\chi(a)-\chi(b)|^2 \chi(c) = -\phi(q) \big( \iota_q(cab^{-1})+\iota_q(cba^{-1}) \big),
\]
where $\iota_q$ is defined in Definition~\ref {iota and Lq and Rqn def}.
\label{just orthogonality prop}
\end{prop}

\begin{proof}
These sums are easy to evaluate using the orthogonality relation~\cite[Corollary 4.5]{magicbook}
\begin{equation}
\sum_{\chi\mod q} \chi(m) = \begin{cases}
\phi(q), &\text{if }m\equiv1\mod q \\
0, &\text{if }m\not\equiv1\mod q \\
\end{cases} \!\Bigg\}
= \phi(q)\iota_q(m).
\label{orthogonality.relation}
\end{equation}
We have
\begin{align*}
\allchisum |\chi(a)-\chi(b)|^2 &= \allchisum (2-\chi(a)\overline{\chi(b)} - \chi(b)\overline{\chi(a)}) \\
&= \allchisum 2 - \allchisum \chi(ab^{-1}) - \allchisum \chi(ba^{-1}) = 2\phi(q)+0+0,
\end{align*}
since $a\not\equiv b\mod q$. Similarly,
\begin{align*}
\allchisum |\chi(a)-\chi(b)|^2 \chi(c) &= \allchisum (2-\chi(a)\overline{\chi(b)} - \chi(b)\overline{\chi(a)}) \chi(c) \\
&= \allchisum 2\chi(c) - \allchisum \chi(cab^{-1}) - \allchisum \chi(cba^{-1}) \\
&= 0-\phi(q) \big( \iota_q(cab^{-1})+\iota_q(cba^{-1}) \big).
\end{align*}
\end{proof}

The results in the next two lemmas were discovered independently by Vorhauer (see \cite[Section 9.1, problem 8]{magicbook}).

\begin{lemma}
For any positive integer $q$, we have
\[
\sum_{d\mid q} \Lambda(q/d)\phi(d) = \phi(q) \sum_{p\mid q} \frac{\log p}{p-1},
\]
while for any proper divisor $s$ of $q$ we have
\[
\sum_{d\mid s} \Lambda(q/d)\phi(d) = \phi(q) \frac{\Lambda(q/s)}{\phi(q/s)}.
\]
\label{lambda.phi.sum.lemma}
\end{lemma}

\begin{proof}
For the first identity, we group together the contributions from the divisors $d$ such that $q/d$ is a power of a particular prime factor $p$ of $q$. If $p^r\exdiv q$, write $q=mp^r$, so that $p\nmid m$. We get a contribution to the sum only when $d=mp^{r-k}$ for some $1\le k\le r$. Therefore
\begin{equation*}
\sum_{d\mid q} \Lambda(q/d)\phi(d) = \sum_{p^r\exdiv q} \sum_{k=1}^r \Lambda(p^k) \phi(mp^{r-k}) = \sum_{p^r\exdiv q} \phi(m)\log p \sum_{k=1}^r \phi(p^{r-k}).
\end{equation*}
Since $\sum_{a\mid b} \phi(a) = b$ for any positive integer $b$, the inner sum is exactly $p^{r-1}$. Noting that $\phi(m) = \phi(q)/\phi(p^r)$ since $p\nmid n$, we obtain
\begin{equation*}
\sum_{d\mid q} \Lambda(q/d)\phi(d) = \sum_{p^r\exdiv q} \frac{\phi(q)}{\phi(p^r)} p^{r-1}\log p = \phi(q) \sum_{p\mid q} \frac{\log p}{p-1}
\end{equation*}
as claimed.

We turn now to the second identity. If $q/s$ has at least two distinct prime factors, then so will $q/d$ for every divisor $d$ of $s$, and hence all of the $\Lambda(q/d)$ terms will be 0. Therefore the entire sum equals 0, which is consistent with the claimed identity as $R_q(s)=0$ as well in this case. Therefore we need only consider the case where $q/s$ equals a prime power $p^t$.

Again write $q=mp^r$ with $p\nmid m$. Since $s=q/p^t = mp^{t-r}$, the only terms that contribute to the sum are $d=mp^{r-k}$ for $t\le k\le r$. By a similar calculation as before,
\begin{align*}
\sum_{d\mid s} \Lambda(q/d)\phi(d) = \sum_{k=t}^r \Lambda(p^k)\phi(mp^{r-k}) &= \phi(m) \log p \sum_{k=t}^r \phi(p^{r-k}) \\
&= \frac{\phi(q)}{\phi(p^r)} p^{r-t} \log p = \phi(q) \frac{\log p}{p^{t-1}(p-1)} = \phi(q) \frac{\Lambda(q/s)}{\phi(q/s)},
\end{align*}
since $q/s=p^t$. This establishes the second identity.
\end{proof}

Recall that $\chi^*$ denotes the primitive character that induces $\chi$ and that $q^*$ denotes the conductor of $\chi^*$.

\begin{prop}
For any positive integer $q$,
\[
\allchisum \log q^* = \phi(q) \bigg( \log q - \sum_{p\mid q} \frac{\log p}{p-1} \bigg),
\]
while if $a\not\equiv1\mod q$ is a reduced residue,
\[
\allchisum \chi(a) \log q^* = -\phi(q) \frac{\Lambda(q/(q,a-1))}{\phi(q/(q,a-1))}.
\]
\label{log qstar sum prop}
\end{prop}

\begin{proof}
First we show that
\begin{equation}
\sum_{\chi\mod q} \chi(a) \log q^* = \log q \sum_{\chi\mod q} \chi(a) - \sum_{d\mid q} \Lambda(q/d) \sum_{\chi\mod d} \chi(a)
\label{splitting.claim}
\end{equation}
for any reduced residue $a\mod q$. Given a character $\chi\mod q$ and a divisor $d$ of $q$, the character $\chi$ is induced by a character\mod d if and only if $d$ is a multiple of $q^*$. Therefore
\[
\sum_{d\mid q} \Lambda(q/d) \sum_{\chi\mod d} \chi(a) = \sum_{\chi\mod q} \chi(a) \sum_{\substack{d\mid q \\ q^*\mid d}} \Lambda(q/d).
\]
Making the change of variables $c=q/d$, this identity becomes
\begin{multline*}
\sum_{d\mid q} \Lambda(q/d) \sum_{\chi\mod d} \chi(a) = \sum_{\chi\mod q} \chi(a) \sum_{c\mid q/q^*} \Lambda(c) \\ = \sum_{\chi\mod q} \chi(a) \log\tfrac q{q^*} = \log q \sum_{\chi\mod q} \chi(a) - \sum_{\chi\mod q} \chi(a) \log q^*,
\end{multline*}
which verifies equation \eqref{splitting.claim}.

If $a\equiv1\mod q$, then equation \eqref{splitting.claim} becomes
\begin{multline*}
\sum_{\chi\mod q} \log q^* = \log q \sum_{\chi\mod q} 1 - \sum_{d\mid q} \Lambda(q/d) \sum_{\chi\mod d} 1 \\
= \phi(q)\log q - \sum_{d\mid q} \Lambda(q/d) \phi(d) = \phi(q)\log q - \phi(q) \sum_{p\mid q} \frac{\log p}{p-1}
\end{multline*}
by Lemma~\ref{lambda.phi.sum.lemma}, establishing the first assertion of the lemma. If on the other hand $a\not\equiv1\mod q$, then applying the orthogonality relation \eqref{orthogonality.relation} to equation \eqref{splitting.claim} yields
\begin{align*}
\sum_{\chi\mod q} \chi(a) \log q^* &= 0 - \sum_{d\mid q} \Lambda(q/d) \phi(d)\iota_d(a) \\
&= - \sum_{d\mid(q,a-1)} \Lambda(q/d) \phi(d) = -\phi(q) \frac{\Lambda(q/(q,a-1))}{\phi(q/(q,a-1))}
\end{align*}
by Lemma~\ref{lambda.phi.sum.lemma} again, establishing the second assertion of the lemma.
\end{proof}

Finally we record a proposition that involves values of both primitive characters and characters induced by them.

\begin{prop}
Let $p$ be a prime and $e$ a positive integer, and let $r$ be a reduced residue\mod q. If $p\nmid q$, then
\[
\sum_ {\chi\mod q} \chi(r) \big( \chi^*(p^e) - \chi(p^e) \big) = 0.
\]
On the other hand, if $p\mid q$ then
\[
\sum_ {\chi\mod q} \chi(r)  \big( \chi^*(p^e) - \chi(p^e) \big) = \begin{cases}
\phi \big( {q/p^{\nu}} \big), & \text{if } rp^e \equiv 1 \mod{q/p^\nu}, \\
0, & \text{otherwise},
\end{cases}
\]
where $\nu\ge1$ is the integer such that $p^\nu \parallel q$.
\label{semi orthogonality relations prop}
\end{prop}

\begin{proof}
The first assertion is trivial: if $p\nmid q$ then $\chi^*(p^e) = \chi(p^e)$ for every character $\chi\mod q$. If $p\mid q$, then $\chi(p^e)=0$ for every $\chi$, and so
$$
\sum_ {\chi\mod q} \chi(r)  \big( \chi^*(p^e) - \chi(p^e) \big) = \sum_{\chi \mod q} \chi(r) \chi^* (p^e) = \sum_{\chi \mod q} \chi^*(rp^e)
$$
since $\chi(r)=\chi^*(r)$ for every $\chi \mod q$ due to the hypothesis that $(r,q)=1$. Also, we have $\chi^* (p^e) = 0$ for any character $\chi$ such that $p\mid q^*$, and so
\[
\sum_{\chi \mod q} \chi^*(rp^e) = \sum_{\substack{\chi \mod q \\ q^* \mid q/p^\nu}} \chi^*(rp^e) = \sum_{\chi \mod {q/p^\nu}} \chi(rp^e),
\]
since $(p^e, q/p^\nu)=1$. The second assertion now follows from the orthogonality relation~\eqref {orthogonality.relation}.
\end{proof}

\subsection{A formula for the variance}
\label{variance formula section}

Recall that $b(\chi)$ was defined in Definition~\ref {bchi and Vqab def}; we record a classical formula for $b(\chi)$ in the next lemma, after which we will be able to prove Theorem~\ref{variance evaluation theorem}.

\begin{lemma}
\label{b.chi.nonprimitive.lemma}
Assume GRH. Let $q\ge3$, and let $\chi$ be any nonprincipal character modulo $q$. Then
\begin{equation*}
b(\chi) = \log\frac{q^*}\pi - \gamma_0 - (1+\chi(-1))\log2 + 2\Re\frac{L'(1,\chi^*)}{L(1,\chi^*)}.
\end{equation*}
\end{lemma}

\begin{proof}
Since the zeros of $L(s,\chi)$ and $L(s,\chi^*)$ on the line $\Re z=\frac12$ are identical, it suffices to show that for any primitive character $\chi$ modulo $q$,
\begin{equation*}
\Lsum{\gamma\in\R} \frac1{\frac14+\gamma^2} = \log\frac q\pi - \gamma_0 - (1+\chi(-1))\log2 + 2\Re\frac{L'(1,\chi)}{L(1,\chi)}.
\end{equation*}
There is a certain constant $B(\chi)$ that appears in the Hadamard product formula for $L(s,\chi)$. One classical formula related to it~\cite[equation (10.38)]{magicbook} is
\begin{equation}
\Re B(\chi) = - \sum_{\substack{\rho\in\C \\ 0<\Re\rho<1 \\ L(\rho,\chi)=0}} \Re \frac1\rho.
\label{Re Bchi}
\end{equation}
We can relate $B(\chi)$ to $b(\chi)$ under GRH by rewriting the previous equation as
\begin{equation}
  -2\Re B(\chi) = \Lsum{\gamma\in\R} \Re \bigg( \frac2{\frac12+i\gamma} \bigg) = \Lsum{\gamma\in\R} \Re \bigg( \frac{1-2i\gamma}{\frac14+\gamma^2} \bigg) = b(\chi).
\label{relate B to b}
\end{equation}
On the other hand, Vorhauer showed in 2006 (see~\cite[equation (10.39)]{magicbook}) that
\[
  B(\chi) = -\frac12\log \frac q\pi - \frac{L'}L(1,\bar\chi) + \frac{\gamma_0}2 + \frac{1+\chi(-1)}2\log 2.
\]
Taking real parts (which renders moot the difference between $\bar\chi$ and $\chi$) and comparing to equation~\eqref{relate B to b} establishes the lemma.
\end{proof}

\begin{proof}[Proof of Theorem~\ref{variance evaluation theorem}]
We begin by applying Lemma~\ref{b.chi.nonprimitive.lemma} to Definition~\ref {bchi and Vqab def} for $V(q;a,b)$, which yields
\begin{align}
V(q;a,b) &= \mostchisum |\chi(a)-\chi(b)|^2 \Big( \log\frac{q^*}\pi - \gamma_0 - (1+\chi(-1))\log2 + 2\Re\frac{L'(1,\chi^*)}{L(1,\chi^*)} \Big) \notag \\
&= \allchisum |\chi(a)-\chi(b)|^2 \log q^* - (\gamma_0+\log2\pi) \allchisum |\chi(a)-\chi(b)|^2 \notag \\
&\qquad{}- \log2 \allchisum |\chi(a)-\chi(b)|^2 \chi(-1) + 2 M^*(q;a,b), \label{V.with.b.substituted}
\end{align}
recalling Definition~\ref{Mqab def} for $M^*(q;a,b)$. We are permitted to reinclude the principal character $\chi_0$ in the three sums on the right-hand side, since the coefficient $|\chi_0(a)-\chi_0(b)|^2$ always equals~0.

The second and third terms on the right-hand side of equation \eqref{V.with.b.substituted} are easy to evaluate using Proposition~\ref {just orthogonality prop}: we have
\begin{equation}
-(\gamma_0+\log2\pi) \allchisum |\chi(a)-\chi(b)|^2 = -2(\gamma_0+\log2\pi)\phi(q)
\label{constant.term}
\end{equation}
and
\begin{align}
-\log2 \allchisum |\chi(a)-\chi(b)|^2 \chi(-1) &= (\log2) \phi(q) \big( \iota_q(-ab^{-1})+\iota_q(-ba^{-1}) \big)  \notag \\ &= (2\log2) \phi(q) \iota_q(-ab^{-1}).
\label{minus.one.term}
\end{align}
The first sum on the right-hand side of equation \eqref{V.with.b.substituted} can be evaluated using Proposition~\ref {log qstar sum prop}:
\begin{align*}
\allchisum |\chi(a)-{} & \chi(b)|^2 \log q^* = \sum_{\chi\mod q} (2-\chi(ab^{-1})-\chi(ba^{-1})) \log q^* \notag \\
&= 2\phi(q)\bigg( \log q - \sum_{p\mid q} \frac{\log p}{p-1} \bigg) + \phi(q) \frac{\Lambda(q/(q,ab^{-1}-1))}{\phi(q/(q,ab^{-1}-1))} + \phi(q) \frac{\Lambda(q/(q,ba^{-1}-1))}{\phi(q/(q,ba^{-1}-1))}.
\end{align*}
Since $(q,mn) = (q,n)$ for any integer $m$ that is relatively prime to $q$, we see that $(q,ab^{-1}-1) = (q,a-b) = (q,b-a) = (q,ba^{-1}-1)$, and therefore
\begin{equation}
\allchisum |\chi(a)-\chi(b)|^2 \log q^* = 2\phi(q) \bigg( \log q - \sum_{p\mid q} \frac{\log p}{p-1} + \frac{\Lambda(q/(q,a-b))}{\phi(q/(q,a-b))} \bigg).
\label{R.has.been.simplified}
\end{equation}
Substituting the evaluations \eqref{constant.term}, \eqref{minus.one.term}, and~\eqref{R.has.been.simplified} into equation~\eqref{V.with.b.substituted}, we obtain
\begin{align*}
V(q;a,b) &= 2\phi(q) \bigg( \log q - \sum_{p\mid q} \frac{\log p}{p-1} + \frac{\Lambda(q/(q,a-b))}{\phi(q/(q,a-b))} \bigg) \\
&\qquad{} - 2(\gamma_0+\log2\pi)\phi(q) + (2\log2)\phi(q) \iota_q(-ab^{-1}) + 2 M^*(q;a,b) \\
&= 2\phi(q) \big( \L(q) + K_q(a-b) + \iota_q(-ab^{-1})\log2 \big) + 2M^*(q;a,b),
\end{align*}
where $\L(q)$ and $K_q(n)$ were defined in Definition~\ref {iota and Lq and Rqn def}. This establishes the theorem.
\end{proof}

Theorem~\ref{variance evaluation theorem} has the following asymptotic formula as a corollary:

\begin{prop}
\label{Vqab asymptotic prop}
Assuming GRH, we have $V(q;a,b) = 2\phi(q)\log q + O(\phi(q)\log\log q)$.
\end{prop}

\begin{proof}
First note that the function $(\log t)/(t-1)$ is decreasing for $t>1$. Consequently, $\Lambda(q)/\phi(q)$ is bounded by $\log 2$. Also, letting $p_j$ denote the $j$th prime, we see that
\[
\sum_{p\mid q} \frac{\log p}{p-1} \le \sum_{j=1}^{\omega(q)} \frac{\log p_j}{p_j-1} \ll \log p_{\omega(q)} \ll \log {\omega(q)} \ll \log\log q,
\]
where the final inequality uses the trivial bound $\omega(q) \le (\log q)/(\log 2)$. From Definition~\ref {iota and Lq and Rqn def}, we conclude that $\L(q) = \log q + O(\log\log q)$. Next, $K_q(a-b)$ is bounded by $\log 2$ as above, and $\iota_q(ab^{-1})$ is of course bounded as well. Finally, on GRH we know that $L'(1,\chi^*)/L(1,\chi^*) \ll \log\log q^* \le \log\log q$ (either see \cite{littlewood}, or take $y=\log ^2 q$ in Proposition~\ref{Sound method prop}), which immediately implies that $M^*(q;a,b) \ll \phi(q)\log\log q$ by Definition~\ref {Mqab def}. The proposition now follows from Theorem~\ref{variance evaluation theorem}.
\end{proof}

\subsection{Evaluation of the analytic term $M^*(q;a,b)$}
\label{M evaluation section}

The goal of this section is a proof of Theorem~\ref{M evaluation theorem}. We start by examining more closely, in the next two lemmas, the relationship between the quantities $M^*(q;a,b)$ and $M(q;a,b)$ defined in Definition~\ref{Mqab def}. Recall that $e(q;p,r)$ was defined in Definition~\ref {hqpr def}.

\begin{lemma}
If $p^\nu \parallel q$, then $\displaystyle \sum_{\substack{e\ge1 \\ rp^e \equiv 1\mod {q/p^\nu}}} \frac1{p^e} = \frac1{p^{e(q;p,r)}(1-p^{-e(q;p,1)})}$.
\label{geometric series lemma}
\end{lemma}

\begin{proof}
If $r$ is not in the multiplicative subgroup$\mod{q/p^\nu}$ generated by $p$, then the left-hand side is clearly zero, while the right-hand side is zero by the convention that $e(q;p,r) = \infty$ in this case. Otherwise, the positive integers $e$ for which $rp^e \equiv 1\mod{q/p^\nu}$ are precisely the ones of the form $e(q;p,r) + ke(q;p,1)$ for $k\ge0$, since $e(q;p,r)$ is the first such integer and $e(q;p,1)$ is the order of $p\mod{q/p^\nu}$. Therefore we obtain the geometric series
\[
\sum_{\substack{e\ge1 \\ rp^e \equiv 1\mod {q/p^\nu}}} \frac1{p^e} = \sum_{k=0}^\infty \frac1{p^{e(q;p,r) + ke(q;p,1)}} = \frac1{p^{e(q;p,r)}(1-p^{-e(q;p,1)})}
\]
as claimed.
\end{proof}

\begin{definition}
If $p^\nu\parallel q$, define
\[
h_0(q;p,r) = \frac{1}{\phi(p^{\nu})}\frac{\log p}{p^{e(q;p,r)} (1-p^{-e(q;p,1)})}
\]
and
\[
H_0(q;a,b) = \sum_{p\mid q} \big( h_0(q;p,ab^{-1}) +h_0(q;p,ba^{-1}) -2h_0(q;p,1) \big).
\]
We will see later in this section, in the proof of Theorem~\ref{M evaluation theorem}, that $h_0$ and $H_0$ are very close to the functions $h$ and $H$ also defined in Definition~\ref {hqpr def}. Notice that if $q$ is prime, then $h_0(q;q,r) = (\log q)/q(q-1)$ independent of $r$ and thus $H(q;a,b)=0$ for any $a$ and $b$.
\qedef
\label{h0 and H0 def}
\end{definition}

The next lemma could be proved under a hypothesis much weaker than GRH, but this is irrelevant to our present purposes.

\begin{lemma}
Assume GRH. If $a$ and $b$ are reduced residues\mod q, then
\[
M^*(q;a,b) = M(q;a,b) + \phi(q) H_0(q;a,b),
\]
where $M^*(q;a,b)$ and $M(q;a,b)$ are defined in Definition~\ref{Mqab def}.
\label{Mstar to M lemma}
\end{lemma}

\begin{proof}
We begin with the identity
\[
\frac{L'(1,\chi)}{L(1,\chi)} = - \lim_{y\to\infty} \sum_{p\le y} \sum_{e=1}^\infty \frac{\chi(p^e)\log p}{p^e}.
\]
This identity follows from the fact that the Euler product of $L(s,\chi)$ converges uniformly for $\Re(s)\ge 1/2+\ep$; this is implied by the estimate $\sum_{p\le x} \chi(p) \ll_{q} x^{1/2} \log^2 x$ which itself is a consequence of GRH.

Therefore
\begin{align*}
M^*(q;a,b) - M(q;a,b) &= \mostchisum \big| \chi(a)-\chi(b) \big|^2 \bigg( \frac{L'(1,\chi^*)}{L(1,\chi^*)} - \frac{L'(1,\chi)}{L(1,\chi)} \bigg) \\
&= - \mostchisum \big| \chi(a)-\chi(b) \big|^2 \lim_{y\to\infty} \sum_{p\le y} \log p \sum_{e=1}^\infty \frac{\chi^*(p^e)-\chi(p^e)}{p^e} \\
&=  \lim_{y\to\infty} \sum_{p\le y} \log p \sum_{e=1}^\infty \frac1{p^e} \allchisum \big( \chi(ab^{-1}) + \chi(ba^{-1}) - 2 \big) \big( \chi^*(p^e)-\chi(p^e) \big),
\end{align*}
where the inserted term involving $\chi_0$ is always zero. Proposition~\ref {semi orthogonality relations prop} tells us that the inner sum vanishes except possibly when the prime $p$ divides $q$; invoking that proposition three times, we see that
\begin{multline*}
M^*(q;a,b) - M(q;a,b) = \sum_{p^\nu \parallel q} \phi(q/p^\nu) \log p \\
{}\times \bigg( \sum_{\substack{e\ge1 \\ ab^{-1}p^e \equiv 1\mod{q/p^\nu}}} \frac1{p^e} + \sum_{\substack{e\ge1 \\ ba^{-1}p^e \equiv 1\mod{q/p^\nu}}} \frac1{p^e} - 2 \sum_{\substack{e\ge1 \\ p^e \equiv 1\mod{q/p^\nu}}} \frac1{p^e} \bigg).
\end{multline*}
We can evaluate these inner sums using Lemma~\ref{geometric series lemma}: by comparison with Definition~\ref {h0 and H0 def},
\begin{align*}
M^*(q;a,b) - M(q;a,b) &= \phi(q) \sum_{p^\nu \parallel q} \frac{\log p}{\phi(p^\nu)} \bigg( \frac1{p^{e(q;p,ab^{-1})}(1-p^{-e(q;p,1)})} \\
&\qquad{} + \frac1{p^{e(q;p,ba^{-1})}(1-p^{-e(q;p,1)})} - 2\frac1{p^{e(q;p,1)}(1-p^{-e(q;p,1)})} \bigg) \\
&= \phi(q)^{\mathstrut} H_0(q;a,b),
\end{align*}
which establishes the lemma.
\end{proof}

We will need the following three propositions, with explicit constants given, when we undertake our calculations and estimations of $\delta(q;a,b)$. Because the need for explicit constants makes their derivations rather lengthy, we will defer the proofs of the first two propositions until Section~\ref {variance bound section} and derive only the third one in this section.

\begin{prop}
\label{Sound method prop}
Assume GRH. Let $\chi$ be a nonprincipal character\mod q. For any positive real number $y$,
$$
\frac{L'(1,\chi)}{L(1,\chi)} = - \sum_{n=1}^{\infty} \frac{\chi(n) \Lambda(n)}{n} e^{-n/y} + \Obar\bigg( \frac{14.27\log q + 16.25}{y^{1/2}}  +  \frac{16.1\log q+17.83} {y^{3/4}} \bigg).
$$
\end{prop}

\begin{prop}
\label{only the first term survives prop}
If $1\le a <q$, then
$$
\sum_ {n\equiv a \mod q} \frac{\Lambda(n)}{n} e^{-{n/q^2}} = \frac{\Lambda(a)}{a} + \Obar\bigg( \frac{2\log^2 q}{q} + \frac{3.935\log q}{q} \bigg).
$$
\end{prop}

Assuming these propositions for the moment, we can derive the following explicit estimate for $M^*(q;a,b)$, after which we will be able to finish the proof of Theorem~\ref{M evaluation theorem}.

\begin{prop}
\label{M exact evaluation prop}
Assume GRH. For any pair $a,b$ of distinct reduced residues modulo $q$, let $r_1$ and $r_2$ denote the least positive residues of $ab^{-1}$ and $ba^{-1}\mod q$. Then for $q\ge150$,
\begin{equation*}
M^*(q;a,b) = \phi(q) \bigg( \frac{\Lambda(r_1)}{r_1} +\frac{\Lambda(r_2)}{r_2} + H_0(q;a,b) \bigg) + \Obar\bigg( \frac{23.619\phi(q)\log^2 q}q \bigg).
\end{equation*}
\end{prop}

\begin{proof}
The bulk of the proof is devoted to understanding $M(q;a,b)$. From Proposition~\ref {Sound method prop}, we have
\begin{align}
M(q;a,b) &= \allchisum \big| \chi(a) - \chi(b) \big|^2 \frac{L'(1,\chi)}{L(1,\chi)} \notag \\
&= \allchisum \big( 2 - \chi(ba^{-1}) - \chi(ab^{-1})\big) \bigg( {-} \sum_{n=1}^\infty \frac{\Lambda(n)\chi(n)}n e^{-n/y} \notag \\
&\qquad\qquad{}+ \Obar\bigg( \frac{14.27\log q + 10.6}{y^{1/2}}  +  \frac{16.1\log q+13.1} {y^{3/4}} \bigg) \bigg) \notag \\
&= \sum_{n=1}^\infty \frac{\Lambda(n)}n e^{-n/y} \allchisum \big( \chi(ba^{-1}n) + \chi(ab^{-1}n) - 2\chi(n) \big) \notag \\
&\qquad\qquad{}+ 4\phi(q) \Obar\bigg( \frac{14.27\log q + 16.25}{y^{1/2}}  +  \frac{16.1\log q+17.83} {y^{3/4}} \bigg) \bigg),
\label{will take y to infinity}
\end{align}
and using the orthogonality relations in Proposition~\ref {just orthogonality prop}, we see that
\begin{multline*}
M(q;a,b) = \phi(q) \bigg( \sum_{n\equiv ab^{-1}\mod q} \frac{\Lambda(n)}n e^{-n/y} + \sum_{n\equiv ba^{-1}\mod q} \frac{\Lambda(n)}n e^{-n/y} - 2\sum_{n\equiv1\mod q} \frac{\Lambda(n)}n e^{-n/y} \bigg)
\\
+ 4\phi(q) \Obar\bigg( \frac{14.27\log q + 16.25}{y^{1/2}}  +  \frac{16.1\log q+17.83} {y^{3/4}} \bigg) \bigg).
\end{multline*}
At this point we choose $y=q^2$. We calculate that ${(14.27\log q + 16.25)/q}  + {(16.1\log q+17.83)/q^{3/2}} < 3.816(\log^2q)/q$ for $q\ge150$, and so
\begin{multline*}
M(q;a,b) = \phi(q) \bigg( \sum_{n\equiv ab^{-1}\mod q} \frac{\Lambda(n)}n e^{-n/q^2} + \sum_{n\equiv ba^{-1}\mod q} \frac{\Lambda(n)}n e^{-n/q^2} - 2\sum_{n\equiv1\mod q} \frac{\Lambda(n)}n e^{-n/q^2} \bigg)
\\
{}+ \Obar\bigg(  \frac{15.263\phi(q)\log^2 q}q \bigg).
\end{multline*}

Let $r_1$ and $r_2$ denote the least positive residues of $ab^{-1}$ and $ba^{-1}\mod q$. Using Proposition~\ref {only the first term survives prop} three times, we see that
\begin{align*}
M(q;a,b) &= \phi(q) \bigg( \frac{\Lambda(r_1)}{r_1} + \frac{\Lambda(r_2)}{r_2} - 2\frac{\Lambda(1)}1 + \Obar\bigg( 3 \Big( \frac{2\log^2 q}{q} + \frac{3.935\log q}{q} \Big) \bigg) \bigg) + \Obar\bigg(  \frac{15.263\phi(q)\log^2 q}q \bigg) \\
&= \phi(q) \bigg( \frac{\Lambda(r_1)}{r_1} + \frac{\Lambda(r_2)}{r_2} \bigg) + \Obar\bigg( \frac{36.619\phi(q)\log^2 q}q \bigg)
\end{align*}
for $q\ge150$. With this understanding of $M(q;a,b)$, the proposition now follows for $M^*(q;a,b)$ by Lemma~\ref{Mstar to M lemma}.
\end{proof}

\begin{proof}[Proof of Theorem~\ref{M evaluation theorem}]
Since Proposition~\ref{M exact evaluation prop} tells us that
\begin{equation*}
M^*(q;a,b) = \phi(q) \bigg( \frac{\Lambda(r_1)}{r_1} +\frac{\Lambda(r_2)}{r_2} + H_0(q;a,b) + O\bigg( \frac{\log^2 q}q \bigg) \bigg),
\end{equation*}
all we need to do to prove the theorem is to show that
\[
H_0(q;a,b) = H(q;a,b) + O\bigg( \frac{\log^2 q}q \bigg).
\]
The key observation is that $p^{e(q;p,1)} \equiv 1\mod{q/p^\nu}$ and $p^{e(q;p,1)} \ge p^1 > 1$, and so $p^{e(q;p,1)} > q/p^\nu$. Therefore by Definitions~\ref{hqpr def} and~\ref{h0 and H0 def}, we have $h_0(q;p,r) = h(q;p,r) (1+O(p^\nu/q))$ and $h(q;p,1) \ll (\log p)/\phi(p^\nu)(q/p^\nu) \ll (\log p)/q$. We see that
\begin{align*}
H_0(q;a,b) &= \sum_{p^{\nu} \parallel q} \big( h_0(q;p,ab^{-1}) +h_0(q;p,ba^{-1}) -2h_0(q;p,1) \big) \\
&= \sum_{p^{\nu} \parallel q} \bigg( \big( h(q;p,ab^{-1}) +h(q;p,ba^{-1}) \big)\big( 1 + O\big( \tfrac{p^\nu}q \big) \big) + O\big( \tfrac{\log p}q \big) \bigg).
\end{align*}
It is certainly true that $h(q;p,r) \ll (\log p)/\phi(p^\nu)  \ll (\log p)/p^\nu$, and so the previous equation becomes
\begin{equation*}
H_0(q;a,b) = H(q;a,b) + O\bigg( \sum_{p^{\nu} \parallel q} \bigg( \frac{\log p}{p^\nu} \frac{p^\nu}q + \frac{\log p}q \bigg) \bigg) = H(q;a,b) + O\bigg( \frac{\log q}q \bigg),
\end{equation*}
which establishes the theorem.
\end{proof}

\subsection{Estimates in terms of arithmetic information only}
\label{arithmetic only}

The purpose of this section is to show that the densities $\delta(q;a,b)$ can be calculated extremely precisely using only ``arithmetic information''. For the purposes of this section, ``arithmetic information'' means finite expressions composed of elementary arithmetic operations involving only integers, logarithms of integers, values of the Riemann zeta function at positive integers, and the constants $\pi$ and $\gamma_0$. (In fact, all of these quantities themselves can in principal be calculated arbitrarily precisely using only elementary arithmetic operations on integers.) The point is that ``arithmetic information'' excludes integrals and such quantities as Dirichlet characters and $L$-functions, Bessel functions, and trigonometric functions. The formula we can derive, with only arithmetic information in the main term, has an error term of the form $O_A(q^{-A})$ for any constant $A>0$ we care to specify in advance.

To begin, we note that letting $y$ tend to infinity in equation~\eqref{will take y to infinity} leads to the heuristic statement
\begin{align*}
M(q;a,b) &= \sum_{n=1}^\infty \frac{\Lambda(n)}n \allchisum \big( \chi(ba^{-1}n) + \chi(ab^{-1}n) - 2\chi(n) \big) \notag \\
&= \phi(q) \sum_{n=1}^\infty \frac{\Lambda(n)}n \big( \iota_q(ba^{-1}n) + \iota_q(ab^{-1}n) - 2 \iota_q(n) \big) \notag \\
&{}\mathrel{\text{``$=$''}} \phi(q) \bigg( \sum_{n\equiv ab^{-1}\mod q} \frac{\Lambda(n)}n + \sum_{n\equiv ba^{-1}\mod q} \frac{\Lambda(n)}n - 2\sum_{n\equiv1\mod q} \frac{\Lambda(n)}n \bigg),
\end{align*}
where the ``$=$'' warns that the sums on the right-hand side do not individually converge. In fact, using a different approach based on the explicit formula, one can obtain
\begin{multline}
M(q;a,b) = \phi(q) \bigg( \sum_{\substack{1\le n\le y \\ n\equiv ab^{-1}\mod q}} \frac{\Lambda(n)}n + \sum_ {\substack{1\le n\le y \\ n\equiv ba^{-1}\mod q}} \frac{\Lambda(n)}n - 2\sum_ {\substack{1\le n\le y \\ n\equiv 1\mod q}} \frac{\Lambda(n)}n \bigg) \\
+ O\bigg( \frac{\phi(q) \log^2qy}{\sqrt y} \bigg).
\label{M explicit formula version}
\end{multline}
In light of Theorem~\ref{variance evaluation theorem} in conjunction with Lemma~\ref{Mstar to M lemma}, we see that we can get an arbitrarily good approximation to $V(q;a,b)$ using only arithmetic information.

By Theorem~\ref {delta series theorem}, we see we can thus obtain an extremely precise approximation for $\delta(q;a,b)$ as long as we can calculate the coefficients $s_{q;a,b}(\ell,j)$ defined in Definition~\ref {s-coeffs def}. Inspecting that definition reveals that it suffices to be able to calculate $W_m(q;a,b)$ (or equivalently $W_m(q;a,b)V(q;a,b)$) arbitrarily precisely using only arithmetic content. With the next several lemmas, we describe how such a calculation can be made.

\begin{lemma}
\label{partial fractions decomposition}
Let $n$ be a positive integer, and set $\ell=\lfloor\frac{n}{2} \rfloor$. There exist rational numbers $C_{n,1}$, \dots, $C_{n,\ell}$ such that
\[
\frac1{(1/4+t^2)^n} = 2\Re \bigg( \frac1 {(1/2-it)^n} \bigg) + \frac{C_{n,1}}{(1/4+t^2)^{n-1}} + \frac{C_{n,2}}{(1/4+t^2)^{n-2}} + \dots + \frac{C_{n,\ell}}{(1/4+t^2)^{n-\ell}}
\]
for any complex number~$t$.
\end{lemma}

\begin{proof}
Since
\[
2\Re \frac{1}{(1/2-it)^n} = \frac{1}{(1/2-it)^n} + \frac{1}{(1/2+it)^n} = \frac{(1/2+it)^n+(1/2-it)^n}{(1/4+t^2)^n},
\]
it suffices to show that
\begin{equation}
\label{partial fractions equation}
\frac{(1/2+it)^n+(1/2-it)^n}{(1/4+t^2)^n} = \frac{C_{n,0}}{(1/4+t^2)^{n}} + \frac{-C_{n,1}}{(1/4+t^2)^{n-1}} + \dots + \frac{-C_{n,\ell}}{ (1/4+t^2)^{n-\ell}},
\end{equation}
where each $C_{n,m}$ is a rational number and $C_{n,0}=1$. In fact, we need only show that this identity holds for some rational number $C_{n,0}$, since multiplying both sides by $(1/4+t^2)^n$ and taking the limit as $t$ tends to $i/2$ proves that $C_{n,0}$ must equal~$1$.

Using the binomial theorem,
\begin{align*}
(1/2+it)^n+(1/2-it)^n &= \sum_{k=0}^n \textstyle \binom nk \big( \frac12 \big)^{n-k} \big( (it)^k + (-it)^k \big) \\
&= \sum_{j=0}^\ell \textstyle \binom n{2j} \big( \frac12 \big)^{n-2j} \big( 2(-1)^j t^{2j} \big) \\
&= 2 \sum_{j=0}^\ell \textstyle \binom n{2j} \big( \frac12 \big)^{n-2j} (-1)^j \big(  (\frac14 + t^{2}) - \frac14 \big)^j \\
&= 2 \sum_{j=0}^\ell {\textstyle \binom n{2j} \big( \frac12 \big)^{n-2j} (-1)^j} \sum_{m=0}^j \textstyle \binom jm  (\frac14 + t^2)^m \big( {-} \frac14 \big)^{j-m},
\end{align*}
which is a linear combination of the expressions $(1/4+t^2)^m$, for $0\le m\le \ell$, with rational coefficients not depending on~$t$. Dividing both sides by $(1/4+t^2)^n$ establishes equation~\eqref {partial fractions equation} for suitable rational numbers $C_{n,m}$ and hence the lemma.
\end{proof}

For the rest of this section, we say that a quantity is a {\em fixed $\Q$-linear combination} of certain elements if the coefficients of this linear combination are rational numbers that are independent of $q,a,b$ and $\chi$ (but may depend on~$n$ and~$j$ where appropriate). Our methods allow the exact calculation of these rational coefficients, but the point of this section would be obscured by the bookkeeping required to record them.

\begin{definition}
As usual, $\Gamma(z)$ denotes Euler's Gamma function.
For any positive integer $n$ and any Dirichlet character $\chi\mod q$, define
\[
b_n(\chi) = \Lsum{\gamma\in\R} \frac1{(\frac14+\gamma^2)^n},
\]
so that $b_1(\chi) = b(\chi)$ for example.
\qedef
\label{bnchi def}
\end{definition}

\begin{lemma}
\label{b(chi) lemma}
Assume GRH. Let $n$ be a positive integer, and let $\chi$ be a primitive character\mod q. Then $b_n(\chi)$ is a fixed $\Q$-linear combination of the quantities
\begin{multline}
\bigg\{\log \frac q\pi, \bigg[ \frac{d}{ds}\log \Gamma(s) \bigg]_{s=(1+\xi)/2}, \dots, \bigg[ \frac{d^n}{ds^n}\log \Gamma(s) \bigg]_{s=(1+\xi)/2} , \\
\Re\bigg[ \frac{d}{ds}\log L(s,\chi) \bigg]_{s=1}, \dots, \Re\bigg[ \frac{d^n}{ds^n}\log L(s,\chi) \bigg]_{s=1}\bigg\},
\label{set def}
\end{multline}
where $\xi=0$ if $\chi(-1)=1$ and $\xi=1$ if $\chi(-1)=-1$.
\end{lemma}

\begin{remark}
Since the critical zeros of $L(s,\chi)$ and $L(s,\chi^*)$ are identical, the lemma holds for any nonprincipal character $\chi$ if, in the set~\eqref{set def}, we replace $q$ by $q^*$ and $L(s,\chi)$ by $L(s,\chi^*)$.
\end{remark}

\begin{proof}
For primitive characters $\chi$, Lemma~\ref{b.chi.nonprimitive.lemma} tells us that
\begin{align*}
b(\chi) &= \log\frac q\pi - \gamma_0 - (1+\chi(-1))\log2 + 2\Re\frac{L'(1,\chi)}{L(1,\chi)} \\
&= \log\frac q\pi + \bigg[ \frac{\Gamma'(s)}{\Gamma(s)} \bigg]_{s=(1+\xi)/2}  + 2\Re\frac{L'(1,\chi)}{L(1,\chi)},
\end{align*}
which establishes the lemma for $n=1$. We proceed by induction on $n$. By Lemma~\ref {partial fractions decomposition}, we see that
\begin{align}
b_n(\chi) &= \Lsum{\gamma\in\R} \frac1{(1/4+\gamma^2)^n} \notag \\
&= \Lsum{\gamma\in\R} \bigg( 2\Re \frac1 {(1/2-i\gamma)^n} + \frac{C_{n,1}}{(1/4+\gamma^2)^{n-1}} + \frac{C_{n,2}}{(1/4+\gamma^2)^{n-2}} + \dots + \frac{C_{n,\ell}}{(1/4+\gamma^2)^{n-\ell}} \bigg) \notag \\
&= C_{n,1} b_{n-1}(\chi) + \dots + C_{n,\ell} b_{n-\ell}(\chi) + 2 \Lsum{\gamma\in\R} \Re \frac1 {(1/2-i\gamma)^n}
\label{leftover sum}
\end{align}
(where $\ell=\lfloor \frac n2 \rfloor$). By the induction hypothesis, each term of the form $C_{n,m} b_{n-m}(\chi)$ is a fixed $\Q$-linear combination of the elements of the set~\eqref {set def}; therefore all that remains is to show that the sum on the right-hand side of equation~\eqref{leftover sum} is also a fixed $\Q$-linear combination of these elements.

Consider the known formula~\cite[equation (10.37)]{magicbook}
\[
\frac d{ds}\log L(s,\chi) = B(\chi) - \frac d{ds}\log\Gamma\bigg( \frac{s+\xi}2 \bigg) - \frac12\log\frac q\pi + \sum_\rho \bigg( \frac1{s-\rho} + \frac1\rho \bigg),
\]
where $\sum_\rho$ denotes a sum over all nontrivial zeros of $L(s,\chi)$ and $B(\chi)$ is a constant (alluded to in the proof of Lemma~\ref {b.chi.nonprimitive.lemma}). If we differentiate this formula $n-1$ times with respect to $s$, we obtain
\[
\frac{d^n}{ds^n}\log L(s,\chi) = - \frac{d^n}{ds^n}\log\Gamma\bigg( \frac{s+\xi}2 \bigg) + \sum_\rho \frac{(-1)^{n-1}(n-1)!}{(s-\rho)^n}.
\]
Setting $s=1$ and taking real parts, and using GRH, we conclude that
\begin{align*}
\Lsum{\gamma\in\R} \Re\frac1{(1/2-i\gamma)^n} &= \sum_\rho \Re\frac1{(1-\rho)^n} \\
&= \frac{(-1)^{n-1}}{(n-1)!} \bigg[ \Re \frac{d^n}{ds^n} \log L(s,\chi) + \frac{d^n}{ds^n} \log \Gamma \bigg( \frac{s+\xi}{2} \bigg) \bigg]_{s=1},
\end{align*}
which is a fixed $\Q$-linear combination of the elements of the set~\eqref {set def} as desired. (Although $\big[ \frac{d^n}{ds^n}\log \Gamma(s) \big]_{s=(1+\xi)/2}$ and $\big[ \frac{d^n}{ds^n} \log \Gamma \big( \frac{s+\xi}{2} \big) \big]_{s=1}$ differ by a factor of $2^n$, this does not invalidate the conclusion.)
\end{proof}

The following three definitions, which generalize earlier notation, will be important in our analysis of the higher-order terms $W_n(q;a,b)V(q;a,b)$.

\begin{definition}
For any positive integers $q$ and $n$, define
\begin{multline*}
\L_n(q) = \sum_{|i|\le n} (-1)^i \binom{2n}{n+i} \bigg(\iota_q(a^ib^{-i}) \bigg( \log \frac q\pi - \sum_{p\mid q} \frac{\log p}{p-1} \bigg) \\
- \big( 1-\iota_q(a^ib^{-i}) \big) \frac{\Lambda(q/(q,a^i-b^i))}{\phi(q/(q,a^i-b^i))} \bigg).
\end{multline*}
\qedef
\label{Lnq def}
\end{definition}

\begin{definition}
Let $\chi$ be a Dirichlet character\mod q, and let $a$ and $b$ be integers. For any positive integers $j\le n$, define
\begin{equation*}
\M_{n,j}^*(q;a,b) = \frac1{\phi(q)} \mostchisum |\chi(a)-\chi(b)|^{2n} \bigg[ \frac{d^j}{ds^j}\log L(s,\chi^*) \bigg]_{s=1}
\end{equation*}
and
\begin{equation*}
\M_{n,j}(q;a,b) = \frac1{\phi(q)} \mostchisum |\chi(a)-\chi(b)|^{2n} \bigg[ \frac{d^j}{ds^j}\log L(s,\chi) \bigg]_{s=1},
\end{equation*}
so that $\M_{1,1}^*(q;a,b) = M^*(q;a,b)/\phi(q)$ and $\M_{1,1}(q;a,b) = M(q;a,b)/\phi(q)$ for example. One can use Lemma~\ref {d lemma} and Perron's formula to show that
$$
\M_{n,j}(q;a,b) = (-1)^j \sum_{|i|\le j} (-1)^i \binom{2j}{j+i} \sum_{\substack{n\le y \\ n\equiv a^ib^{-i} \mod q}} \frac{\Lambda(n)\log^{j-1} n}{n} + O_j\bigg( \frac {\log^{j+1} qy}{\sqrt y} \bigg),
$$
in analogy with equation~\eqref{M explicit formula version}.
\qedef
\label{Mnjqab def}
\end{definition}

\begin{definition}
For any distinct reduced residue classes $a$ and $b\mod q$, define
\[
H_{n,j}(q;a,b) = (-1)^j \sum_{p^{\nu} \parallel q} \frac{(\log p)^j} {\phi(p^{\nu})} \sum_{|i|\le j} (-1)^i \binom{2n}{n+i} \sum_{\substack{e\ge1 \\ a^ib^{-i}p^e \equiv 1\mod{q/p^\nu}}} \frac{e^{j-1}}{p^e}
\]
for any integers $1\le j\le n$. Notice that the inner sum is
\[
\sum_{\substack{e\ge1 \\ a^ib^{-i}p^e \equiv 1\mod{q/p^\nu}}} \frac{e^{j-1}}{p^e} = \sum_{\substack{e\ge1 \\ e \equiv e(q;p,a^ib^{-i}) \mod{e(q;p,1)}}} \frac{e^{j-1}}{p^e},
\]
where $e(q;p,r)$ is defined in Definition~\ref {hqpr def}. It turns out that the identity
\[
\sum_{\substack{e\ge1 \\ e\equiv r\mod s}} \frac{e^m}{p^e} = \frac1{p^r(1-p^{-s})} \sum_{g=0}^m \binom mg s^g r^{m-g} \sum_{\ell=0}^g \stirling g\ell \frac{\ell!}{(p^s-1)^\ell}
\]
(in which $\stirling g\ell$ denotes the Stirling number of the second kind) is valid for any positive integers $m$, $p$, $r$, and $s$ such that $r\le s$ (as one can see by expanding $(sk+r)^m$ by the binomial theorem and then invoking the identity~\cite[(equation 7.46)]{GKP}). Consequently, we see that $H_{n,j}(q;a,b)$ is a rational linear combination of the elements of the set $\{ (\log p)^j\colon p\mid q\}$ (although the rational coefficients depend upon $q$, $a$, and~$b$).
\qedef
\label{d definition}
\end{definition}

Once we determine how to expand the coefficient $|\chi(a)-\chi(b)|^{2n}$ as a linear combination of individual values of $\chi$, we can establish Proposition~\ref {cumulants in elementary terms} which describes how the cumulant $W_n(q;a,b)V(q;a,b)$ can be evaluated in terms of the arithmetic information already defined.

\begin{lemma}
\label{d lemma}
Let $\chi$ be a Dirichlet character$\mod q$, and let $a$ and $b$ be reduced residues\mod q. For any nonnegative integer $n$, we have
$$
|\chi(a)-\chi(b)|^{2n} = \sum_{|i|\le n} (-1)^i \binom{2n}{n+i} \chi(a^ib^{-i}).
$$
\end{lemma}

\begin{proof}
The algebraic identity
\[
\big( 2-t-t^{-1} \big)^n = \sum_{|i|\le n} (-1)^i \binom{2n}{n+i} t^i
\]
can be verified by a straightforward induction on~$n$. Since
\[
|\chi(a)-\chi(b)|^2 = \big( \chi(a)-\chi(b) \big) \overline{\big( \chi(a)-\chi(b) \big)} = 2 - \chi(ab^{-1}) - \chi(ab^{-1})^{-1},
\]
the lemma follows immediately.
\end{proof}

\begin{prop}
\label{cumulants in elementary terms}
Assume GRH. Let $a$ and $b$ be reduced residues\mod q. For any positive integer $n$, the expression $W_n(q;a,b) V(q;a,b) / \phi(q)$ can be written as a fixed $\Q$-linear combination of elements in the set
\begin{multline}
\label{other set def}
\{ \L_n(q)\} \cup \big\{ \iota_q\big( a^ib^{-i} \big)\log2,\, \iota_q\big( {-}a^ib^{-i} \big)\log2,\, \iota_q\big( a^ib^{-i} \big) \gamma_0 \colon |i|\le n \big\} \\
\cup \big\{ \iota_q\big( a^ib^{-i} \big)\zeta(j),\, \iota_q\big( {-}a^ib^{-i} \big)\zeta(j) \colon |i|\le n,\, 2\le j \le n \big\} \\
\cup \big\{ H_{n,j}(q;a,b),\, \M_{n,j}(q;a,b) \colon 1\le j\le n \big\}.
\end{multline}
\end{prop}

\begin{proof}
From the definitions~\eqref{Wkqab def} and~\eqref {bnchi def} of $W_n(q;a,b)$ and $b_n(\chi)$, we have
\begin{align*}
\frac{W_n(q;a,b) V(q;a,b)}{\phi(q)} &=  \frac{2^{2n}|\lambda_{2n}|}{\phi(q)} \allchisum |\chi(a)-\chi(b)|^{2n} \Lsum{\gamma>0} \frac{1} {(1/4+\gamma^2)^n} \\
&=  2^{2n-1}|\lambda_{2n}| \cdot \frac1 {\phi(q)} \allchisum |\chi(a)-\chi(b)|^{2n} b_n(\chi).
\end{align*}
Lemma~\ref {log bessel lemma}(d) tells us that the numbers $\lambda_{2n}$ are rational. Therefore by Lemma~\ref{b(chi) lemma}, it suffices to establish that three types of expressions, corresponding to the three types of quantities in the set~\eqref{set def}, are fixed $\Q$-linear combinations of elements of the set~\eqref{other set def}.

\noindent{\bf Type 1}: $\displaystyle \frac1{\phi(q)} \allchisum |\chi(a)-\chi(b)|^{2n}\log \frac{q^*}\pi$.

Note that Proposition~\ref{log qstar sum prop} can be rewritten in the form
\begin{equation}
\label{log qstar sum prop rewritten}
\frac1{\phi(q)} \allchisum \chi(a)\log q^* = \iota_q(a) \bigg( \log q - \sum_{p\mid q} \frac{\log p}{p-1} \bigg) - \big( 1-\iota_q(a) \big) \frac{\Lambda(q/(q,a-1))}{\phi(q/(q,a-1))}.
\end{equation}
By Lemma~\ref {d lemma} and the orthogonality relation~\eqref {orthogonality.relation}, we have
\begin{equation}
\label{orthog for 2n}
\frac1{\phi(q)} \allchisum |\chi(a)-\chi(b)|^{2n} \chi(c) = \sum_{|i|\le n} (-1)^i \binom{2n}{n+i} \iota_q(a^ib^{-i}c).
\end{equation}
Therefore, using equation~\eqref {log qstar sum prop rewritten} and Proposition~\ref {just orthogonality prop}, we get
\begin{align*}
\frac1{\phi(q)} \allchisum & |\chi(a)-\chi(b)|^{2n} \log \frac{q^*}\pi \\
&= \frac1{\phi(q)} \allchisum |\chi(a)-\chi(b)|^{2n} \log q^* - \frac1{\phi(q)} \allchisum |\chi(a)-\chi(b)|^{2n} \log\pi \\
&= \sum_{|i|\le n} (-1)^i \binom{2n}{n+i} \bigg( \iota_q(a^ib^{-i}) \bigg( \log q- \sum_{p\mid q} \frac{\log p}{p-1} \bigg) \\
&\qquad{}- (1-\iota_q(a^ib^{-i})) \frac{\Lambda(q/(q,a^ib^{-i}-1))}{\phi(q/(q,a^ib^{-i}-1))} - \iota_q(a^ib^{-i}) \log \pi \bigg) = \L_n (q),
\end{align*}
since $(q,a^ib^{-i}-1) = (q,a^i-b^i)$.

\noindent{\bf Type 2}: $\displaystyle \frac1{\phi(q)} \allchisum |\chi(a)-\chi(b)|^{2n} \bigg[ \frac{d^j}{ds^j}\log \Gamma(s) \bigg]_{s=(1+\xi)/2}$ for some $1\le j\le n$.

The following identities hold for $j\ge2$ (see \cite[equations 6.4.2 and 6.4.4]{handbook}):
\begin{align*}
\bigg[ \frac{d^j}{ds^j}\log \Gamma(s) \bigg]_{s=1} &= (-1)^j (j-1)! \zeta(j); \\
\bigg[ \frac{d^j}{ds^j}\log \Gamma(s) \bigg]_{s=1/2} &= (-1)^j (j-1)! \zeta(j)(2^j-1).
\end{align*}
Because $\xi=0$ when $\chi(-1)=1$ and $\xi=1$ when $\chi(-1)=-1$, we may thus write
\[
\bigg[ \frac{d^j}{ds^j}\log \Gamma(s) \bigg]_{s=(1+\xi)/2} = (-1)^j (j-1)! \zeta(j) \big( 2^{j-1} + \chi(-1)(2^{j-1}-1) \big),
\]
whence by equation~\eqref {orthog for 2n},
\begin{align*}
\frac1{\phi(q)} \allchisum & |\chi(a)-\chi(b)|^{2n} \bigg[ \frac{d^j}{ds^j}\log \Gamma(s) \bigg]_{s=(1+\xi)/2} \\
&= \frac1{\phi(q)} \allchisum |\chi(a)-\chi(b)|^{2n} (-1)^j (j-1)! \zeta(j) \big( 2^{j-1} + \chi(-1)(2^{j-1}-1) \big) \\
&= (-1)^j (j-1)! \zeta(j) \bigg( 2^{j-1} \sum_{|i|\le n} (-1)^i \binom{2n}{n+i} \iota_q(a^ib^{-i}) \\
&\qquad{}+( 2^{j-1}-1) \sum_{|i|\le n} (-1)^i \binom{2n}{n+i} \iota_q(-a^ib^{-i}) \bigg),
\end{align*}
which is a linear combination of the desired type. The case $j=1$ can be handled similarly using the identity
\[
\bigg[ \frac{d}{ds}\log \Gamma(s) \bigg]_{s=(1+\xi)/2} = - \gamma_0 - \big( 1 + \chi(-1) \big) \log2.
\]

\noindent{\bf Type 3}: $\displaystyle \frac1{\phi(q)} \allchisum |\chi(a)-\chi(b)|^{2n} \Re\bigg[ \frac{d^j}{ds^j}\log L(s,\chi^*) \bigg]_{s=1}$ for some $1\le j\le n$.

The expression in question is exactly $\M_{n,j}^*(q;a,b)$, and so it suffices to show that $\M_{n,j}^*(q;a,b) = \M_{n,j}(q;a,b) + H_{n,j}(q;a,b)$. Note that the identity
\[
\frac{d^j}{ds^j} \log L(s,\chi) = \frac{d^{j-1}}{ds^{j-1}} \bigg( {-} \sum_{n=1}^\infty \frac{\Lambda(n)\chi(n)}{n^s} \bigg) = (-1)^j \sum_{n=1}^\infty \frac{\Lambda(n)(\log n)^{j-1}\chi(n)}{n^s}
\]
implies
\[
\bigg[ \frac{d^j}{ds^j} \log L(s,\chi) \bigg]_{s=1} = (-1)^j \sum_p (\log p)^j \sum_{e=1}^\infty \frac{e^{j-1}}{p^e} \chi(p^e).
\]
The proof of Lemma~\ref{Mstar to M lemma} can then be adapted to obtain the equation
\begin{align*}
\M_{n,j}^*(q;a,b) - \M_{n,j}(q;a,b) &= \frac{(-1)^j}{\phi(q)} \sum_{p\mid q} (\log p)^j \sum_{e=1}^\infty \frac{e^{j-1}}{p^e} \allchisum \big| \chi(a) - \chi(b) \big|^{2n} \chi^*(p^e) \\
&= \frac{(-1)^j}{\phi(q)} \sum_{p\mid q} (\log p)^j \sum_{e=1}^\infty \frac{e^{j-1}}{p^e} \sum_{|i|\le n} (-1)^i \binom{2n}{n+i} \allchisum \chi(a^ib^{-i}) \chi^*(p^e)
\end{align*}
by Lemma~\ref {d lemma}. Evaluating the inner sum by Proposition~\ref {semi orthogonality relations prop} shows that this last expression is precisely the definition of $H_{n,j}(q;a,b)$, as desired.
\end{proof}

As described at the beginning of this section, Proposition~\ref {other set def} is exactly what we need to justify the assertion that we can calculate $\delta(q;a,b)$, using only arithmetic information, to within an error of the form $O_A(q^{-A})$. That some small primes in arithmetic progressions\mod q enter the calculations is not surprising; interestingly, though, the arithmetic progressions involved are the residue classes $a^jb^{-j}$ for $|j|\le n$, rather than the residue classes $a$ and $b$ themselves!


To give a better flavor of the form these approximations take, we end this section by explicitly giving such a formula with an error term better than $O(q^{-5/2+\ep})$ for any $\ep>0$. Taking $K=1$ in Theorem~\ref {delta series theorem} gives the formula
\[
\delta(q;a,b) = \frac12 + \frac{\rho(q)}{\sqrt{2\pi V(q;a,b)}}\bigg( 1 - \frac {\rho(q)^2}{6V(q;a,b)} - \frac{3W_2(q;a,b)}{V(q;a,b)} \bigg) + O \bigg( \frac{\rho(q)^5}{V(q;a,b)^{5/2}} \bigg).
\label{asymptotic formula K=1 case}
\]
Going through the above proofs, one can laboriously work out that
\begin{multline*}
\frac{W_2(q;a,b)V(q;a,b)}{\phi(q)} = \tfrac14 \L_2(q) \\
- \tfrac1{4\phi(q)} \allchisum \big| \chi(a)-\chi(b) \big|^4 \big\{ \big( \gamma_0 + \log2 + \tfrac12\zeta(2) \big) + \chi(-1) \big( \log2 + \tfrac14\zeta(2) \big) \big\} \\
+ \tfrac12 \big( \M_{2,1}(q;a,b) + H_{2,1}(q;a,b) \big) - \tfrac14
\big( \M_{2,2}(q;a,b) + H_{2,2}(q;a,b) \big),
\end{multline*}
to which Lemma~\ref {d lemma} can be applied with $n=2$. Combining these two expressions and expanding $V(q;a,b)$ as described after equation~\eqref {M explicit formula version} results in the following formula:

\begin{prop}
\label{Daniel's k=2 prop}
Assume GRH and LI. Suppose $a$ and $b$ are reduced residues$\mod q$ such that $a$ is a nonsquare and $b$ is a square$\mod q$. Then
\begin{align*}
\delta(q;a,b) &= \frac{1}{2} +  \frac{\rho(q)}{2\sqrt{ \pi \phi(q) (\tilde\L(q;a,b) + \tilde\Rcal(q;a,b)) }} \bigg(  1- \frac{\rho(q)^2}{12\phi(q) \tilde\L(q;a,b) } \\
&\qquad{}- \frac{3}{16 \phi(q) \tilde\L(q;a,b)^2} \Big\{ \L_2(q) - (6+2\iota_q(a^2b^{-2}))\big( \gamma_0 + \log2 + \tfrac12\zeta(2) \big) \\
&\qquad\qquad\qquad{}- (2\iota_q(-a^2b^{-2})-8\iota_q(-ab^{-1})) \big( \log2 + \tfrac14\zeta(2) \big) \\
&\qquad\qquad\qquad{}+ 2 \F_1(q;a,b) + 2H_{2,1}(q;a,b) - \F_2(q;a,b) - H_{2,2}(q;a,b) \Big\} \bigg) \\
&\qquad{}+ O \bigg( \frac{\rho(q)^5 \sqrt{\log q}}{\phi(q)^{5/2}} \bigg),
\end{align*}
where $\L_2(q)$ is defined in Definition~\ref{Lnq def} and $H_{2,j}(q;a,b)$ is defined in Definition~\ref {d definition}, and
\begin{align*}
\tilde \L(q;a,b) &= \L(q) +K_q(a-b) + \iota_q(-ab^{-1})\log 2 +H_0(q;a,b) + \frac{\Lambda(ab^{-1})}{ab^{-1}} + \frac{\Lambda(ba^{-1})}{ba^{-1}} \\
\tilde\Rcal(q;a,b) &= \sum_{\substack{q\le n\le q^4 \\ n\equiv ab^{-1} \mod q}} \frac{\Lambda(n)}{n} + \sum_{\substack{q\le n\le q^4 \\ n\equiv ba^{-1} \mod q}} \frac{\Lambda(n)}{n} - 2 \sum_{\substack{q\le n\le q^4 \\ n\equiv 1 \mod q}} \frac{\Lambda(n)}{n} \\
\F_1(q;a,b) &= \frac{\Lambda(a^2b^{-2})}{a^2b^{-2}}-4\frac{\Lambda(ab^{-1})}{ab^{-1}}-4\frac{\Lambda(ba^{-1})}{ba^{-1}}+\frac{\Lambda(b^2a^{-2})}{b^2a^{-2}} \\
\F_2(q;a,b) &= \frac{\Lambda(a^2b^{-2})\log(a^2b^{-2})}{a^2b^{-2}}-4\frac{\Lambda(ab^{-1})\log(ab^{-1})}{ab^{-1}}-4\frac{\Lambda(ba^{-1})\log(ba^{-1})}{ba^{-1}}+\frac{\Lambda(b^2a^{-2})\log(b^2a^{-2})}{b^2a^{-2}}.
\end{align*}
In all these definitions, expressions such as $a^2b^{-2}$ refer to the smallest positive integer congruent to $a^2b^{-2}\mod q$.
\end{prop}

\subsection{A central limit theorem}
\label{CLT section}

In this section we prove a central limit theorem for the functions
\[
E(x;q,a)-E(x;q,b) = \phi(q)(\pi(x;q,a)-\pi(x;q,b))x^{-1/2}\log x.
\]
The technique we use is certainly not without precedent. Hooley \cite{hooley} and Rubinstein and Sarnak \cite{RS} both prove central limit theorems for similar normalized error terms under the same hypotheses GRH and LI (though each with different acronyms).

\begin{theorem}
\label{central limit theorem}
Assume GRH and LI. As $q$ tends to infinity, the limiting logarithmic distributions of the functions
\begin{equation}
\frac{E(x;q,a)-E(x;q,b)}{\sqrt{2\phi(q)\log q}}
\label{variance.on.the.bottom}
\end{equation}
converge in measure to the standard normal distribution of mean 0 and variance 1, uniformly for all pairs $a,b$ of distinct reduced residues modulo $q$.
\end{theorem}

We remark that this result can in fact be derived from Rubinstein and Sarnak's 2-dimensional central limit theorem~\cite[Section 3.2]{RS} for $\big( E(x;q,a),E(x;q,b) \big)$, although this implication is not made explicit in their paper. In general, let
$$
\phi_{X,Y}(s,t) = \int_0^{\infty}\int_0^{\infty}\exp\big(i(sx+ty)\big) f_{X,Y}(x,y)\,dx\,dy
$$
denote the joint characteristic function of a pair $(X,Y)$ of real-valued random variables, where $f_{X,Y}(x,y)$ is the joint density function of the pair. Then the characteristic function of the real-valued random variable $X-Y$ is
\begin{align*}
\phi_{X-Y}(t) &= \E\big( \exp(it(X-Y)) \big) \\
&= \int_0^{\infty}\int_0^{\infty}\exp(it(x-y)) f_{X,Y}(x,y)\,dx\,dy = \phi_{X,Y}(t,-t).
\end{align*}
The derivation of Theorem~\ref{central limit theorem} from Rubinstein and Sarnak's 2-dimensional central limit theorem then follows by taking $X$ and $Y$ to be the random variables having the same limiting distributions as $E(x;q,a)$ and $E(x;q,b)$, respectively (which implies that $X-Y = X_{q;a,b}$).

On the other hand, we note that our analysis of the variances of these distributions has the benefit of providing a better quantitative statement of the convergence of our limiting distributions to the Gaussian distribution: see equation \eqref{better.quantitative} below.

\begin{proof}[Proof of Theorem~\ref{central limit theorem}]
Since the Fourier transform of the limiting logarithmic distribution of $E(x;q,a)-E(x;q,b)$ is $\hat X_{q;a,b}(\eta)$, the Fourier transform of the limiting logarithmic distribution of the quotient~\eqref{variance.on.the.bottom} is $\hat X_{q;a,b}(\eta/\sqrt{2\phi(q)\log q})$. A theorem of L\'evy from 1925 \cite[Section 4.2, Theorem 4]{rao}, the Continuity Theorem for characteristic functions, asserts that all we need to show is that
\begin{equation}
\lim_{q\to\infty} \hat X_{q;a,b}\Big( \frac\eta{\sqrt{2\phi(q)\log q}} \Big) = e^{-\eta^2/2}
\label{to.apply.Levy}
\end{equation}
for every fixed real number $\eta$. Because the right-hand side is continuous at $\eta=0$, it is automatically the characteristic function of the measure to which the limiting logarithmic distributions of the quotients~\eqref{variance.on.the.bottom} converge in distribution, according to L\'evy's theorem.

When $q$ is large enough in terms of $\eta$, we have $|\eta/\sqrt{2\phi(q)\log q}| \le \frac14$. For such $q$, Proposition~\ref{used in central limit theorem prop} implies that
\begin{align}
\log \hat X_{q;a,b}\Big( \frac\eta{\sqrt{2\phi(q)\log q}} \Big) &=  - \frac{V(q;a,b)}{2\phi(q)\log q} \frac{\eta^2}2 + O\Big( \frac{(c(q,a)-c(q,b))|\eta|}{\sqrt{\phi(q)\log q}} + \frac{V(q;a,b)\eta^4}{(\phi(q)\log q)^2} \Big) \notag \\
&= -\frac{\eta^2}2 + O\Big( \frac{\eta^2\log\log q}{\log q} + \frac{|\eta|\rho(q)}{\sqrt{\phi(q)\log q}} + \frac{\eta^4}{\phi(q)\log q} \Big)
\label{better.quantitative}
\end{align}
using the asymptotic formula for $V(q;a,b)$ given in Proposition~\ref{Vqab asymptotic prop}. Since $\eta$ is fixed, this is enough to verify \eqref{to.apply.Levy}, which establishes the theorem.
\end{proof}

\subsection{Racing quadratic nonresidues against quadratic residues}
\label{quadratic race section}

This section is devoted to understanding the effect of low-lying zeros of Dirichlet $L$-functions on prime number races between quadratic residues and quadratic nonresidues. This phenomenon has already been studied by many authors---see for instance \cite{low-lying zeros}. Let $q$ be an odd prime, and define $\pi(x;q,N)=\#\{ p\le x\colon p$ is a quadratic nonresidue\mod q\} and $\pi(x;q,R)=\#\{ p\le x\colon p$ is a quadratic residue\mod q\}. Each of $\pi(x;q,N)$ and $\pi(x;q,R)$ is asymptotic to $\pi(x)/2$, but Chebyshev's bias predicts that the difference $\pi(x;q,N)-\pi(x;q,R)$, or equivalently the normalized difference
$$
E(x;N,R)=\frac{\log x}{\sqrt{x}} \big( \pi(x;q,N)-\pi(x;q,R) \big),
$$
is more often positive than negative.

Our methods lead to an asymptotic formula for $\delta(q;N,R)$, the logarithmic density of the set of real numbers $x\ge1$ satisfying $\pi(x;q,N)>\pi(x;q,R)$, that explains the effect of low-lying zeros in a straightfoward and quantitative way. We sketch this application now.

First, define the random variable
$$
X_{q;N,R} = 2+2\sum_{\substack{\gamma>0\\L(1/2+i\gamma,\chi_1)=0}} \frac{X_{\gamma}}{\sqrt{\frac14+\gamma^2}},
$$
where $\chi_1$ is the unique quadratic character\mod q. Under GRH and LI, the distribution of $X_{q;N,R}$ is the same as the limiting distribution of the normalized error term $E(x;N,R)$. The methods of Section~\ref {variance analysis section} then lead to an asymptotic formula analogous to equation~\eqref{asymptotic formula K=0 case}:
\begin{equation}
\delta(q;N,R) = \frac12 + \sqrt{\frac{2}{\pi V(q;N,R)}} + O \bigg( \frac{1}{V(q;N,R)^{3/2}} \bigg),
\label{delta N R}
\end{equation}
where
$$
V(q;N,R)=b(\chi_1)=\sum_{\substack{\gamma\in\R \\ L(1/2+i\gamma,\chi_1)=0}} \frac1{\frac 14 + \gamma^2}.
$$
To simplify the discussion, we explore only the effect of the lowest zero (the zero closest to the real axis) on the size of $V(q;N,R)$.

By the classical formula for the zero-counting function $N(T,\chi)$, the average height of the lowest zero of $L(s,\chi_1)$ is $2\pi/\log q$. Suppose we have a lower-than-average zero, say at height $c\cdot 2\pi/\log q$ for some $0<c<1$. Then we get a higher-than-average contribution to the variance of size
\[
\frac1{1/4+(c\cdot 2\pi/\log q)^2} - \frac1{1/4+(2\pi/\log q)^2}.
\]
Since the variance $V(q;N,R) = b(\chi_1)$ is asymptotically $\log q$ by Lemma~\ref{b.chi.nonprimitive.lemma}, this increases the variance by roughly a percentage $t$ given by
\begin{equation}
t \sim \frac1{\log q} \bigg( \frac1{1/4+(c\cdot 2\pi/\log q)^2} - \frac1{1/4+(2\pi/\log q)^2} \bigg).
\label{three way}
\end{equation}
Therefore, given any two of the three parameters
\begin{itemize}
\item how low the lowest zero is (in terms of the percentage $c$ of the average),
\item how large a contribution we see to the variance (in terms of the percentage $t$), and
\item the size of the modulus $q$,
\end{itemize}
we can determine the range for the third parameter from equation~\eqref{three way}.

For example, as $c$ tends to 0, the right-hand side of equation~\eqref{three way} is asymptotically
\[
\frac{64\pi^2}{(\log^2q + 16\pi^2)\log q}.
\]
So if we want to see an increase in variance of 10\%, an approximation for the range of $q$ for which this might be possible is given by setting $64\pi^2/(\log^2q + 16\pi^2)\log q = 0.1$ and solving for $q$, which gives $\log q = 15.66$ or about $q = {}$6,300,000. This assumes that $c$ tends to~0---in other words, that $L(s,\chi_1)$ has an extremely low zero. However, even taking $c=\frac13$ on the right-hand side of equation~\eqref{three way} and setting the resulting expression equal to 0.1 yields about $q= {}$1,600,000. In other words, having a zero that's only a third as high as the average zero, for example, will give a ``noticeable'' (at least 10\%) lift to the variance up to roughly $q= {}$1,600,000.

It turns out that unusually low zeros of this sort are not particularly rare. The Katz-Sarnak model predicts that the proportion of $L$-functions in the family $\{L(s,\chi) \colon \chi \text{ primitive of order 2} \}$ having a zero as low as $c\cdot 2\pi/\log q$ is asymptotically ${2\pi^2 c^3/9}$ as $c$ tends to~0. Continuing with our example value $c=\frac13$, we see that roughly 8\% of the moduli less than 1,600,000 will have a 10\% lift in the variance $V(q;N,R)$ coming from the lowest-lying zero.

Well-known examples of $L$-functions having low-lying zeros are the $L(s,\chi_1)$ corresponding to prime moduli $q$ for which the class number $h(-q)$ equals~1, as explained in \cite{low-lying zeros} with the Chowla--Selberg formula for $q=163$. For this modulus, the imaginary part of the lowest-lying zero is $0.202901\ldots = 0.16449\ldots \cdot2\pi/\log 163$. According to our approximations, this low-lying zero increases the variance by roughly $t=56\%$; considering this increased variance in equation~\eqref {delta N R} explains why the value of $\delta(163;N,R)$ is exceptionally low. The actual value of $\delta(163;N,R)$, along with some neighboring values, are shown in Table~\ref{delta for q about 163}.

\begin{table}[b]
 \caption{Values of $\delta(q;N,R)$ for $q=163$ and nearby primes}
\begin{tabular}{|c|c|}
 \hline
$q$ & $\delta(q;N,R)$ \\
\hline
151 & 0.745487 \\
157 & 0.750767 \\
163 & 0.590585 \\
167 & 0.780096 \\
173 & 0.659642 \\
\hline
\end{tabular}
\label{delta for q about 163}
\end{table}

Other Dirichlet $L$-functions having low-lying zeros are the $L(s,\chi_1)$ corresponding to prime moduli $q$ for which the class number $h(-q)$ is relatively small; a good summary of the first few class numbers is given in \cite[Table~VI]{low-lying zeros}.


Notice that in principle, racing quadratic residues against quadratic nonresidues makes sense for any modulus $q$ for which $\rho(q)=2$, which includes powers of odd primes and twice these powers. However, being a quadratic residue modulo a prime $q$ is exactly equivalent to being a quadratic residue modulo any power of $q$, and also (for odd numbers) exactly equivalent to being a quadratic residue modulo twice a power of $q$. Therefore $\delta(q;N,R) = \delta(q^k;N,R) = \delta(2q^k;N,R)$ for every odd prime~$q$. The only other modulus for which $\rho(q)=2$ is $q=4$, which has been previously studied: Rubinstein and Sarnak~\cite{RS} calculated that $\delta(4;N,R) = \delta(4;3,1) \approx 0.9959$.

\section{Fine-scale differences among races to the same modulus}
\label{fine-scale section}

In this section we probe the effect that the specific choice of residue classes $a$ and $b$ has on the density $\delta(q;a,b)$. We begin by proving Corollary~\ref{isolate ab contribution cor}, which isolates the quantitative influence of $\delta(q;a,b)$ on $a$ and $b$ from its dependence on $q$, in Section~\ref {impact section}. We then dissect the relevant influence, namely the function $\Delta(q;a,b)$, showing how particular arithmetic properties of the residue classes $a$ and $b$ predictably affect the density; three tables of computational data are included to illustrate these conclusions. In Section~\ref {predictability section} we develop this theme even further, proving Theorem~\ref {Delta in terms of rating theorem} and hence its implication Theorem~\ref{fix a and b theorem}, which establishes a lasting ``meta-bias'' among these densities. Finally, in Section~\ref {mirror section} we apply our techniques to the seemingly unrelated ``mirror image phenomenon'' observed by Bays and Hudson, explaining its existence with a similar analysis.

\subsection{The impact of the residue classes $a$ and $b$}
\label{impact section}

The work of the previous sections has provided us with all the tools we need to establish Corollary~\ref{isolate ab contribution cor}.

\begin{proof}[Proof of Corollary~\ref{isolate ab contribution cor}]
We begin by showing that the function
\begin{equation*}
\Delta(q;a,b) = K_q(a-b) + \iota_q(-ab^{-1})\log2 + \frac{\Lambda(r_1)}{r_1} +\frac{\Lambda(r_2)}{r_2} + H(q;a,b)
\end{equation*}
defined in equation~\eqref{Delta def} is bounded above by an absolute constant (the fact that it is nonnegative is immediate from the definitions of its constituent parts). It has already been remarked in Definition~\ref {iota and Lq and Rqn def} that $K_q$ is uniformly bounded, as is $\iota_q$. We also have $\Lambda(r)/r \le (\log r)/r$, and this function is decreasing for $r\ge3$, so the third and fourth terms are each uniformly bounded as well. Finally, from Definition~\ref {hqpr def}, we see that
\[
h(q;p,r) = \frac{1}{\phi(p^{\nu})}\frac{\log p}{p^{e(q;p,r)}} \le \frac1{p-1} \frac{\log p}{p^1},
\]
and so $H(q;a,b) < \sum_p 2(\log p)/p(p-1)$ is uniformly bounded by a convergent sum as well.

We now turn to the main assertion of the corollary. By Theorems~\ref {variance evaluation theorem} and~\ref {M evaluation theorem}, we have
\begin{align*}
V(q;a,b) &= 2\phi(q) \big( \L(q) + K_q(a-b) + \iota_q(-ab^{-1})\log2 \big) + 2M^*(q;a,b) \\
&= 2\phi(q) \bigg( \L(q) + K_q(a-b) + \iota_q(-ab^{-1})\log2 + \frac{\Lambda(r_1)}{r_1} +\frac{\Lambda(r_2)}{r_2} + H(q;a,b) + O \bigg( \frac{\log^2 q}{q} \bigg) \bigg) \\
&= 2\phi(q) \bigg( \L(q) + \Delta(q;a,b) + O \bigg( \frac{\log^2 q}{q} \bigg) \bigg) = 2\phi(q)\L(q) \bigg( 1 + \frac{\Delta(q;a,b)}{\L(q)} + O \bigg( \frac{\log q}{q} \bigg) \bigg).
\end{align*}
Since $\Delta(q;a,b)$ is bounded while $\L(q) \sim \log q$, we see that $V(q) \sim 2\phi(q)\log q$; moreover, the power series expansion of $(1+t)^{-1/2}$ around $t=0$ implies that
\begin{align*}
V(q;a,b)^{-1/2} &= \big( 2\phi(q)\L(q) \big)^{-1/2} \bigg( 1 - \frac{\Delta(q;a,b)}{2\L(q)} + O \bigg( \frac{\Delta(q;a,b)^2}{\L(q)^2} + \frac{\log q}q \bigg) \bigg) \\
&= \big( 2\phi(q)\L(q) \big)^{-1/2} \bigg( 1 - \frac{\Delta(q;a,b)}{2\L(q)} + O \bigg( \frac1{\log^2 q} \bigg) \bigg).
\end{align*}
(Recall that we are assuming that $q\ge43$, which is enough to ensure that $\L(q)$ is positive.) Together with the last assertion of Theorem~\ref {delta series theorem}, this formula implies that
\[
\delta(q;a,b) = \frac12 + \frac{\rho(q)}{2\sqrt{\pi\phi(q)\L(q)}} \bigg( 1 - \frac{\Delta(q;a,b)}{2\L(q)} + O \bigg( \frac1{\log^2 q} \bigg) \bigg) + O \bigg( \frac{\rho(q)^3}{V(q;a,b)^{3/2}} \bigg).
\]
Since the last error term is $\ll_\ep q^\ep/(\phi(q)\log q)^{3/2}$, it can be subsumed into the first error term, and the proof of the corollary is complete.
\end{proof}

Corollary~\ref{isolate ab contribution cor} tells us that larger values of $\Delta(q;a,b)$ lead to smaller values of the density $\delta(q;a,b)$. Computations of the values of $\delta(q;a,b)$ (using methods described in Section~\ref {explicit computation of the delta}) illustrate this relationship nicely. Since $\delta(q;a,b) = \delta(q;ab^{-1},1)$ when $b$ is a square\mod q, we restrict our attention to densities of the form $\delta(q;a,1)$.

\begin{table}[bt]
\caption{The densities $\delta(q;a,1)$ computed for $q=163$}
\smaller\smaller
\begin{tabular}{|c|c|c|c||c|c|c|c|}
\hline
$q$ & $a$ & $a^{-1}$ & $\delta(q;a,1)$ & $q$ & $a$ & $a^{-1}$ & $\delta(q;a,1)$ \\
\hline
163 &   162 &   162 &   0.524032    &   163 &   30  &   125 &   0.526809    \\
163 &   3   &   109 &   0.525168    &   163 &   76  &   148 &   0.526815    \\
163 &   2   &   82  &   0.525370    &   163 &   92  &   101 &   0.526829    \\
163 &   5   &   98  &   0.525428    &   163 &   86  &   127 &   0.526869    \\
163 &   7   &   70  &   0.525664    &   163 &   128 &   149 &   0.526879    \\
163 &   11  &   89  &   0.525744    &   163 &   129 &   139 &   0.526879    \\
163 &   13  &   138 &   0.526079    &   163 &   80  &   108 &   0.526894    \\
163 &   17  &   48  &   0.526083    &   163 &   114 &   153 &   0.526898    \\
163 &   19  &   103 &   0.526090    &   163 &   117 &   124 &   0.526900    \\
163 &   23  &   78  &   0.526213    &   163 &   20  &   106 &   0.526906    \\
163 &   31  &   142 &   0.526378    &   163 &   42  &   66  &   0.526912    \\
163 &   67  &   73  &   0.526437    &   163 &   28  &   99  &   0.526914    \\
163 &   37  &   141 &   0.526510    &   163 &   44  &   63  &   0.526925    \\
163 &   29  &   45  &   0.526532    &   163 &   12  &   68  &   0.526931    \\
163 &   27  &   157 &   0.526578    &   163 &   72  &   120 &   0.526941    \\
163 &   32  &   107 &   0.526586    &   163 &   112 &   147 &   0.526975    \\
163 &   59  &   105 &   0.526620    &   163 &   110 &   123 &   0.526981    \\
163 &   8   &   102 &   0.526638    &   163 &   122 &   159 &   0.526996    \\
163 &   79  &   130 &   0.526682    &   163 &   50  &   75  &   0.526997    \\
163 &   94  &   137 &   0.526746    &   163 &   52  &   116 &   0.527002    \\
163 &   18  &   154 &   0.526768    &   &&&\\

\hline
\end{tabular}
\label{163 data}
\end{table}

We begin by investigating a prime modulus $q$, noting that
\[
\Delta(q;a,1) = \iota_q(-a)\log2 + \frac{\Lambda(a)}{a} +\frac{\Lambda(a^{-1})}{a^{-1}} + \frac{2\log q}{q(q-1)}
\]
when $q$ is prime (here $a^{-1}$ denotes the smallest positive integer that is a multiplicative inverse of $a\mod q$). Therefore we obtain the largest value of $\Delta(q;a,b)$ when $a\equiv-1\mod q$, and the next largest values are when $a$ is a small prime, so that the $\Lambda(a)/a$ term is large. (These next large values also occur when $a^{-1}$ is a small prime, and in fact we already know that $\delta(q;a,1) = \delta(q;a^{-1},1)$. When $q$ is large, it is impossible for both $a$ and $a^{-1}$ to be small.) Notice that $\Lambda(a)/a$ is generally decreasing on primes $a$, except that $\Lambda(3)/3 > \Lambda(2)/2$. Therefore the second, third, and fourth-largest values of $\Delta(q;a,1)$ will occur for $a$ congruent to 3, 2, and 5\mod q, respectively.

This effect is quite visible in the calculated data. We use the prime modulus $q=163$ as an example, since the smallest 12 primes, as well as $-1$, are all nonsquares\mod{163}. Table~\ref {163 data} lists the values of all densities of the form $\delta(163,a,1)$ (remembering that $\delta(q;a,1) = \delta(q;a^{-1},1)$ and that the value of any $\delta(q;a,b)$ is equal to one of these). Even though the relationship between $\Delta(q;a,1)$ and $\delta(q;a,1)$ given in Corollary~\ref{isolate ab contribution cor} involves an error term, the data is striking. The smallest ten values of $\delta(q;a,1)$ are exactly in the order predicted by our analysis of $\Delta(q;a,1)$: the smallest is $a=162\equiv-1\mod{163}$, then $a=3$ and $a=2$, then the seven next smallest primes in order. (This ordering, which is clearly related to Theorem~\ref {fix a and b theorem}, will be seen again in Figure~\ref {littleshapes}.)

One can also probe more closely the effect of the term $M(q;a,1)$ upon the density $\delta(q;a,1)$. Equation~\eqref {M explicit formula version} can be rewritten as the approximation
\begin{equation}
\frac{M(q;a,1)}{\phi(q)} + 2\sum_{\substack{n\le y \\ n\equiv1\mod q}} \frac{\Lambda(n)}n \approx \sum_{\substack{n\le y \\ n\equiv a\mod q}} \frac{\Lambda(n)}n + \sum_{\substack{n\le y \\ n\equiv a^{-1}\mod q}} \frac{\Lambda(n)}n
\label{M tilde}
\end{equation}
(where we are ignoring the exact form of the error term). Taking $y=q$ recovers the approximation $M(q;a,1) \approx \phi(q) \big( \Lambda(a)/a + \Lambda(a^{-1})/a^{-1} \big)$ used in the definition of $\Delta(q;a,b)$, but taking $y$ larger would result in a better approximation.

\begin{table}[bt]
\caption{The effect of medium-sized prime powers on the densities $\delta(q;a,1)$, illustrated with $q=101$ and $y=10^6$}
\smaller\smaller
\begin{tabular}{|c|c|c|cccc|c|c|}
\hline
$q$ & $a$ &  $a^{-1^{\mathstrut}}$    &  \multicolumn{4}{c|}{First four prime powers} & RHS of \eqref{M tilde} & $\delta(101,a,1)$    \\
\hline
101 & 7   &   29  &   7   &   29  &   433 &   512 &   0.563304    &   0.534839    \\
101 & 2   &   51  &   2   &   103 &   709 &   859 &   0.554043    &   0.534928    \\
101 & 3   &   34  &   3   &   337 &   811 &   1013    &   0.528385    &   0.535103    \\
101 & 11  &   46  &   11  &   349 &   617 &   1021    &   0.383090 &   0.536123    \\
101 & 8   &   38  &   8   &   109 &   139 &   311 &   0.332888    &   0.536499    \\
101 & 53  &   61  &   53  &   61  &   263 &   457 &   0.329038    &   0.536522    \\
101 & 12  &   59  &   59  &   113 &   463 &   719 &   0.276048    &   0.536955    \\
101 & 67  &   98  &   67  &   199 &   269 &   401 &   0.271567    &   0.536993    \\
101 & 41  &   69  &   41  &   243 &   271 &   647 &   0.268766    &   0.537013    \\
101 & 28  &   83  &   83  &   331 &   487 &   937 &   0.235130 &   0.537284    \\
101 & 15  &   27  &   27  &   128 &   229 &   419 &   0.235035    &   0.537293    \\
101 & 66  &   75  &   167 &   277 &   479 &   571 &   0.230291    &   0.537340    \\
101 & 18  &   73  &   73  &   523 &   881 &   1129    &   0.215281    &   0.537463    \\
101 & 50  &   99  &   151 &   353 &   503 &   757 &   0.211209    &   0.537500    \\
101 & 55  &   90  &   191 &   257 &   661 &   797 &   0.205833    &   0.537537    \\
101 & 42  &   89  &   89  &   547 &   1153    &   1301    &   0.202289    &   0.537586    \\
101 & 44  &   62  &   163 &   347 &   751 &   769 &   0.199652    &   0.537607    \\
101 & 72  &   94  &   173 &   397 &   577 &   599 &   0.196417    &   0.537623    \\
101 & 32  &   60  &   32  &   739 &   941 &   1171    &   0.191447    &   0.537660    \\
101 & 26  &   35  &   127 &   439 &   641 &   733 &   0.190601    &   0.537688    \\
101 & 39  &   57  &   241 &   443 &   461 &   1049    &   0.187848    &   0.537708    \\
101 & 40  &   48  &   149 &   343 &   1151    &   1361    &   0.178698    &   0.537780    \\
101 & 10  &   91  &   293 &   313 &   919 &   1303    &   0.180422    &   0.537792    \\
101 & 74  &   86  &   389 &   983 &   1399    &   1601    &   0.165153    &   0.537900    \\
101 & 63  &   93  &   467 &   1103    &   1709    &   2083    &   0.146466    &   0.538067    \\

\hline
\end{tabular}
\label{101 data}
\end{table}

We examine this effect on the calculated densities for the
medium-sized prime modulus $q=101$. In Table~\ref {101 data}, the
second group of columns records the first four prime powers that are
congruent to $a$ or $a^{-1}\mod{101}$. The second-to-last column
gives the value of the right-hand side of equation~\eqref{M tilde},
computed at $y=10^6$. Note that smaller prime powers in the second
group of columns give large contributions to this second-to-last
column, a trend that can be visually confirmed. Finally, the last
column lists the values of the densities $\delta(q;a,b)$, according
to which the rows have been sorted in ascending order. The
correlation between larger values of the second-to-last column and
smaller values of $\delta(q;a,b)$ is almost perfect (the
adjacent entries $a=40$ and $a=10$ being the only exception): the existence of smaller primes and
prime powers in the residue classes $a$ and $a^{-1}\mod{101}$ really
does contribute positively to the variance $V(q;a,1)$ and hence
decreases the density $\delta(q;a,1)$. (Note that the effect of the
term $\iota_{101}(-a)\log 2$ is not present here, since 101 is a
prime congruent to 1\mod4 and hence $-1$ is not a nonsquare.)

\begin{table}[bt]
\caption{The densities $\delta(q;a,1)$ computed for $q=420$, together with the values of $K_{q}(a-1) = \Lambda(420/(420,a-1))/\phi(420/(420,a-1))$}
\smaller\smaller
\begin{tabular}{|c|c|c|c|c|c||c|c|c|c|c|c|}
\hline
$q$ & $a$ & $a^{-1}$ & $(q,a-1)$ & $K_q(a-1)$ & $\delta(q;a,1)$ & $q$ & $a$ & $a^{-1}$ & $(q,a-1)$ & $K_q(a-1)$ & $\delta(q;a,1)$ \\
\hline
 420 & 211 & 211 & 210 & $\log 2$ & 0.770742 & 420 & 113 & 197 & 28 & 0 & 0.807031 \\
 420 & 419 & 419 & 2 & 0 & 0.772085 & 420 & 149 & 389 & 4 & 0 & 0.807209 \\
 420 & 281 & 281 & 140 & $(\log 3)/2$ & 0.779470 & 420 & 103 & 367 & 6 & 0 & 0.807284 \\
 420 & 253 & 337 & 84 & $(\log 5)/4$ & 0.788271 & 420 & 223 & 307 & 6 & 0 & 0.807302 \\
 420 & 61 & 241 & 60 & $(\log 7)/6$ & 0.788920 & 420 & 83 & 167 & 2 & 0 & 0.807505 \\
 420 & 181 & 181 & 60 & $(\log 7)/6$ & 0.789192 & 420 & 151 & 331 & 30 & 0 & 0.809031 \\
 420 & 17 & 173 & 4 & 0 & 0.795603 & 420 & 59 & 299 & 2 & 0 & 0.809639 \\
 420 & 47 & 143 & 2 & 0 & 0.796173 & 420 & 137 & 233 & 4 & 0 & 0.809647 \\
 420 & 29 & 29 & 28 & 0 & 0.796943 & 420 & 139 & 139 & 6 & 0 & 0.810290 \\
 420 & 13 & 97 & 12 & 0 & 0.797669 & 420 & 73 & 397 & 12 & 0 & 0.811004 \\
 420 & 187 & 283 & 6 & 0 & 0.797855 & 420 & 157 & 313 & 12 & 0 & 0.811197 \\
 420 & 53 & 317 & 4 & 0 & 0.798207 & 420 & 251 & 251 & 10 & 0 & 0.811557 \\
 420 & 11 & 191 & 10 & 0 & 0.798316 & 420 & 349 & 349 & 12 & 0 & 0.811706 \\
 420 & 107 & 263 & 2 & 0 & 0.798691 & 420 & 323 & 407 & 14 & 0 & 0.811752 \\
 420 & 41 & 41 & 20 & 0 & 0.800067 & 420 & 179 & 359 & 2 & 0 & 0.811765 \\
 420 & 19 & 199 & 6 & 0 & 0.800937 & 420 & 229 & 409 & 12 & 0 & 0.811776 \\
 420 & 43 & 127 & 42 & 0 & 0.801609 & 420 & 131 & 311 & 10 & 0 & 0.811913 \\
 420 & 23 & 347 & 2 & 0 & 0.802681 & 420 & 277 & 373 & 12 & 0 & 0.812052 \\
 420 & 37 & 193 & 12 & 0 & 0.803757 & 420 & 239 & 239 & 14 & 0 & 0.812215 \\
 420 & 79 & 319 & 6 & 0 & 0.804798 & 420 & 247 & 403 & 6 & 0 & 0.812215 \\
 420 & 89 & 269 & 4 & 0 & 0.804836 & 420 & 227 & 383 & 2 & 0 & 0.812777 \\
 420 & 101 & 341 & 20 & 0 & 0.805089 & 420 & 221 & 401 & 20 & 0 & 0.813594 \\
 420 & 71 & 71 & 70 & 0 & 0.805123 & 420 & 293 & 377 & 4 & 0 & 0.813793 \\
 420 & 67 & 163 & 6 & 0 & 0.805196 & 420 & 379 & 379 & 42 & 0 & 0.813818 \\
 420 & 31 & 271 & 30 & 0 & 0.806076 & 420 & 209 & 209 & 4 & 0 & 0.815037 \\
 420 & 257 & 353 & 4 & 0 & 0.806638 & 420 & 391 & 391 & 30 & 0 & 0.815604 \\
 \hline
\end{tabular}
\label{420 data}
\end{table}

Finally we investigate a highly composite modulus $q$ to witness the effect of the term $K_q(a-1) = \Lambda(q/(q,a-1))/\phi(q/(q,a-1)) -{\Lambda(q)}/{\phi(q)}$ on the size of $\Delta(q;a,1)$. This expression vanishes unless $a-1$ has such a large factor in common with $q$ that the quotient $q/(q,a-1)$ is a prime power. Therefore we see a larger value of $\Delta(q;a,1)$, and hence expect to see a smaller value of $\delta(q;a,1)$, when $q/(q,a-1)$ is a small prime, for example when $a=\frac q2+1$.

Table~\ref {420 data} confirms this observation with the modulus $q=420$. Of the six smallest densities $\delta(420;a,1)$, five of them correspond to the residue classes $a$ (and their inverses) for which $q/(q,a-1)$ is a prime power; the sixth corresponds to $a\equiv-1\mod{420}$, echoing the effect already seen for $q=163$. Moreover, the ordering of these first six densities are exactly as predicted: even the battle for smallest density between $a\equiv-1\mod{420}$ and $a=420/2-1$ is appropriate, since both residue classes cause an increase in $\Delta(420;a,1)$ of size exactly $\log 2$. (Since 420 is divisible by the four smallest primes, the largest effect that the $\Lambda(a)/a$ term could have on $\Delta(q;a,b)$ is $(\log 11)/11$, and so these effects are not nearly as large.) The magnitude of this effect is quite significant: note that the difference between the first and seventh-smallest values of $\delta(420;a,1)$ (from $a=211$ to $a=17$) is larger than the spread of the largest 46 values (from $a=17$ to $a=391$).

\subsection{The predictability of the relative sizes of densities}
\label{predictability section}

The specificity of our asymptotic formulas to this point suggests comparing, for fixed integers $a_1$ and $a_2$, the densities $\delta(q;a_1,1)$ and $\delta(q;a_2,1)$ as $q$ runs through all moduli for which both $a_1$ and $a_2$ are nonsquares. (We have already seen that every density is equal to one of the form $\delta(q;a,1)$.) Theorem~\ref {fix a and b theorem}, which we will derive shortly from Corollary~\ref{partial order cor},  is a statement about exactly this sort of comparison.

In fact we can investigate even more general families of race games: fix rwo rational numbers $r$ and $s$, and consider the family of densities $\delta(q;r+sq,1)$ as $q$ varies.  We need $r+sq$ to be an integer and relatively prime to $q$ for this density to be sensible; we further desire $r+sq$ to be a nonsquare\mod q, or else $\delta(q;r+sq,1)$ simply equals~$\frac12$. Therefore, we define the set of qualified moduli
\[
  Q(r,s) = \{ q\in\N\colon r+sq\in\Z,\, (r+sq,q)=1;\, \text{there are no solutions to } x^2\equiv r+sq\mod q\}.
\]
(Note that translating $s$ by an integer does not change the residue class of $r+sq\mod q$, so one could restrict $s$ to the interval $[0,1)$ without losing generality if desired.)

It turns out that every pair $(r,s)$ of rational numbers can be assigned a ``rating'' $R(r,s)$ that dictates how the densities in the family $\delta(q;r+sq,1)$ compare to other densities in similar families.

\begin{definition}
\label{rating definition}
Define a rating function $R(r,s)$ as follows:
\begin{itemize}
  \item Suppose that the denominator of $s$ is a prime power $p^k$ ($k\ge1$).
  \begin{itemize}
    \item If $r$ is a power $p^j$ of the same prime, then $R(r,s) = (\log p)/\phi(p^{j+k})$.
    \item If $r=1$ or $r=1/p^j$ for some integer $1\le j<k$, then $R(r,s) = (\log p)/\phi(p^k)$.
    \item If $r=1/p^k$, then $R(r,s) = (\log p)/p^k$.
    \item Otherwise $R(r,s)=0$.
  \end{itemize}
  \item Suppose that $s$ is an integer.
  \begin{itemize}
    \item If $r=-1$, then $R(r,s) = \log2$.
    \item If $r$ is a prime power $p^j$ ($j\ge1$), then $R(r,s) = (\log p)/p^j$.
    \item Otherwise $R(r,s)=0$.
  \end{itemize}
  \item $R(r,s)=0$ for all other values of $s$.
\qedef
\end{itemize}
\end{definition}

\begin{theorem}
\label{Delta in terms of rating theorem}
Let $\Delta(q;a,b)$ be defined as in equation~\eqref{Delta def}. For fixed rational numbers $r$ and $s$,
\[
\Delta(q;r+sq,1) = R(r,s) + O_{r,s}\bigg( \frac{\log q}q \bigg)
\]
as $q$ tends to infinity within the set $Q(r,s)$.
\end{theorem}

We will be able to prove this theorem at the end of the section; first, however, we note an interesting corollary.

\begin{cor}
\label{partial order cor}
Assume GRH and LI. If $r_1,s_1,r_2,s_2$ are rational numbers such that $R(r_1,s_1) > R(r_2,s_2)$, then
  \[
    \delta(q;r_1+s_1q,1) < \delta(q;r_2+s_2q,1) \text{ for all but finitely many } q\in Q(r_1,s_1) \cap Q(r_2,s_2).
  \]
\end{cor}

\begin{proof}
We may assume that $q\ge43$. Inserting the conclusion of Theorem~\ref {Delta in terms of rating theorem} into the formula for $\delta(q;a,b)$ in Corollary~\ref{isolate ab contribution cor}, we obtain
\begin{equation}
\label{yields strange normalization}
\delta(q;r+sq,1) = \frac12 + \frac{\rho(q)}{2\sqrt{\pi\phi(q)\L(q)}} \bigg( 1 - \frac{R(r,s)}{2\L(q)} + O\bigg( \frac1{\log^2q} \bigg) \bigg)
\end{equation}
for any $q\in Q(r,s)$. Therefore for all $q\in Q(r_1,s_1) \cap Q(r_2,s_2)$,
\[
\delta(q;r_1+s_1q,1) - \delta(q;r_2+s_2q,1) = \bigg( \frac{-R(r_1,s_1)+R(r_2,s_2)}{2\L(q)} + O\bigg( \frac1{\log^2q} \bigg) \bigg) \frac{\rho(q)}{2\sqrt{\pi\phi(q)\L(q)}}.
\]
Since the constant $-R(r_1,s_1)+R(r_2,s_2)$ is negative by hypothesis, we see that $\delta(q;r_1+s_1q,1) - \delta(q;r_2+s_2q,1)$ is negative when $q$ is sufficiently large in terms of $r_1$, $s_1$, $r_2$, and $s_2$.
\end{proof}

Notice, from the part of Definition~\ref {rating definition} where $s$ is an integer, that Theorem~\ref {fix a and b theorem} is precisely the special case of Corollary~\ref {partial order cor} where $s_1=s_2=0$. Therefore we have reduced Theorem~\ref {fix a and b theorem} to proving Theorem~\ref {Delta in terms of rating theorem}.

\begin{figure}[b]
\caption{Normalized densities $\delta(q;a,1)$ for primes $q$, using the normalization~\eqref {referred to in figure caption}}
\includegraphics[width=6.5in]{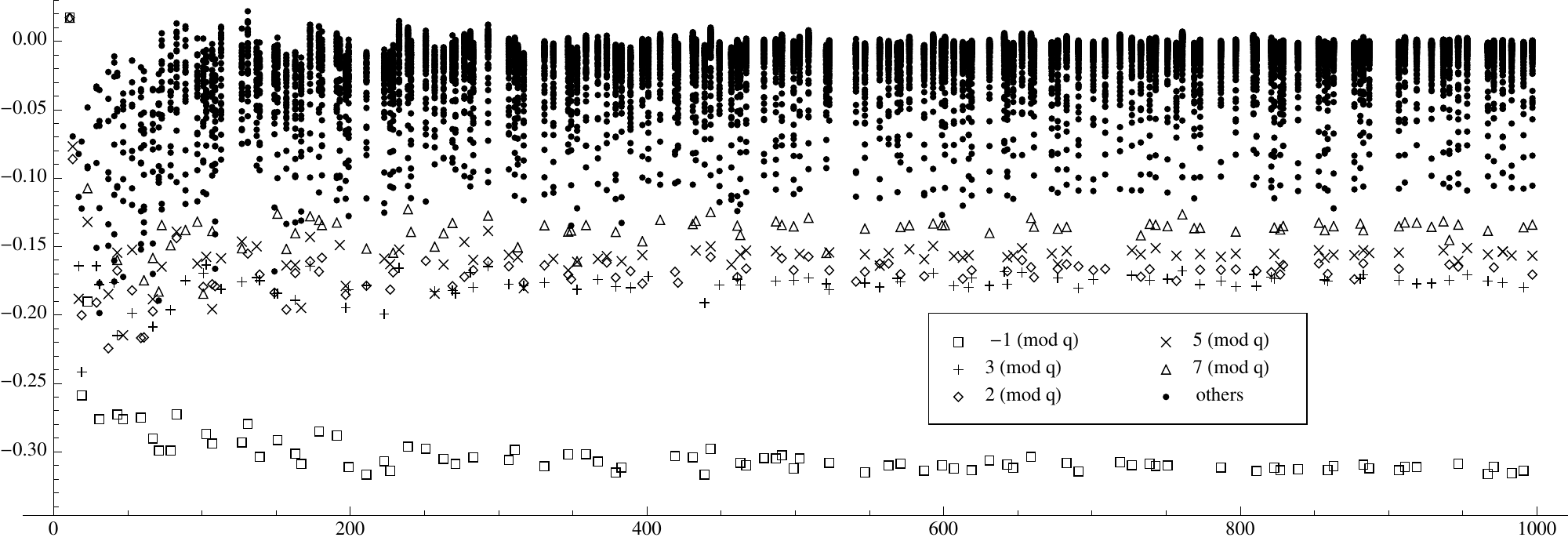}
\label{littleshapes}
\end{figure}

Theorem~\ref {fix a and b theorem} itself is illustrated in Figure~\ref {littleshapes}, using the computed densities for prime moduli to most clearly observe the relevant phenomenon. For each prime $q$ up to 1000, and for every nonsquare $a\mod q$, the point
\begin{multline}
\label{referred to in figure caption}
\bigg( q, \frac{2\sqrt{\pi\phi(q)\L(q)^3}}{\rho(q)} \big( \delta(q;a,1) - \tfrac12 \big) - {\L(q)} \bigg) \\
= \bigg( q, \sqrt{\pi(q-1)} \bigg( \log \frac q{2\pi e^{\gamma_0}} \bigg)^{3/2} \big( \delta(q;a,1) - \tfrac12 \big) - \log \frac q{2\pi e^{\gamma_0}} \bigg)
\end{multline}
has been plotted; the values corresponding to certain residue classes have been emphasized with the listed symbols. The motivation for the seemingly strange (though order-preserving) normalization in the second coordinate is equation~\eqref{yields strange normalization}, which shows that the value in the second coordinate is $-R(a,0)/2 + O(1/\log q)$. In other words, on the vertical axis the value 0 corresponds to $\delta(q;a,b)$ being exactly the ``default'' value $\frac12 + \rho(q)/2\sqrt{\pi\phi(q)\L(q)}$, the value $-0.05$ corresponds to $\delta(q;a,b)$ being less than the default value by $0.05 \rho(q)/2\sqrt{\pi\phi(q)\L(q)^3}$, and so on. We clearly see in Figure~\ref {littleshapes} the normalized values corresponding to $\delta(q;-1,1)$, $\delta(q;3,1)$, $\delta(q;2,1)$, and so on sorting themselves out into rows converging on the values $-\frac12\log2$, $-\frac16\log3$, $-\frac14\log2$, and so on.

We need to establish several lemmas before we can prove Theorem~\ref {Delta in terms of rating theorem}. The recurring theme in the following analysis is that solutions to linear congruences\mod q with fixed coefficients must be at least a constant times $q$ in size, save for specific exceptions that can be catalogued.

\begin{lemma}
\label{iota lemma}
Let $r$ and $s$ be rational numbers. If $r\ne-1$ or if $s$ is not an integer, then there are only finitely many positive integers $q$ such that $r+sq$ is an integer and $r+sq\equiv-1\mod q$.
\end{lemma}

\begin{proof}
Write $r=\frac ab$ and $s=\frac cd$. The congruence $\frac ab+\frac cdq \equiv-1\mod q$ implies that  $\frac{ad}b+cq\equiv-d\mod q$, which means that $q$ must divide $\frac{ad}b+d$. This only happens for finitely many $q$ unless $\frac{ad}b+d=0$, which is equivalent (since $d\ne0$) to $\frac ab=-1$. In this case the congruence is $\frac cdq\equiv 0\mod q$, which can happen only if $\frac cd$ is an integer.
\end{proof}

\begin{lemma}
\label{K bound lemma}
Let $r$ and $s$ be rational numbers. Suppose that $q$ is a positive
integer such that $r+sq$ is an integer. If $r\ne1$, then
$K_q(r+sq-1) \ll_{r,s} (\log q)/q$.
\end{lemma}

\begin{proof}
We first note that $\Lambda(t)/\phi(t) \ll (\log t)/t$ for all
positive integers $t$: if $\Lambda(t)$ is nonzero, then $t$ is a
prime power, which means $\phi(t) \ge t/2$. Therefore it suffices to
show that $(q,r+sq-1)$ is bounded, since then $q/(q,r+sq-1)
\gg_{r,s} q$ and consequently $K_q(r+sq-1) \ll_{r,s} (\log q)/q$
since $(\log t)/t$ is decreasing for $t\ge3$. But writing $r=\frac
ab$ and $s=\frac cd$, we have
\[
\big(q,\tfrac ab+\tfrac cdq-1\big) \mid (q,d(a-b)+bcq) = (q,d(a-b)) \mid d(a-b).
\]
Since $r\ne 1$, we see that $d(a-b)$ is nonzero, and hence $(q,r+sq-1) \le d|a-b| \ll_{r,s}1$ as required.
\end{proof}

\begin{lemma}
Let $r$ and $s$ be rational numbers. Assume that $r$ is not a positive integer or $s$ is not an integer. If $q$ and $y$ are positive integers such that $r+sq$ is an integer and $y\equiv r+sq\mod q$, then $y\gg_{r,s} q$.
\label{I'm big lemma}
\end{lemma}

\begin{proof}
Suppose first that $s$ is not an integer, and write $s=c/d$ where $d>1$. Then $s$ is at least $1/d$ away from the nearest integer, so that $sq$ is at least $q/d$ away from the nearest multiple of $q$. Since $y = r+sq-mq$ for some integer $m$, we have $y \ge |sq - mq| - |r| \ge q/d - |r| \gg_{r,s} q$ when $q$ is sufficiently large in terms of $r$ and~$s$.

On the other hand, if $s$ is an integer, then $r$ must also be an integer. If $r$ is nonpositive, then the least integer $y$ congruent to $r\mod q$ is $q-|r| \gg_r q$ when $q$ is sufficiently large in terms of~$r$.
\end{proof}

\begin{lemma}
Let $r$ and $s$ be rational numbers. Assume that either $r$ is not the reciprocal of a positive integer or that $\frac sr$ is not an integer. Suppose that positive integers $q$ and $y$ are given such that $r+sq$ is an integer and $(r+sq)y \equiv 1\mod q$. Then $y\gg_{r,s} q$.
\label{inverse is big lemma}
\end{lemma}

\begin{proof}
Write $r=\frac ab$ and $s=\frac cd$ with $(a,b)=(c,d)=1$ and $b,d>0$. We may assume that $q>2d^2$, for if $q\le 2d^2$ then $y \ge 1 \ge \frac q{2d^2} \gg_{s} q$. Note that $a\ne0$, since $0+\frac cdq = c\frac qd$ cannot be invertible modulo $q$ when $q>d$. The assumption that $r+sq$ is an integer implies that $d(r+sq) = \frac{ad}b + cq$ is also an integer; since $(a,b)=1$, this implies that $b\mid d$. Therefore we may write $d=b\delta$ for some integer $\delta$. Similarly, it must be true that $b(r+sq) = a + \frac{cq}\delta$ is an integer; since $(c,\delta) \mid (c,d) = 1$, this implies that $q$ is a multiple of $\delta$.

\medskip\noindent
{\em Case 1:}
Suppose first that $\delta=1$. If $a=1$, then $r=\frac1b$ would be the reciprocal of a positive integer and $\frac sr = \frac{c/b}{1/b}$ would be an integer, contrary to assumption; therefore $a\ne1$. The condition $(r+sq)y\equiv 1\mod q$, when multiplied by $b$, becomes $ay\equiv b\mod q$. Now if $a=-1$, then the congruence in question is equivalent to $y\equiv-b\mod q$; since $b>0$, this implies that $y\ge q-b\gg_r q$ as desired. Therefore for the rest of Case 1, we can assume that $|a|>1$.

Since any common factor of $a$ and $q$ would consequently be a factor of $b$ as well, but $(a,b)=1$, we must have $(a,q)=1$. Thus we may choose $u$ such that $uq\equiv-1\mod a$, so that $y_0=b(uq+1)/a$ is an integer. We see by direct calculation that $y_0$ is a solution to $ay\equiv b\mod q$, and all other solutions differ from this one by a multiple of $q/(b,q)$, which is certainly a multiple of $\frac qb$. In other words, $y=q(\frac{bu}a + \frac zb) + \frac ba$ for some integer $z$. If $\frac{bu}a + \frac zb=0$ then $-z = b(\frac{bu}a + \frac zb)-z = \frac{b^2u}a$ would be an integer, but this is impossible since both $b$ and $u$ are relatively prime to $a$ (here we use $|a|\ne1$).  Therefore $\big| \frac{bu}a + \frac zb\big| \ge \frac1{|a|b}$, and so $y \ge \frac q{|a|b} - \frac b{|a|}$; since $q>2d^2=2b^2$, this gives $y \ge \frac q{2|a|b} \gg_{r,s} q$.

\medskip\noindent
{\em Case 2:}
Suppose now that $\delta>1$. The condition $(\frac ab+\frac cdq)y \equiv 1\mod q$ forces $(y,q)=1$ and so $(y,\delta)=1$ as well. Multiplying the condition by $b$ yields $ay+cy\frac q\delta \equiv b\mod q$, which we write as $\frac{cyq}\delta-qm = b -ay$ for some integer $m$. But notice that $(cy,\delta)=1$, so that $\frac{cy}\delta$ is at least $\frac 1\delta$ away from every integer (here we use $\delta>1$); therefore $\frac{cyq}\delta$ is at least $\frac q\delta$ away from the nearest multiple of $q$. Therefore $\frac q\delta \le\big |\frac{cyq}\delta-qm\big| = |b-ay| \le b + |a|y$, and hence $y \ge (q-b\delta)/|a|\delta$; since $q>2d=2b\delta$, this gives $y \ge \frac q{2|a|\delta} \gg_{r,s} q$.
\end{proof}

\begin{cor}
\label{Lambda terms big cor}
Let $r$ and $s$ be rational numbers, and let $q$ be a positive integer such that $r+sq$ is an integer.
\begin{enumerate}
\item Assume that $r$ is not a positive integer or $s$ is not an integer. Suppose that $y$ is a positive integer such that $y \equiv r+sq\mod q$. Then $\Lambda(y)/y \ll_{r,s} (\log q)/q$.
\item Assume that either $r$ is not the reciprocal of a positive integer or that $\frac sr$ is not an integer. Suppose that $y$ is a positive integer such that $(r+sq)y \equiv 1\mod q$. Then $\Lambda(y)/y \ll_{r,s} (\log q)/q$.
\end{enumerate}
\end{cor}

\begin{proof}
Since $\Lambda(y)/y \le (\log y)/y$, which is a decreasing function for $y\ge3$, this follows from Lemmas~\ref{I'm big lemma} and~\ref{inverse is big lemma}.
\end{proof}

\begin{lemma}
Let $r$ and $s$ be rational numbers. Let $q$ be a positive integer such that $r+sq$ is an integer, and let $p$ be a prime such that $p^\nu \parallel q$ with $\nu\ge1$.
\renewcommand{\theenumi}{\alph{enumi}}
\begin{enumerate}
  \item Suppose that $e$ is a positive integer such that $p^e \equiv r+sq \mod{q/p^\nu}$. Then either $p^e = r$ or $p^e \gg_{r,s} q/p^\nu$.
  \item Suppose that $e$ is a positive integer such that $p^e(r+sq) \equiv 1 \mod{q/p^\nu}$. Then either $p^e = 1/r$ or $p^e \gg_{r,s} q/p^\nu$.
\end{enumerate}
\label{h term}
\end{lemma}

\noindent
Notice that if $p^e = r$ in (a) then $sp^\nu$ is an integer; also, if $p^e=1/r$ in (b) then $sp^{e+\nu}$ is an integer. In both cases, it is necessary that the denominator of $s$ be a power of $p$ as well.

\begin{proof}
We may assume that $q/p^\nu$ is sufficiently large in terms of $r$ and $s$, for otherwise any positive integer is $\gg_{r,s} q/p^\nu$. We have two cases to examine.
\renewcommand{\theenumi}{\alph{enumi}}
\begin{enumerate}
  \item We are assuming that $p^e \equiv r+sq \mod{q/p^\nu}$. Suppose first that $sp^\nu$ is an integer. Then $sq$ is an integer multiple of $q/p^\nu$, and so $p^e \equiv r \mod{q/p^\nu}$. This means that either $p^e = r$ or $p^e \ge q/p^\nu + r \gg_r q/p^\nu$, since $q/p^\nu$ is sufficiently large in terms of $r$.

  On the other hand, suppose that $sp^\nu$ is not an integer. Then
  \[
    p^e \equiv r+sq = r + (sp^\nu) q/p^\nu \equiv r + (sp^\nu - \lfloor sp^\nu \rfloor) q/p^\nu \mod{q/p^\nu}.
  \]
  If the denominator of $s$ is $d$, then the difference $sp^\nu - \lfloor sp^\nu \rfloor$ is at least $\frac1d$, and therefore $p^e \ge q/dp^\nu + r \gg_{r,s} q/p^\nu$ as well, since $q/p^\nu$ is sufficiently large in terms of $r$ and $s$.
  \item We are assuming that $p^e(r+sq) \equiv 1 \mod{q/p^\nu}$. We apply Lemma~\ref {inverse is big lemma} with $q/p^\nu$ in place of $q$ and with $y=p^e$, which yields the desired lower bound $p^e \gg_{r,s} q/p^\nu$ unless $r$ is the reciprocal of a positive integer and $\frac sr$ is an integer. In this case, multiplying the assumed congruence by the integer $1/r$ gives $p^e(1+\frac sr q) \equiv 1/r \mod{q/p^\nu}$, which implies $p^e \equiv 1/r \mod{q/p^\nu}$ since $\frac sr$ is an integer. Therefore, since $q/p^\nu$ is sufficiently large in terms of $r$, either $p^e = 1/r$ or $p^e \ge q/p^\nu + 1/r > q/p^\nu$.
\end{enumerate}
\end{proof}

The next two lemmas involve the functions $h(q;p,r)$ and $H(q;a,b)$ that were defined in Definition~\ref {hqpr def}. Since we are dealing with rational numbers, we make the following clarification: when we say ``power of $p$'', we mean $p^k$ for some {\em positive} integer $k$ (so $p^2$ and $p^1$ are powers of $p$, but neither 1 nor $p^{-1}$ is).

\begin{lemma}
\label{small h lemma}
Let $r$ and $s$ be rational numbers, and suppose that $q$ is a positive integer such that $r+sq$ is an integer that is relatively prime to~$q$. Let $p$ be a prime dividing~$q$, and choose $\nu\ge1$ such that $p^\nu\parallel q$.
\renewcommand{\theenumi}{\alph{enumi}}
\begin{enumerate}
  \item If both $r$ and the denominator of $s$ are powers of $p$ (note that if the denominator of $s$ equals $p^k$, these conditions imply $\nu=k$), then
  \[
    h(q;p,(r+sq)^{-1}) = \frac{\log p}{r\phi(p^\nu)} + O_{r,s}\bigg( \frac{\log p}q \bigg);
  \]
  otherwise $h(q;p,(r+sq)^{-1}) \ll_{r,s} (\log p)/q$.
  \item If both $1/r$ and the denominator of $s$ are powers of $p$ (note that if $r=1/p^j$ and the denominator of $s$ equals $p^k$, these conditions imply $\nu = k-j$), then
  \[
    h(q;p,r+sq) = \frac{r\log p}{\phi(p^\nu)} + O_{r,s}\bigg( \frac{\log p}q \bigg),
  \]
  otherwise $h(q;p,r+sq) \ll_{r,s} (\log p)/q$.
\end{enumerate}
\end{lemma}

\begin{proof}
\renewcommand{\theenumi}{\alph{enumi}}
\begin{enumerate}
\item Assume $p^e\equiv r+sq \mod{q/p^\nu} $. By Lemma~\ref {h term}, we have that either $r=p^e$ (which implies that the denominator of $s$ is a power of $p$), or else $h(q;p,(r+sq)^{-1}) \ll_{r,s} (\log p)/q$. So we only need to compute $h(q;p,(r+sq)^{-1})$ in the case where $r$ is any power of $p$ (say $r=p^e$) and where $s$ has a denominator which is a power of $p$ (say $s=c/p^z$, where $z\le \nu$ since $q$ is a multiple of the denominator of $s$).

In this case the congruence $p^e\equiv r+sq \mod{q/p^\nu}$ is satisfied. Furthermore, $e$ is the minimal such positive integer if $q$ is sufficiently large in terms of~$r$ and~$s$. If $e$ is minimal we have $h(q;p,(r+sq)^{-1}) = (\log p)/\phi(p^\nu) p^{e} = (\log p)/r\phi(p^\nu)$ by definition; if $e$ is not minimal we have $h(q;p,(r+sq)^{-1}) \ll_{r,s} (\log p)/q$ since there are only finitely many possible values of~$q$. In both cases, the proposition is established (the ``main term'' $(\log p)/r\phi(p^\nu)$ is actually dominated by the error term in the latter case).

\item Assume $p^e(r+sq)\equiv 1 \mod{q/p^\nu} $. By Lemma~\ref {h term}, we have that either $1/r=p^e$ (which implies that the denominator of $s$ is a power of $p$), or else $h(q;p,r+sq) \ll_{r,s} (\log p)/q$. So we only need to compute $h(q;p,r+sq)$ in the case where $1/r$ is any power of $p$ (say $1/r=p^j$) and where $s$ has a denominator which is a power of $p$ (say $s=c/p^k$, where $k-j= \nu>0$).

In this case the congruence $p^j (r+sq) \equiv 1 \mod{q/p^\nu} $ is satisfied (since $p^j sq \equiv 0\mod{q/p^\nu}$). We can rewrite this congruence as $p^j  \equiv 1/r \mod{q/p^\nu}$.
As above, either $j$ is the minimal such positive integer, in which case $h(q;p,r+sq) = (\log p)/\phi(p^\nu) p^{e} = (r\log p)/\phi(p^\nu)$ by definition, or else $q$ is bounded in terms of $r$ and~$s$, in which case $h(q;p,r+sq) \ll_{r,s} (\log p)/q$. In both cases, the proposition is established.
\end{enumerate}
\end{proof}

\begin{cor}
\label{big h cor}
Let $r$ and $s$ be rational numbers, and suppose that $q$ is a positive integer such that $r+sq$ is an integer that is relatively prime to~$q$.
\renewcommand{\theenumi}{\alph{enumi}}
\begin{enumerate}
  \item Suppose both $r$ and the denominator of $s$ are powers of the same prime $p$. Then
  \[
    H(q; r+sq, 1) = \frac{\log p}{\phi(p^{j+k})} + O_{r,s}\bigg( \frac{\log q}q \bigg),
  \]
  where $r=p^j$ and the denominator of $s$ is $p^k$.
  \item Suppose both $1/r$ and the denominator of $s$ are powers of the same prime $p$, with $1/r < s$. Then
  \[
    H(q; r+sq, 1) = \frac{\log p}{\phi(p^k)} + O_{r,s}\bigg( \frac{\log q}q \bigg),
  \]
  where the denominator of $s$ is $p^k$.
 \item If neither of the above sets of conditions holds, then $H(q; r+sq, 1) \ll_{r,s} (\log q)/q$.
\end{enumerate}
\end{cor}

\begin{proof}
We sum the conclusion of Lemma~\ref{small h lemma} over all prime divisors $p$ of~$q$ (and, according to Definition~\ref {hqpr def}, over both residue classes $r+sq$ and $(r+sq)^{-1}$ for each prime divisor). For each such $p$ there is a contribution of $O_{r,s} \big( (\log p)/q \big)$ from error terms, and the sum of all these terms is $\ll_{r,s} \frac1q \sum_{p\mid q} \log p \le (\log q)/q$. The only remaining task is to consider the possible main terms.

If $r=p^j$ and the denominator $p^k$ of $s$ are powers of the same prime $p$, then this prime $p$ must divide any $q$ for which $r+sq$ is an integer; hence by Lemma~\ref{small h lemma}, we have $p^k\parallel q$ and the term $h(q;p,(r+sq)^{-1})$ contributes $(\log p)/r\phi(p^\nu) = (\log p)/\phi(p^{j+k})$ to $H(q; r+sq, 1)$. Similarly, if $r=1/p^j$ and the denominator $p^k$ of $s$ are powers of the same prime $p$ with $j<k$, then this prime $p$ must divide any $q$ for which $r+sq$ is an integer (this would be false if $j=k$); hence by Lemma~\ref{small h lemma}, we have $p^{k-j} \parallel q$ and so the term $h(q;p,r+sq)$ contributes $r(\log p)/\phi(p^\nu) = (\log p)/\phi(p^k)$ to $H(q; r+sq, 1)$. For other pairs $(r,s)$, no main term appears, and so the corollary is established.
\end{proof}

\begin{proof}[Proof of Theorem~\ref{Delta in terms of rating theorem}]
From the definition~\eqref{Delta def} of $\Delta(q;a,b)$, we have
\[
\Delta(q;r+sq,1) = \iota_q(-(r+sq))\log2 + K_q(r+sq-1) + \frac{\Lambda(r_1)}{r_1} + \frac{\Lambda(r_2)}{r_2} + H(q;r+sq,1),
\]
where $r_1$ and $r_2$ are the least positive integers congruent to $r+sq$ and $(r+sq)^{-1}$, respectively, modulo~$q$. The results in this section allow us to analyze each term individually:
\begin{itemize}
  \item If $r=-1$ and $s$ is an integer, then $\iota_q(-(r+sq))\log2 = \log2$. Otherwise, $\iota_q(-(r+sq))\log2=0$ for all but finitely many (depending on $r$ and $s$) integers $q$ by Lemma~\ref{iota lemma}, whence in particular $\iota_q(-(r+sq))\log2 \ll_{r,s} (\log q)/q$.
  \item If $r=1$, then $(r+sq-1,q) = (sq,q) = q/d$ where $d$ is the denominator of $s$, and so $K_q(r+sq-1) = \Lambda(d)/\phi(d)$ by Definition~\ref {iota and Lq and Rqn def}. Otherwise, $K_q(r+sq-1) \ll_{r,s} (\log q)/q$ by Lemma~\ref {K bound lemma}; this bound also holds if the denominator $d$ of $s$ is not a prime power, since then $\Lambda(d)/\phi(d)=0$.
  \item If $r$ is a positive integer and $s$ is an integer, then $r_1=r$ for all but finitely many $q$, in which case $\Lambda(r_1)/r_1 = \Lambda(r)/r$. Otherwise $\Lambda(r_1)/r_1 \ll_{r,s} (\log q)/q$ by Corollary~\ref{Lambda terms big cor}; this bound also holds if $r$ is not a prime power, since then $\Lambda(r)/r=0$.

  Similarly, if $r=1/b$ is the reciprocal of a positive integer and $\frac sr = bs$ is an integer, then $b(r+sq) = 1+(bs)q \equiv1\mod q$; moreover, $b$ will be the smallest positive integer (for all but finitely many $q$) such that $b(r+sq)\equiv1\mod q$, and so $\Lambda(r_2)/r_2 = \Lambda(b)/b$. Otherwise $\Lambda(r_2)/r_2 \ll_{r,s} (\log q)/q$ by Corollary~\ref{Lambda terms big cor}; this bound also holds if the reciprocal $b$ of $r$ is not a prime power, since then $\Lambda(b)/b=0$. Note also that if $b$ is a prime power, then the denominator of $s$ must be the same prime power, since $bs$ and $r+sq$ are both integers.
  \item Corollary~\ref{big h cor} tells us exactly when we have a contribution from $H(q;r+sq,1)$ other than the error term $O_{r,s}\big( (\log q)/q \big)$: the denominator of $s$ must be a prime power, and $r$ must be either a power of the same prime or else the reciprocal of a smaller power of the same prime.
\end{itemize}
In summary, there are six situations in which there is a contribution to $\Delta(q;r+sq,1)$ beyond the error term $O_{r,s}\big( (\log q)/q \big)$: four situations when the denominator of $s$ is a prime power and two situations when $s$ is an integer. All six situations are disjoint, and the contribution to $\Delta(q;r+sq,1)$ in each situation is exactly $R(r,s)$ as defined in Definition~\ref{rating definition}. This establishes the theorem.
\end{proof}

\subsection{The Bays--Hudson ``mirror image phenomenon''}
\label{mirror section}

In 1983, Bays and Hudson~\cite{cyclic} published their observations of some curious phenomena in the prime number race among the reduced residue classes modulo~11. They graphed normalized error terms corresponding to $\pi(x;11,1)$, \dots, $\pi(x;11,10)$, much like the functions $E(x;11,a)$ discussed in this paper, and from the graph they saw that the terms corresponding to the nonsquare residue classes tended to be positive, while the terms corresponding to the square residue classes tended to be negative, as Chebyshev's bias predicts. Unexpectedly, however, they noticed~\cite[Figure 1]{cyclic}) that the graph corresponding to $\pi(x;11,1)$ had a tendency to look like a mirror image of the graph corresponding to $\pi(x;11,10)$, and similarly for the other pairs $\pi(x;11,a)$ and $\pi(x;11,11-a)$. They deemed this observation the ``additive inverse phenomenon''; we use the physically suggestive name ``mirror image phenomenon''.

This prompted them to graph the various normalized error terms corresponding to the sums $\pi(x;11,a) + \pi(x;11,b)$ where $a$ is a nonsquare\mod{11} and $b$ is a square\mod{11}; all such normalized sums have the same mean value. They witnessed a noticeable difference between the cases $a+b=11$, when the graph corresponding to the sum was typically quite close to the average value (as in \cite[Figure 2]{cyclic}), and all other cases which tended to result in more spread-out graphs.

The ideas of the current paper can be used to explain this phenomenon. We consider more generally the limiting logarithmic distributions of the sums of error terms $E(x;q,a)+E(x;q,b)$, where $a$ is a nonsquare\mod{q} and $b$ is a square\mod{q}. The methods of Section~\ref {RV section} are easily modified to show (under the usual assumptions of GRH and LI) that this distribution has variance
\begin{equation}
V^+(q;a,b) = \mostchisum |\chi(a)+\chi(b)|^2 b(\chi).
\label{Vplusdef}
\end{equation}
Following the method of proof of Theorem \ref {variance evaluation theorem}, one can show that for any modulus $q$ and any pair $a,b$ of reduced residues modulo $q$, we have
\begin{multline}
V^+(q;a,b) = 2\phi(q) \bigg( \log q - \sum_{p\mid q} \frac{\log p}{p-1} - \frac{\Lambda(q)}{\phi(q)} - (\gamma_0+\log2\pi) \\
{}- K_q(a-b) - \iota_q(-ab^{-1})\log2 \bigg) + 2M^+(q;a,b) - 4b(\chi_0),
\label{Vplus.final.formula}
\end{multline}
where
\begin{equation*}
M^+(q;a,b) = \mostchisum |\chi(a)+\chi(b)|^2 \frac{L'(1,\chi^*)}{L(1,\chi^*)}.
\end{equation*}
In particular, we note the term $-\iota_q(-ab^{-1})\log 2$; many of the other terms vanish or simplify in the special case that $q$ is prime. We also note that the primary contribution to $M^+(q;a,b)$ is the expression $-\Lambda(r_1)/r_1 -\Lambda(r_2)/r_2$, where $r_1$ and $r_2$ are the least positive residues of $ab^{-1}$ and $ba^{-1}\mod q$. Both of these expressions are familiar to us from our analysis of $V(q;a,b)$, although their signs are negative in the current setting rather than positive as before.


We see that the variance $V^+(q;a,b)$ of this distribution $E(x;q,a)+E(x;q,b)$ is somewhat smaller than the typical size if there is a small prime congruent to $ab^{-1}$ or $ba^{-1}\mod q$; more importantly, it is smallest of all if $-ab^{-1}\equiv1\mod q$, which is precisely the situation $a+b=q$. In other words, we see very explicitly that the cases where $a+b=q$ yield distributions with smaller-than-normal variance, as observed for $q=11$ by Bays and Hudson. In particular, our theory predicts that for any prime $q\equiv3\mod4$ (so that exactly one of $a$ and $-a$ is a square), the graphs of $E(x;q,a)$ and $E(x;q,q-a)$ will tend to resemble mirror images of each other, more so than the graphs of two functions $E(x;q,a)$ and $E(x;q,b)$ where $a$ and $b$ are unrelated. On the other hand, the contribution of the $\iota_q$ term is in a secondary main term, and so the theory predicts that this mirror-image tendency becomes weaker as $q$ grows larger.

We can use the numerical data in the case $q=11$, computed first by Bays and Hudson, to confirm our theoretical evaluation of these variances. We computed the values of each of the twenty-five functions $E(x;11,a)+E(x;11,b)$, where $a$ is a square and $b$ a nonsquare\mod q, on 400 logarithmically equally spaced points spanning the interval $[10^3,10^7]$. We then computed the variance of our sample points for each function, in order to compare them with the theoretical variance given in equation~\eqref {Vplus.final.formula}, which we computed numerically. It is evident from equation~\eqref {Vplusdef} that multiplying both $a$ and $b$ by the same factor does not change $V^+(q;a,b)$, and therefore there are only three distinct values for these theoretical variances: the functions $E(x;q,a)+E(x;q,b)$ where $a+b=11$ all give the same variance, as do the functions where $ab^{-1}\equiv 2$ or $ab^{-1}\equiv2^{-1}\equiv6\mod{11}$, and the functions where $ab^{-1}\equiv 7$ or $ab^{-1}\equiv7^{-1}\equiv8\mod{11}$. Table~\ref{11 var table} summarizes our calculations, where the middle column reports the mean of the variances calculated for the functions in each set.

\begin{table}[bt]
\caption{Observed and theoretical variances for $E(x;11,a)+E(x;11,b)$}
\begin{tabular}{|c|c|c|}
\hline
Set of functions & Average variance calculated & Theoretical \\
$E(x;11,a)+E(x;11,b)$ & from sampled data & variance \\
\hline
$a+b=11$ & 5.60 & 5.31 \\
$\{ab^{-1},ba^{-1}\} \equiv \{2,6\}\mod{11}$ & 7.10 & 6.82 \\
$\{ab^{-1},ba^{-1}\} \equiv \{7,8\}\mod{11}$ & 9.59 & 9.06 \\
\hline
\end{tabular}
\label{11 var table}
\end{table}

Looking directly at the definition~\eqref{Vplusdef} of $V^+(q;a,b)$, we see that when $a\equiv-b\mod q$, the only characters that contribute to the sum are the even characters, since we have $\chi(a)+\chi(b) = \chi(a) + \chi(-1)\chi(a) = 0$ when $\chi(-1)=-1$. As seen earlier in Lemma~\ref {b.chi.nonprimitive.lemma}, the quantity $b(\chi)$ is smaller for even characters than for odd characters, which is another way to express the explanation of the Bays--Hudson observations.

\section{Explicit bounds and computations}
\label{computations section}

We concern ourselves with explicit numerical bounds and computations of the densities $\delta(q;a,b)$ in this final section. We begin in Section~\ref {classical section} by establishing auxiliary bounds for $\Gamma(z)$, for $\frac{L'}L(s,\chi)$, and for the number of zeros of $L(s,\chi)$ near a given height. In Section~\ref {variance bound section} we use these explicit inequalities to provide the proofs of two propositions stated in Section~\ref {M evaluation section}; we also establish computationally accessible upper and lower bounds for the variance $V(q;a,b)$. Explicit estimates for the density $\delta(q;a,b)$ are proved in Section~\ref {density bounds section}, including two theorems that give explicit numerical upper bounds for $\delta(q;a,b)$ for $q$ above 1000. Finally, in Section~\ref {explicit computation of the delta} we describe the two methods we used to calculate numerical values for $\delta(q;a,b)$; we include some sample data from these calculations, including the 120 largest density values that ever occur.

\subsection{Bounds for classical functions}
\label{classical section}

The main goals of this section are to bound the number of zeros of $L(s,\chi)$ near a particular height and to estimate the size of $\frac{L'}L(s,\chi)$ inside the critical strip, both with explicit constants. To achieve this, we first establish some explicit inequalities for the Euler Gamma-function.

\begin{prop}
\label{gamma bounds prop}
If $\Re z\ge\frac18$, then
$$
\bigg|\log \Gamma (z) - \bigg(z-\frac{1}{2}\bigg) \log z +z -\frac{1}{2} \log 2\pi \bigg| \le \frac{1}{4|z|}
$$
and
$$
\bigg|\frac{\Gamma'(z)}{\Gamma(z)} - \log(z+1) + \frac{1}{2z+2} +\frac 1z \bigg| < 0.2.
$$
\end{prop}

\begin{proof}
The first inequality follows from \cite[equations (1) and (9) of Section 1.3]{luke}, both taken with $n=1$. As for the second inequality, we begin with the identity~\cite[equation (21)]{spouge}, taken with $a=1$:
\[
\Psi(z+1) = \log(z+1) - \frac1{2(z+1)} + f'_1(z).
\]
Here $\Psi(z) = \frac{\Gamma'}\Gamma(z)$ has its usual meaning; we use the identity $\frac{\Gamma'(z)}{\Gamma(z)}+\frac 1z = \frac{\Gamma'(z+1)}{\Gamma(z+1)}$ to obtain
\[
\frac{\Gamma'(z)}{\Gamma(z)}+\frac 1z - \log(z+1) + \frac1{2(z+1)} = f'_1(z),
\]
and therefore it suffices to show that $|f'_1(z)| \le 0.2$ when $\Re(z) \ge \frac18$. The notation $f_1(z) = \log F_{1,1/2}(z)$ is defined in~\cite[equation (9)]{spouge}, and therefore $f'_1(z) = F'_{1,1/2}(z)/F_{1,1/2}(z)$. By~\cite[Lemma 1.1.1]{spouge}, the denominator $F_{1,1/2}(z)$ is bounded below in modulus by $\sqrt{e/\pi}$; by~\cite[Lemma 2.2.1]{spouge} taken with $a=n=1$, the numerator is bounded above in modulus by
\[
\big| F'_{1,1/2}(z) \big| < \log \frac{x+1}{x+1/2} - \frac1{2x+2},
\]
where $x=\Re z$ (unfortunately~\cite[equation (27)]{spouge} contains the misprint $f^{(n)}_{a,1/2}$ where $F^{(n)}_{a,1/2}$ is intended). The right-hand side of this inequality is a decreasing function of $x$, and its value at $x=\frac18$ is $\log\frac95-\frac49$. We conclude that for $\Re z\ge\frac18$, we have $|f'_1(z)| \le \big( \log\frac95-\frac49 \big) / \sqrt{e/\pi} < 0.2$, as needed. 
\end{proof}

\begin{lemma}
\label{Gamma difference extracted lemma}
Let $a=0$ or $a=1$. For any real numbers $\frac14\le \sigma\le 1$ and $T$, we have
\begin{equation}
\bigg| \frac{\Gamma'(\frac12(\sigma+iT+a))}{\Gamma(\frac12(\sigma+iT+a))} - \frac{\Gamma'(\frac12(2+iT+a))}{\Gamma(\frac12(2+iT+a))} \bigg| < 7.812.
\label{gamma prime over gamma difference}
\end{equation}
\end{lemma}

\begin{proof}
By symmetry we may assume that $T\ge0$. We first dispose of the case $T\le3$. When $a=0$, a computer calculation shows that the maximum value of the left-hand side of equation~\eqref{gamma prime over gamma difference} in the rectangle $\{ \sigma+iT\colon \frac14\le \sigma\le 1,\, 0\le T\le3\}$ occurs at $\sigma=\frac14$ and $T=0$: the value of the left-hand side at that point is a bit less than 7.812. When $a=1$, a similar calculation shows that the left-hand side of equation~\eqref{gamma prime over gamma difference} is always strictly less than 7.812.

For the rest of the proof, we may therefore assume that $T\ge3$. By Proposition~\ref {gamma bounds prop},
\begin{multline*}
\frac{\Gamma'(\frac12(\sigma+iT+a))}{\Gamma(\frac12(\sigma+iT+a))} - \frac{\Gamma'(\frac12(2+iT+a))}{\Gamma(\frac12(2+iT+a))} = \log\frac{\sigma+iT+a+2}2 - \frac1 {\sigma+iT+a+2}-\frac2 {\sigma+iT+a} \\
- \log\frac{4+iT+a}2 + \frac1 {4+iT+a} +\frac2 {2+iT+a} +\Obar( 0.4 ),
\end{multline*}
and therefore
\begin{multline*}
\bigg| \frac{\Gamma'(\frac12(\sigma+iT+a))}{\Gamma(\frac12(\sigma+iT+a))} - \frac{\Gamma'(\frac12(2+iT+a))}{\Gamma(\frac12(2+iT+a))} \bigg| \le \bigg| \log\bigg( 1 - \frac{2-\sigma}{4+iT+a} \bigg) \bigg| \\
+ \bigg| \frac{2-\sigma}{(\sigma+iT+a+2)(4+iT+a)} \bigg| +2\bigg| \frac{2-\sigma}{(\sigma+iT+a)(2+iT+a)} \bigg|+ 0.4.
\end{multline*}
Under the assumptions on $\sigma$, $a$, and $T$, we always have the inequality $\big| 2-\sigma/(4+iT+a) \big| \le \frac12$. The maximum modulus principle implies the inequality $\big| \frac1z \log(1-z) \big| \le \log 4$ for $|z|\le\frac12$, and so
\begin{multline*}
\bigg| \frac{\Gamma'(\frac12(\sigma+iT+a))}{\Gamma(\frac12(\sigma+iT+a))} - \frac{\Gamma'(\frac12(2+iT+a))}{\Gamma(\frac12(2+iT+a))} \bigg| \le \bigg| \frac{2-\sigma}{4+iT+a} \bigg| \log 4 \\
+ \bigg| \frac{2-\sigma}{(\sigma+iT+a+3)(5+iT+a)} \bigg|+2\bigg| \frac{2-\sigma}{(\sigma+iT+a)(2+iT+a)} \bigg| + 0.4
\end{multline*}
Finally we use the inequalities on $\sigma$, $a$, and $T$ to conclude that
\[
\bigg| \frac{\Gamma'(\frac12(\sigma+iT+a))}{\Gamma(\frac12(\sigma+iT+a))} - \frac{\Gamma'(\frac12(2+iT+a))}{\Gamma(\frac12(2+iT+a))} \bigg| \le \frac 25 \log 4 + \frac 2{5\sqrt{13}} +\frac 4 {3\sqrt{13}} +0.4 <1.4353,
\]
which amply suffices to finish the proof.
\end{proof}

We turn now to estimates for quantities associated with Dirichlet $L$-functions. The next few results do not require GRH to be true, and in fact their proofs cite identities from the literature that hold more generally no matter where the zeros of $L(s,\chi)$ might lie. Accordingly, we use the usual notation $\rho = \beta + i\gamma $ to denote a nontrivial zero of $L(s,\chi)$, and all sums in this section of the form $\sum_\rho$ denote sums over all such nontrivial zeros of the Dirichlet $L$-function.

\begin{lemma}
\label{Bchi used again lemma}
Let $q\ge2$, and let $\chi$ be a nonprincipal character\mod q. For any real number $T$,
$$
\sum_\rho \frac1{|2+iT-\rho|^2} < \frac{1}{2} \log \big( 0.609 q (|T|+5)  \big).
$$
\end{lemma}

\begin{proof}
It suffices to prove the lemma for primitive characters. For $\chi$ primitive, it is known \cite[equation~(10.37)]{magicbook} that as meromorphic functions on the complex plane,
\begin{equation}
\frac{L'(s,\chi)}{L(s,\chi)} = -\frac{1}{2}\log \frac{q}{\pi} - \frac{1}{2}\frac{\Gamma'(\frac{1}{2}(s+a) )}{\Gamma(\frac{1}{2}(s+a) )} + B(\chi) + \sum_{\rho} \bigg( \frac{1}{s-\rho}+\frac{1}{\rho} \bigg),
\label{from p83}
\end{equation}
where the constant $B(\chi)$ was described earlier in the proof of Lemma~\ref {b.chi.nonprimitive.lemma}, and where $a=0$ if $\chi(-1)=1$ and $a=1$ if $\chi(-1)=-1$. Taking real parts of both sides and using the identity~\eqref{Re Bchi}, we obtain after rearrangement
\begin{equation}
\Re \sum_{\rho} \frac{1}{s-\rho} = \Re \frac{L'(s,\chi)}{L(s,\chi)} + \frac{1}{2}\log \frac{q}{\pi} + \frac{1}{2}\Re \frac{\Gamma'(\frac{1}{2}(s+a) )}{\Gamma(\frac{1}{2}(s+a) )}.
\label{Re L'/L}
\end{equation}
If we put $z=\frac{1}{2}(s +a ) $ in Proposition~\ref {gamma bounds prop}, we see that for $\Re s\ge\frac18$,
\begin{align*}
\Re\frac{\Gamma'(\frac{1}{2}(s+a) )}{\Gamma(\frac{1}{2}(s+a) )} &= \Re\log\frac{s+a+2}2 - \Re\frac1{s+a+2}-\Re\frac2{s+a} + 0.2 \\
&\le \log | s+a+1| - \log 2 + 0 + 0.2 \le \log |s+3| - 0.493.
\end{align*}
Inserting this bound into equation~\eqref{Re L'/L} and putting $s=2+iT$,
\[
\Re \sum_{\rho} \frac{1}{2+iT-\rho} \le \Re \frac{L'(2+iT,\chi)}{L(2+iT,\chi)} + \frac{1}{2} \log\frac{q}{\pi} + \frac{1}{2} \log |5+iT| - 0.246.
\]
Now notice that
\begin{equation}
\bigg| \frac{L'(2+iT,\chi)}{L(2+iT,\chi)} \bigg| = \bigg| {-} \sum_{n=1}^{\infty} \frac{\chi(n)\Lambda(n)}{n^{2+iT}} \bigg| \le \sum_{n=1}^{\infty} \frac{\Lambda(n)}{n^2} = - \frac{\zeta'(2)}{\zeta(2)} < 0.57,
\label{bounding Lchi by zeta}
\end{equation}
and therefore
\begin{align*}
\Re \sum_{\rho} \frac{1}{2+iT-\rho} &\le 0.57 + \frac{1}{2} \log\frac{q}{\pi} + \frac{1}{2} \log |5+iT| - 0.246 \\
&\le \frac12\log q + \frac{1}{2} \log(|T|+5) + 0.57 - \frac12\log\pi - 0.246 \\
&\le \frac{1}{2} \log \big( q (|T|+5) \big) - 0.248 \le \frac12 \log \big( 0.609q (|T|+5) \big).
\end{align*}
We obtain finally
\begin{align}
\sum_\rho \frac1{|2+iT-\rho|^2} &< \sum_\rho \frac{2-\beta}{|2+iT-\rho|^2} \notag \\
&= \Re \sum_\rho \frac1{2+iT-\rho} \le \frac{1}{2} \log \big( 0.609 q (|T|+5)  \big)
\label{four and a half stars}
\end{align}
as claimed.
\end{proof}

\begin{prop}
\label{not too many nearby zeros prop}
For any nonprincipal character $\chi$ and any real number $T$, we have
\[
\#\{ \rho \colon |T-\Im \rho | \le 2 \} \le 4 \log\big(0.609q(|T|+5)\big).
\]
\end{prop}

\begin{proof}
This follows immediately from equation~\eqref {four and a half stars} and the inequalities
\[
\sum_{\substack{\rho \\ |T-\gamma| \le 2}} 1 \le 8 \sum_\rho \frac{1}{(2-\sigma)^2+(T-\gamma)^2} \le 8 \sum_\rho \frac{2-\beta}{|2+iT-\rho|^2}.
\]
\end{proof}

\begin{lemma}
\label{bound for Lprime over L lemma}
Let $s=\sigma+iT$ with $\frac14\le\sigma\le1$. For any primitive character $\chi\mod q$ with $q\ge2$, if $L(s,\chi)\ne0$ then
$$
\bigg|\frac{L'(s,\chi)}{L(s,\chi)} - \sum_{\substack{\rho \\ |T-\gamma|\le 2}} \frac{1}{s-\rho} \bigg| \le \sqrt2 \log \big(0.609q(|T|+5)\big) + 4.48.
$$
\end{lemma}

\begin{proof}
Applying equation~\eqref {from p83} at $s=\sigma+iT$ and again at $2+iT$, we obtain
\begin{equation*}
\frac{L'(s,\chi)}{L(s,\chi)} - \frac{L'(2+iT,\chi)}{L(2+iT,\chi)} = \frac{1}{2}\frac{\Gamma'(\frac{1}{2}(2+iT+a) )}{\Gamma(\frac{1}{2}(2+iT+a) )} - \frac{1}{2}\frac{\Gamma'(\frac{1}{2}(s+a) )}{\Gamma(\frac{1}{2}(s+a) )} + \sum_{\rho} \bigg(   \frac{1}{s-\rho} - \frac{1}{2+iT-\rho}  \bigg),
\end{equation*}
which implies
\begin{multline*}
\bigg|\frac{L'(s,\chi)}{L(s,\chi)} - \sum_{\substack{\rho \\ |T-\gamma|\le 2}} \frac{1}{s-\rho} \bigg| \le \bigg| \frac{L'(2+iT,\chi)}{L(2+iT,\chi)} \bigg| + \frac12 \bigg| \frac{\Gamma'(\frac{1}{2}(2+iT+a) )}{\Gamma(\frac{1}{2}(2+iT+a) )} - \frac{\Gamma'(\frac{1}{2}(s+a) )}{\Gamma(\frac{1}{2}(s+a) )} \bigg| \\
+ \sum_{\substack{\rho \\ |T-\gamma| > 2}} \bigg | \frac{1}{s-\rho}- \frac{1}{2+iT-\rho} \bigg| + \sum_{\substack{\rho \\ |T-\gamma|\le 2}} \frac{1}{|2+iT-\rho|}.
\end{multline*}
Using equation~\eqref {bounding Lchi by zeta} and Lemma~\ref {Gamma difference extracted lemma} to bound the first two terms on the right-hand side, we see that
\begin{multline}
\bigg|\frac{L'(s,\chi)}{L(s,\chi)} - \sum_{\substack{\rho \\ |T-\gamma|\le 2}} \frac{1}{s-\rho} \bigg| < 0.57 + 3.906 \\
+ \sum_{\substack{\rho \\ |T-\gamma| > 2}} \frac{2-\sigma}{|s-\rho||2+iT-\rho|} + \sum_{\substack{\rho \\ |T-\gamma|\le 2}} \frac{1}{|2+iT-\rho|}.
\label{two sums to prepare}
\end{multline}
To prepare the last two sums for an application of Lemma~\ref {Bchi used again lemma}, we note that when $|T-\gamma| > 2$,
\[
\frac{2-\sigma}{|s-\rho||2+iT-\rho|} < 2 \frac{|2+iT-\rho|}{|s-\rho|} \frac1{|2+iT-\rho|^2} < 2\sqrt2 \frac1{|2+iT-\rho|^2};
\]
on the other hand, when $|T-\gamma|\le2$,
\[
\frac{1}{|2+iT-\rho|} = \frac{|2+iT-\rho|} {|2+iT-\rho|^2} < \frac{2\sqrt2} {|2+iT-\rho|^2}.
\]
Therefore equation~\eqref{two sums to prepare} becomes, by Lemma~\ref {Bchi used again lemma},
\begin{align*}
\bigg|\frac{L'(s,\chi)}{L(s,\chi)} - \sum_{\substack{\rho \\ |T-\gamma|\le 2}} \frac{1}{s-\rho} \bigg| &< 0.57 + 3.906 + 2\sqrt2 \sum_\rho \frac1{|2+iT-\rho|^2} \\
&< 4.48 + \sqrt2 \log \big( 0.609 q (|T|+5)  \big)
\end{align*}
as claimed.
\end{proof}

We restore the assumption of GRH for the last proposition of this section, which is used in the proof of Lemma~\ref {third term} below.

\begin{prop}
Assume GRH. Let $s=\sigma+iT$ with $\frac14\le\sigma\le1$, $\sigma\ne\frac12$. If $\chi$ is any nonprincipal character\mod q, then
$$
\bigg| \frac{L'(s,\chi)}{L(s,\chi)} \bigg| \le \bigg( \frac{4}{ | \sigma - \frac12 |} + \sqrt2 \bigg) \log \big( 0.609q(|T|+5) \big) + 4.48 + \frac{\log q}{2^{\sigma} - 1 }.
$$
Furthermore, if $\chi$ is primitive and $q\ge2$, then the summand $(\log q)/(2^{\sigma} - 1)$ can be omitted from the upper bound.
\label{with and without primitive prop}
\end{prop}

\begin{proof}
Assume first that $\chi$ is primitive. Lemma~\ref {bound for Lprime over L lemma} tells us that
\begin{align*}
\bigg|\frac{L'(s,\chi)}{L(s,\chi)} \bigg| &\le \sum_{\substack{\rho \\ |T-\gamma|\le 2}} \frac{1}{|s-\rho|} + \sqrt2 \log \big(0.609q(|T|+5)\big) + 4.48 \\
&\le \frac{1}{|\sigma-\frac12|} \,\#\{ \rho \colon |T-\gamma | \le 2 \} + \sqrt2 \log \big(0.609q(|T|+5)\big) + 4.48
\end{align*}
under the assumption of GRH; the proposition for primitive $\chi$ now follows immediately from Proposition~\ref{not too many nearby zeros prop}.

If $\chi$ is not primitive, then $  L(s,\chi) = L(s,\chi^*) \prod_{p\mid q} \big(   1-\frac{\chi^*(p)}{p^s}  \big) $; we then have the identity
$$
\frac{L'(s,\chi)}{L(s,\chi)} = \frac{L'(s,\chi^*)}{L(s,\chi^*)} + \sum_{p\mid q} \frac{\chi^*(p)\log p}{p^s-\chi^*(p)}.
$$
Therefore
$$
\bigg| \frac{L'(s,\chi)}{L(s,\chi)} - \frac{L'(s,\chi^*)}{L(s,\chi^*)} \bigg| \le \sum_{p\mid q} \frac{ \log p }{p^{\sigma}-1} \le \frac{1}{2^\sigma - 1 } \sum_{p\mid q} \log p \le \frac{\log q}{2^{\sigma} - 1 },
$$
which finishes the proof of the proposition in full.
\end{proof}

\subsection{Bounds for the variance $V(q;a,b)$}
\label{variance bound section}

This section has two main purposes. First, we provide the proofs of Propositions~\ref{Sound method prop} and~\ref{only the first term survives prop}, two statements involving smoothed sums of the von Mangoldt function which were stated in Section~\ref {M evaluation section}. Second, we establish two sets of upper and lower bounds for the variance $V(q;a,b)$, one when $q$ is prime and one valid for all $q$. All of these results are stated with explicit constants and are valid for explicit ranges of~$q$.

\begin{lemma}
\label{funny complex inequality lemma}
For any real number $t$, we have $\big |\tfrac d{dt} \big| \Gamma\big( {-}\tfrac{1}{2} + it \big) \big|\big| \le \big| \Gamma'\big( {-}\tfrac{1}{2} + it \big) \big|$.
\end{lemma}

\begin{proof}
We show more generally that if $f(t)$ is any differentiable complex-valued function that never takes the value~0, then $|f(t)|$ is also differentiable and $\big| \frac d{dt} |f(t)| \big| \le | f'(t) |$; the lemma then follows since $\Gamma$ never takes the value~0. Write $f(t) = u(t) + iv(t)$ where $u$ and $v$ are real-valued; then
\[
\tfrac d{dt} |f(t)| = \tfrac d{dt} \sqrt{u(t)^2 + v(t)^2} = \frac{u(t)u'(t) + v(t)v'(t)}{\sqrt{u(t)^2 + v(t)^2}}
\]
while $|f'(t)| = |u'(t) + iv'(t)| = \sqrt{u'(t)^2 + v'(t)^2}$. The asserted inequality is therefore equivalent to $|u(t)u'(t) + v(t)v'(t)| \le {\sqrt{u(t)^2 + v(t)^2}} \sqrt{u'(t)^2 + v'(t)^2}$, which is a consequence of the Cauchy-Schwarz inequality.
\end{proof}

\begin{lemma}
We have $|\Gamma(s)| \le |\Gamma(\Re s)|$ for all complex numbers $s$.
\label{gamma bound}
\end{lemma}

\noindent Note that this assertion is trivially true if $\Re s$ is a nonpositive integer, under the convention $|\Gamma(-n)|=\infty$ for $n\ge0$.

\begin{proof}
We prove that the assertion holds whenever $\Re s>-n$, by induction on $n$. The base case $n=0$ can be derived from the integral representation $\Gamma(s)=\int_0^{\infty}t^{s-1}e^{-t}dt$, which gives
$$
|\Gamma(s)|\le \int_0^{\infty}|t^{s-1}|e^{-t}dt = \int_0^{\infty}t^{\Re s -1}e^{-t}dt=\Gamma(\Re s).
$$
Now assume that the assertion holds whenever $\Re s>-n$. Given a complex number $s$ for which $\Re s > -(n+1)$, we use the identity $\Gamma(s+1) = s\Gamma(s)$ and the induction hypothesis to write
$$
|\Gamma(s)|=\frac{|\Gamma(s+1)|}{|s|} \le \frac{|\Gamma(\Re s+1)|}{|s|} =\frac{|\Re s|}{|s|} |\Gamma(\Re s)| \le |\Gamma(\Re s)|,
$$
as desired.
\end{proof}

\begin{lemma}
\label{second term lemma}
For any nonprincipal character $\chi$,
\[
\Lsum{\gamma\in\R} \big| \Gamma\big({-}\tfrac{1}{2} +i\gamma\big) \big| \le 14.27\log q + 16.25.
\]
\end{lemma}

\noindent We remark that this lemma does not assume GRH, since the sum on the left-hand side only decreases if some of the zeros of $L(s,\chi)$ lie off the critical line.

\begin{proof}
First, by Proposition~\ref {not too many nearby zeros prop} applied with $T=0$, the number of zeros of $L(s,\chi)$ with $|\gamma|\le2$ is at most $4 \log(3.045q)$; thus by Lemma~\ref{gamma bound},
\begin{align}
\Lsum{|\gamma| \le 2} \big| \Gamma\big({-}\tfrac{1}{2} +i\gamma\big) \big| &\le \big| \Gamma\big({-}\tfrac{1}{2} \big) \big| \Lsum{|\gamma| \le 2} 1 \notag \\
&\le 8\sqrt\pi \log(3.045q) \le 14.18\log q + 15.79.
\label{zeros up to height 2}
\end{align}
We can write the remainder of the sum using Riemann-Stieltjes integration as
\begin{align*}
\Lsum{|\gamma| >2} \big| \Gamma\big({-}\tfrac{1}{2} +i\gamma\big) \big| &= \int_2^\infty \big| \Gamma\big({-}\tfrac{1}{2} +it\big) \big| \,d\big( N(t,\chi) - N(2,\chi) \big) \\
&= {-} \int_2^\infty \big( N(t,\chi) - N(2,\chi) \big) \frac d{dt} \big| \Gamma\big( {-}\tfrac{1}{2} + it \big) \big| \,dt;
\end{align*}
the vanishing of the boundary terms is justified by the upper bound $N(t,\chi) \ll_q t\log t$ (see Proposition~\ref{McCurley bound} for example) and the exponential decay of $\Gamma(s)$ on vertical lines.
We conclude from Lemma~\ref {funny complex inequality lemma} that
\begin{align*}
\Lsum{|\gamma| \le 2} \big| \Gamma\big({-}\tfrac{1}{2} +i\gamma\big) \big| &\le \int_2^\infty N(t,\chi) \big| \Gamma'\big( {-}\tfrac{1}{2} + it \big) \big| \,dt \\
&\le \int_2^\infty \bigg( \bigg( \frac t\pi + 0.68884 \bigg) \log \frac{qt}{2\pi e} + 10.6035 \bigg) \big| \Gamma'\big( {-}\tfrac{1}{2} + it \big) \big| \,dt
\end{align*}
by Proposition~\ref{McCurley bound}. Since $\log(qt/2\pi e) = \log q + \log(t/2\pi e)$, the right-hand side is simply a linear function of $\log q$; using numerical integration we see that
\[
\Lsum{|\gamma| \le 2} \big| \Gamma\big({-}\tfrac{1}{2} +i\gamma\big) \big| \le 0.09 \log q + 0.46.
\]
Combining this upper bound with the bound in equation~\eqref{zeros up to height 2} establishes the lemma.
\end{proof}

\begin{lemma}
\label{third term}
Assume GRH. For any nonprincipal character $\chi$,
\[
\int_{-3/4-i\infty}^{-3/4+i\infty} \bigg| \frac{L'(s+1,\chi)}{L(s+1,\chi)} \Gamma(s) \bigg| \,ds \le 101\log q+112.
\]
\end{lemma}

\begin{proof}
Proposition~\ref{with and without primitive prop} with $\sigma = \frac{1}{4}$ tells us that for any real number $t$,
\begin{align*}
\bigg|\frac{L'(\frac{1}{4}+it,\chi)}{L(\frac{1}{4}+it,\chi)}\bigg| &\le 17.42 \log \big( 0.609q(|t|+5) \big) + 4.48 + \frac{\log q}{0.1892} \\
&\le 22.71\log q + 17.42\log(|t|+5) - 4.159,
\end{align*}
and therefore
\[
\int_{-3/4-i\infty}^{-3/4+i\infty} \bigg| \frac{L'(s+1,\chi)}{L(s+1,\chi)} \Gamma(s) \bigg| \,ds \le \int_{-\infty}^\infty \big( 22.71\log q + 17.42\log(|t|+5) - 4.159 \big) \big| \Gamma\big( {-}\tfrac34+it \big) \big| \,dt.
\]
Again this integral is a linear function of $\log q$, and a numerical calculation establishes the particular constants used in the statement of the lemma.
\end{proof}

With these lemmas in hand, we are now able to provide the two proofs deferred until now from Section~\ref {M evaluation section}.

\begin{proof}[Proof of Proposition~\ref{Sound method prop}]
We begin with the Mellin transform formula, valid for any real number $c>0$,
$$
- \sum_{n=1}^{\infty} \frac{\chi(n) \Lambda(n)}{n} e^{-n/y} = \frac{1}{2\pi i} \int_{c-i\infty}^{c+i\infty} \frac{L'(s+1,\chi)}{L(s+1,\chi)} \Gamma(s) y^s \,ds
$$
(see \cite[equations (5.24) and (5.25)]{magicbook}). We move the contour to the left, from the vertical line $\Re s=c$ to the vertical line $\Re s=-\frac34$, picking up contributions from the pole of $\Gamma$ at $s=0$ as well as from each nontrivial zero of $L(s,\chi)$. The result is
\begin{multline}
- \sum_{n=1}^{\infty} \frac{\chi(n) \Lambda(n)}{n} e^{-n/y}  =  \frac{L'(1,\chi)}{L(1,\chi)} + \Lsum{\gamma\in\R} \Gamma\big({-}\tfrac{1}{2} +i\gamma\big) y^{-1/2 + i \gamma} \\
+ \frac{1}{2\pi i} \int_{-3/4-i\infty}^{-3/4+i\infty} \frac{L'(s+1,\chi)}{L(s+1,\chi)} \Gamma(s) y^s \,ds
\label{three pieces}
\end{multline}
since we are assuming GRH. (Strictly speaking, we should consider truncations of these infinite integrals; however, the exponential decay of $\Gamma(s)$ in vertical strips implies that the contributions at large height do vanish in the limit.)

The sum on the right-hand side can be bounded by
\[
\bigg| \Lsum{\gamma\in\R} \Gamma\big({-}\tfrac{1}{2} +i\gamma\big) y^{-1/2 + i \gamma} \bigg| \le y^{-1/2} \Lsum{\gamma\in\R} \big| \Gamma\big({-}\tfrac{1}{2} +i\gamma\big) \big| \le \frac{14.27\log q + 16.25}{y^{1/2}}
\]
by Lemma~\ref{second term lemma}, while the integral can be bounded by
\begin{align*}
\bigg| \frac{1}{2\pi i} \int_{-3/4-i\infty}^{-3/4+i\infty} \frac{L'(s+1,\chi)}{L(s+1,\chi)} \Gamma(s) y^s \,ds \bigg| &\le \frac1{2\pi y^{3/4}} \int_{-3/4-i\infty}^{-3/4+i\infty} \bigg| \frac{L'(s+1,\chi)}{L(s+1,\chi)} \Gamma(s) \bigg| \,ds \\
&\le \frac{101\log q+112} {2\pi y^{3/4}}
\end{align*}
by Lemma~\ref {third term}. Using these two inequalities in equation~\eqref{three pieces} establishes the proposition.
\end{proof}

\begin{proof}[Proof of Proposition~\ref{only the first term survives prop}]
Since $1\le a<q$, we may write
\begin{equation}
\sum_ {n\equiv a \mod q} \frac{\Lambda(n)}{n} e^{-n/q^2} = \frac{\Lambda(a)}{a} e^ {-a/q^2} + \Obar\bigg( \sum_{\substack{q \le n \le q^2 \\ n\equiv a \mod q}} \frac{\Lambda(n)}{n} + \sum_{\substack{n > q^2 \\ n\equiv a \mod q}} \frac{\Lambda(n)}{n} e^ {-n/q^2}  \bigg).
\label{easier than last prop}
\end{equation}
Since $\Lambda(n)/n \le (\log n)/n$, which is a decreasing function of $n$ for $n\ge3$, we have
\begin{equation*}
\sum_{\substack{n > q^2 \\ n\equiv a \mod q}} \frac{\Lambda(n)}{n} e^ {-n/q^2}  \le \frac{\log q^2}{q^2} \sum_{j=q}^{\infty} e^{-(qj +a)/q^2} \le \frac{2\log q}{q^2} e^{-1} \sum_{k=0}^{\infty} e^{-j/q} =  \frac{2\log q}{q^2} e^{-1} \frac{1}{1-e^{-1/q}};
\end{equation*}
note here that $1\le a < q$ so $q\ge 2$. As the function $t/(1-e^{-t})$ is bounded by $1/2(1-e^{-1/2})$ for $0<t\le\frac12$, we conclude that
\[
\sum_{\substack{n > q^2 \\ n\equiv a \mod q}} \frac{\Lambda(n)}{n} e^ {-n/q^2} \le \frac{2\log q}{q} e^{-1}\frac1{2(1-e^{-1/2})} < 0.935 \frac{\log q}q.
\]
We bound the second term of equation~\eqref{easier than last prop} crudely:
$$
\sum_{\substack{q \le n \le q^2 \\ n\equiv a \mod q}} \frac{\Lambda(n)}{n} \le (\log q^2) \sum_{j=1}^{q-1} \frac{1}{qj + a}  \le \frac{2\log q}{q} \sum_{j=1}^{q-1}\frac{1}{j} \le \frac{2\log q}{q} ( \log q + 1 ).
$$
Finally, for the first term of equation~\eqref{easier than last prop}, the estimate $e^{-t} = 1 + \Obar(t)$ for $t\ge0$ allows us to write
$$
\frac{\Lambda(a)}{a} e^{-a/q} = \frac{\Lambda(a)}{a} \bigg(   1 + \Obar \bigg( \frac{a}{q} \bigg)    \bigg)    =  \frac{\Lambda(a)}{a} + \Obar\bigg( \frac{\log q}{q}    \bigg).
$$
Using these three deductions transforms equation~\eqref {easier than last prop} into the statement of the proposition.
\end{proof}

We now turn to the matter of giving explicit upper and lower bounds for $V(q;a,b)$. In the case where $q$ is prime, we are already able to establish such estimates.

\begin{prop}
\label{prime V bounds prop}
If $q\ge 150$ is prime, then
\[
2 (q-1) (\log q - 2.42) - 47.238\log^2 q \le V(q;a,b) \le 2 (q-1) (\log q - 0.99) + 47.238\log^2 q.
\]
\end{prop}

\begin{proof}
Combining Theorem~\ref {variance evaluation theorem} with Proposition~\ref {M exact evaluation prop}, we see that
\begin{align}
V(q;a,b) &= 2\phi(q) \big( \L(q) + K_q(a-b) + \iota_q(-ab^{-1})\log2 \big) + 2M^*(q;a,b) \notag \\
&= 2\phi(q) \bigg( \L(q) + K_q(a-b) + \iota_q(-ab^{-1})\log2 + \frac{\Lambda(r_1)}{r_1} +\frac{\Lambda(r_2)}{r_2} + H_0(q;a,b) \bigg) \notag \\
&\qquad{}+ \Obar\bigg( \frac{47.238\phi(q)\log^2 q}q \bigg)
\label{V explicit}
\end{align}
for any $q\ge150$, where $r_1$ and $r_2$ denote the least positive residues of $ab^{-1}$ and $ba^{-1}\mod q$. Since we are assuming $q$ is prime, both $K_q(a-b)$ and $H_0(q;a,b)$ vanish, and we have
\[
V(q;a,b) = 2(q-1) \bigg( \log \frac q{2\pi e^{\gamma_0}} + \iota_q(-ab^{-1})\log2 + \frac{\Lambda(r_1)}{r_1} +\frac{\Lambda(r_2)}{r_2} \bigg) + \Obar( 47.238\log^2 q).
\]
The function $\Lambda(n)/n$ is nonnegative and bounded above by $(\log 3)/3$, and the function $\iota_q$ takes only the values~0 and ~1; therefore the quantity in large parentheses satisfies the bounds
\[
\log q - 2.42 \le \log \frac q{2\pi e^{\gamma_0}} + \iota_q(-ab^{-1})\log2 + \frac{\Lambda(r_1)}{r_1} +\frac{\Lambda(r_2)}{r_2} \le \log q - 0.99,
\]
which establishes the proposition.
\end{proof}

We require two additional lemmas before we can treat the case of general (possibly composite)~$q$.

\begin{lemma}
\label{H0 bound lemma}
With $H_0$ defined in Definition~\ref {h0 and H0 def}, we have $ -(4\log q)/q \le H_0(q;a,b) \le 4.56$ for any reduced residues $a$ and $b\mod q$.
\end{lemma}

\begin{proof}
Since $e(q;p,r) \ge 1$ always, we have
\[
h_0(q;p,r) = \frac{1}{\phi(p^{\nu})}\frac{\log p}{p^{e(q;p,r)} (1-p^ {-e(q;p,1)})} \le \frac1{p-1} \frac{\log p}{p-1} \le 4 \frac{ \log p}{p^2}.
\]
Therefore
\[
H_0(q;a,b) \le \sum_{p\mid q} \big( h_0(q;p,ab^{-1}) + h_0(q;p,ba^{-1}) \big) \le 8 \sum_{p\mid q} \frac{ \log p}{p^2} < 8 \sum_{n=1}^\infty \frac{\Lambda(n)}{n^2} = 8\bigg| \frac{\zeta'(2)}{\zeta(2)} \bigg| \le 4.56,
\]
which establishes the upper bound. On the other hand, note that $p^{e(q;p,1)}$ is an integer larger than 1 that is congruent to 1\mod{q/p^\nu}. Therefore $p^{e(q;p,1)} \ge q/p^\nu + 1$, and so
\begin{align*}
H_0(q;a,b) \ge -2 \sum_{p\mid q} h_0(q;p,1) &= -2 \sum_ {p\mid q} \frac{1}{\phi(p^{\nu})}\frac{\log p}{p^{e(q;p,1)} -1} \\
&\ge -2 \sum_{p\mid q} \frac{1}{p^{\nu}(1-1/p)}\frac{\log p}{q/p^\nu} \ge -\frac4q \sum_{p\mid q} \log p \ge -\frac{4\log q}q,
\end{align*}
which establishes the lower bound.
\end{proof}

\begin{lemma}
\label{log p over p-1 sum}
If $q\ge2$ is any integer, then
$$
\sum_{p\mid q} \frac{\log p }{p-1} \le 1.02 \log\log q + 3.04.
$$
\end{lemma}

\begin{proof}
We separate the sum into two intervals at the point $1+\log q$. The contribution from the larger primes is at worst
$$
\sum_{\substack{p\mid q\\ p \ge 1+\log q}} \frac{\log p }{p-1} \le \frac{1}{\log q} \sum_{p\mid q} \log p \le \frac{\log q}{\log q} =1.
$$
For the smaller primes, recall the usual notation $\theta(t) = \sum_{p\le t} \log p$. We will use the explicit bound $\theta(t)\le 1.01624 t$ for $t>0$ from Theorem 9 of \cite{rosser} , and so the contribution from the smaller primes is bounded by
\begin{align*}
\sum_{\substack{p\mid q\\ p < 1+\log q}} \frac{\log p }{p-1}  &\le \sum_{p < 1+\log q} \frac{\log p }{p-1} = \int_{2^-}^{1+\log q} \frac{d\theta(t)}{t-1} \\
&{}= \frac{\theta(1+\log q)}{\log q} + \int_2^{1+\log q} \frac{\theta(t)}{(t-1)^2}\,dt \\
&{}\le 1.01624\bigg(\frac{1+\log q}{\log q}+ \int_2^{1+\log q} \frac{t\,dt}{(t-1)^2}  \bigg) \\
&{}=  1.01624\bigg(1+ \frac{1}{\log q} + \log  \log q - \frac{1}{\log q} + 1 \bigg)\\
&{}=  1.01624\log  \log q +2.03248,
\end{align*}
which finishes the proof of the lemma.
\end{proof}

\begin{prop}
\label{composite V bound prop}
If $q\ge 500$, then
\[
2\phi(q) ( \log q -1.02\log\log q - 7.34 ) \le V(q;a,b) \le 2\phi(q) (\log q + 6.1).
\]
\end{prop}

\begin{proof}
We begin with equation~\eqref{V explicit}, expanding the functions $\L$ and $K_q$ according to Definition~\ref {iota and Lq and Rqn def}:
\begin{multline}
V(q;a,b) = 2\phi(q) \bigg( \log \frac q{2\pi e^{\gamma_0}} - \sum_{p\mid q} \frac{\log p}{p-1} + \frac{\Lambda(q/(q,a-b))}{\phi(q/(q,a-b))} \\
+ \iota_q(-ab^{-1})\log2 + \frac{\Lambda(r_1)}{r_1} +\frac{\Lambda(r_2)}{r_2} + H_0(q;a,b) + \Obar\bigg( \frac{23.62\log^2 q}q \bigg) \bigg).
\label{need for final calculation}
\end{multline}
The last term on the first line is nonnegative and bounded above by $\log2$, while the first three terms on the second line are nonnegative and bounded together by $\log 2+\frac 2 3 \log 3$ as in the proof of Proposition~\ref {prime V bounds prop}. The term $H_0(q;a,b)$ is bounded above by 4.56 and below by $(-4\log q)/q$ by Lemma~\ref {H0 bound lemma}. Therefore
\begin{multline}
2\phi(q) \bigg( \log q - \log 2\pi e^{\gamma_0} - \sum_{p\mid q} \frac{\log p}{p-1} - \frac{4\log q}q + \Obar\bigg( \frac{23.62\log^2 q}q \bigg) \bigg) \le V(q;a,b) \\
\le 2\phi(q) \bigg( \log q - \log 2\pi e^{\gamma_0} + \log2 + \log 2+\frac 2 3 \log 3 + 4.56 + \Obar\bigg( \frac{23.62\log^2 q}q \bigg) \bigg).
\label{need this intermediate step}
\end{multline}
The sum being subtracted on the top line is bounded above by $1.02 \log\log q + 3.04$ by Lemma~\ref{log p over p-1 sum}. Lastly, a calculation shows that the $\Obar$ error term is at most $1.83$ for $q\ge500$, and therefore
\begin{multline*}
2\phi(q) \bigg( \log q - \log 2\pi e^{\gamma_0} - (1.02 \log\log q + 3.04) - \frac{4\log q}q - 1.83 \bigg) \le V(q;a,b) \\
\le 2\phi(q) \big( \log q - \log 2\pi e^{\gamma_0} + \log2 + \log 2+\tfrac 2 3 \log 3+ 4.56 + 1.83 \big),
\end{multline*}
which implies the assertion of the proposition.
\end{proof}

\subsection{Bounds for the density $\delta(q;a,b)$}
\label{density bounds section}

We use the results of the previous section to obtain explicit upper and lower bounds on $\delta(q;a,b)$; from these bounds, we can prove in particular that all of the largest values of these densities occur when the modulus $q$ is less than an explicit bound.
In the proof of Theorem~\ref{delta series theorem}, we expanded several functions, including an instance of $\sin$, into their power series at the origin. While this yielded an excellent theoretical formula, for numerical purposes we will take a slightly different approach involving the error function $\Erf(z) = \frac2{\sqrt\pi} \int_0^z e^{-t^2}\,dt$. The following two lemmas allow us to write the density $\delta(q;a,b)$ in terms of the error function.

\begin{lemma}
\label{special erf integral}
For any constants $v>0$ and $\rho$,
\[
\int_{-\infty}^\infty t^4 e^{-vt^2/2}\,dt = \frac{3\sqrt{2\pi}}{v^{5/2}}
\qquad\text{and}\qquad
\int_{-\infty}^\infty \frac{\sin\rho t}t e^{-vt^2/2}\,dt = \pi \Erf\bigg( \frac\rho{\sqrt{2v}} \bigg).
\]
\end{lemma}

\begin{proof}
For the first identity, a change of variables gives
\[
\int_{\infty}^\infty t^4 e^{-vt^2/2}\,dt = v^{-5/2} \int_{\infty}^\infty w^4 e^{-w^2/2}\,dw = v^{-5/2} M_2(\infty) = \frac{3\sqrt{2\pi}}{v^{5/2}}
\]
by Lemma~\ref {normal moments lemma}.
Our starting point for the second identity is~\cite[equation (7.4.6)]{handbook}: for any constants $a>0$ and $x$,
\[
\int_0^\infty e^{-at^2} \cos 2xt\,dt = \frac12 \sqrt{\frac\pi a} e^{-x^2/a},
\]
which can be rewritten as
\[
\sqrt{\frac\pi a} e^{-x^2} = \int_{-\infty}^\infty e^{-at^2} \cos (2xt\sqrt a) \,dt.
\]
Integrating both sides from $x=0$ to $x=w$ yields
\[
\frac\pi{2\sqrt a} \Erf(w) = \int_{-\infty}^\infty e^{-at^2} \bigg( \int_0^w \cos (2xt\sqrt a)\,dx \bigg) \,dt = \int_{-\infty}^\infty e^{-at^2} \frac{\sin(2wt\sqrt a)}{2t\sqrt a} \,dt
\]
(the interchanging of the integrals in the middle expression is justified by the absolute convergence of the integral). Setting $a=\frac v2$ and $w=\frac\rho{\sqrt{2v}}$, we obtain
\[
\frac\pi{\sqrt{2v}} \Erf\bigg( \frac\rho{\sqrt{2v}} \bigg) = \int_{-\infty}^\infty e^{-vt^2/2} \frac{\sin\rho t}{t\sqrt{2v}} \,dt,
\]
which establishes the lemma.
\end{proof}

\begin{lemma}
\label{explicit delta in terms of erf lemma}
Assume GRH and LI. Let $a$ be a nonsquare\mod q and $b$ a square\mod q.
If $V(q;a,b) \ge 531$, then
\begin{multline*}
\delta(q;a,b) = \frac12 + \frac12 \Erf\bigg( \frac{\rho(q)}{\sqrt{2V(q;a,b)}} \bigg) \\
{}+ \Obar\bigg( \frac{47.65 \rho(q)}{V(q;a,b)^{3/2}} + 0.03506
\frac{e^{-9.08\phi(q)}}{\phi(q)} + 63.68 \rho(q) e^
{-V(q;a,b)^{1/2}/2}\bigg).
\end{multline*}
\end{lemma}

\begin{proof}
From Definition~\ref{Fzchi and Phiqabz def}, we know that
\begin{align*}
\log \Phi_{q;a,b}(x) &= \allchisum \Lsum{\gamma>0} \log J_0\bigg( \frac{2|\chi(a)-\chi(b)|x}{\sqrt{\frac14+\gamma^2}} \bigg) \\
&= \tfrac12 \allchisum \Lsum{\gamma\in\R} \log J_0\bigg( \frac{2|\chi(a)-\chi(b)|x}{\sqrt{\frac14+\gamma^2}} \bigg)
\end{align*}
by the functional equation for Dirichlet $L$-functions. If $|x|\le\frac14$, then the argument of $J_0$ is at most $2\cdot 2\cdot \frac14/\frac12 = 2$ in absolute value. Since the Taylor expansion $\log J_0(x)= -{x^2}/{4} + \Obar(.0311 x^4)$ is valid for $|x|\le2$, we see that
\begin{align}
\log \Phi_{q;a,b}(x) &= \tfrac12 \allchisum \Lsum {\gamma\in\R} \bigg( {-} \frac{|\chi(a)-\chi(b)|^2x^2}{\frac14+\gamma^2} + \Obar \bigg( .0311 \frac{16|\chi(a)-\chi(b)|^4x^4}{(\frac14+\gamma^2)^2} \bigg) \bigg) \notag \\
&= -\tfrac12x^2 \allchisum \Lsum {\gamma\in\R} \frac{|\chi(a)-\chi(b)|^2}{\frac14+\gamma^2} \notag \\
&\qquad{}+ \Obar \bigg( \tfrac12x^4 \allchisum \Lsum {\gamma\in\R} .0311 \frac{16\cdot4|\chi(a)-\chi(b)|^2}{\frac14(\frac14+\gamma^2)} \bigg) \notag \\
&= -\tfrac12 V(q;a,b)x^2 + \Obar( 39.81 V(q;a,b)x^4 )
\label{saratoga}
\end{align}
when $|x|\le\frac14$. Moreover, the error term in the expansion $\log J_0(x)= -{x^2}/{4} + \Obar(.0311 x^4)$ is always nonpositive as a consequence of Lemma~\ref {log bessel lemma}(c), and hence the same is true for the recently obtained error term $\Obar( 39.81 V(q;a,b)x^4 )$. This knowledge allows us to use the expansion $e^t = 1 + \Obar(t)$ for $t\le0$, which yields
\[
\Phi_{q;a,b}(x) = e^{-V(q;a,b)x^2/2} \big( 1 + \Obar( 39.81 V(q;a,b)x^4 ) \big)
\]
when $|x|\le\frac14$.

Proposition~\ref {delta comes down to small interval prop} says that when $V(q;a,b) \ge 531$,
\begin{multline*}
\delta(q;a,b) = \frac12 + \frac1{2\pi} \int_{-V(q;a,b)^{-1/4}}^{V(q;a,b)^{-1/4}} \frac{\sin \rho(q) x}x \Phi_{q;a,b}(x) \,dx \\
{}+ \Obar\bigg( 0.03506 \frac{e^{-9.08\phi(q)}}{\phi(q)} + 63.67 \rho(q) e^ {-V(q;a,b)^{1/2}/2}\bigg).
\end{multline*}
Notice that $V(q;a,b)^{-1/4} \le 531^{-1/4} < \frac14$, and so we may use our approximation for $\Phi_{q;a,b}(x)$ to deduce that
\begin{multline}
\delta(q;a,b) = \frac12 + \frac1{2\pi} \int_{-V(q;a,b)^{-1/4}}^{V(q;a,b)^{-1/4}} \frac{\sin \rho(q) x}x e^{-V(q;a,b)x^2/2} \big( 1 + \Obar( 39.81 V(q;a,b)x^4 ) \big) \,dx \\
{}+ \Obar\bigg( 0.03506 \frac{e^{-9.08\phi(q)}}{\phi(q)} + 63.67
\rho(q) e^ {-V(q;a,b)^{1/2}/2}\bigg). \label{erf is coming}
\end{multline}
The main term can be evaluated by the second identity of Lemma~\ref {special erf integral}:
\begin{align*}
\frac1{2\pi} \int_{-V(q;a,b)^{-1/4}}^{V(q;a,b)^{-1/4}} & \frac{\sin \rho(q) x}x e^{-V(q;a,b)x^2/2} \,dx \\
&= \frac1{2\pi} \int_{-\infty}^\infty \frac{\sin \rho(q) x}x e^{-V(q;a,b)x^2/2} \,dx + \Obar \bigg( \frac1{\pi} \int_{V(q;a,b)^{-1/4}}^\infty \bigg| \frac{\sin \rho(q) x}x \bigg| e^{-V(q;a,b)x^2/2} \,dx \bigg) \\
&= \frac12 \Erf\bigg( \frac{\rho(q)}{\sqrt{2V(q;a,b)}} \bigg) + \Obar \bigg( \frac1{\pi} \int_{V(q;a,b)^{-1/4}}^\infty \rho(q)V(q;a,b)^{1/4}x e^{-V(q;a,b)x^2/2} \,dx \bigg) \\
&= \frac12 \Erf\bigg( \frac{\rho(q)}{\sqrt{2V(q;a,b)}} \bigg) + \Obar \bigg( \frac{\rho(q)}{\pi V(q;a,b)^{3/4}} e^{-V(q;a,b)^{1/2}/2} \bigg).
\end{align*}
The error term in the integral in equation~\eqref {erf is coming} can be estimated by the first identity of Lemma~\ref {special erf integral}:
\begin{multline*}
\frac1{2\pi} \int_{-V(q;a,b)^{-1/4}}^{V(q;a,b)^{-1/4}} \bigg| \frac{\sin \rho(q) x}x \bigg| e^{-V(q;a,b)x^2/2} 39.81 V(q;a,b)x^4 \,dx \\
\le 6.336 \rho(q) V(q;a,b) \int_{-\infty}^\infty x^4 e^{-V(q;a,b)x^2/2} \, dx \le 19.008\sqrt{2\pi} \cdot \rho(q) V(q;a,b)^{-3/2}.
\end{multline*}
Therefore equation~\eqref {erf is coming} becomes
\begin{multline*}
\delta(q;a,b) = \frac12 + \frac12 \Erf\bigg( \frac{\rho(q)}{\sqrt{2V(q;a,b)}} \bigg) + \Obar \bigg( \frac{\rho(q)}{\pi V(q;a,b)^{3/4}} e^{-V(q;a,b)^{1/2}/2} \bigg) \\
{}+ \Obar\bigg( \frac{47.65 \rho(q)}{V(q;a,b)^{3/2}} + 0.03506
\frac{e^{-9.08\phi(q)}}{\phi(q)} + 63.67 \rho(q) e^
{-V(q;a,b)^{1/2}/2}\bigg).
\end{multline*}
Since $1/\pi V(q;a,b)^{3/4} \le 1/\pi(531)^{3/4} < 0.01$, this last estimate implies the statement of the lemma.
\end{proof}

We are now ready to bound $\delta(q;a,b)$ for all large prime moduli~$q$.

\begin{theorem}
\label{delta bound prime q theorem}
Assume GRH and LI. If $q\ge400$ is prime, then $\delta(q;a,b) <
0.5262$ for all reduced residues $a$ and $b\mod q$. If $q\ge1000$ is
prime, then $\delta(q;a,b) < 0.51$.
\end{theorem}

\begin{proof}
We may assume that $a$ is a nonsquare\mod q and $b$ is a square\mod q, for otherwise $\delta(q;a,b) \le \frac12$.
When $q\ge 331$ is prime, Proposition~\ref{prime V bounds prop} and
a quick calculation yield
\begin{equation}
V(q;a,b) \ge 2 (q-1) (\log q - 2.42) - 47.238\log^2 q \ge 2 q (\log
q - 2.42) - 48\log^2 q \ge 531. \label{little simpler lower bound}
\end{equation}
Therefore Lemma~\ref {explicit delta in terms of erf lemma} applies, yielding (since $\rho(q)=2$ and $\phi(q)=q-1$)
\begin{align*}
\delta(q;a,b) &= \frac12 + \frac12 \Erf\bigg( \sqrt{ \frac2{V(q;a,b)} } \bigg) + \Obar\bigg( \frac{95.3}{V(q;a,b)^{3/2}} + 0.03506 \frac{e^{-9.08 q}}{q-1} + 127.36 e^ {-V(q;a,b)^{1/2}/2}\bigg) \\
&\le \frac12 + \frac12 \Erf\bigg( \sqrt{ \frac 2 {2q (\log q - 2.42) - 48\log^2 q} } \bigg) \\
&\qquad{}+ \frac{95.3}{(2q (\log q - 2.42) - 48\log^2 q)^{3/2}} +
0.03506 \frac{e^{-9.08q}}{q-1} + 127.36 e^ {-\sqrt{q (\log q -
2.42)/2 - 12\log^2 q}} ,
\end{align*}
using the second inequality in equation~\eqref {little simpler lower
bound}. This upper bound is decreasing for $q\ge331$, and so
calculating it at $q=400$ and $q=1000$ establishes the inequalities
given in the theorem.
\end{proof}

A similar bound for composite moduli $q$ requires one last estimate.

\begin{lemma}
\label{rho bound lemma}
For all $q\ge3$, we have $\rho(q) \le 2q^{1.04/\log\log q}$.
\end{lemma}

\begin{proof}
We first record some explicit estimates on the prime counting
functions $\pi(y) = \sum_{p\le y}1$ and $\theta(y) = \sum_{p\le y}\log p$. Rosser and
Shoenfeld \cite[Corollary 1 and Theorems 9 and 10]{rosser} give, for $y\ge 101$, the bounds $0.84y \le \theta(y)\le
1.01624y$ and $\pi(y) \le 1.25506y/\log y$. Therefore
\begin{equation}
\pi(y) \le \frac{1.25506 y}{\log y} \le \frac{1.25506\theta(y)/0.84}{\log \theta(y) - \log 1.01624}\le \frac{1.5\theta(y)}{\log\theta(y)}
\label{RS bounds}
\end{equation}
(a calculation shows that the last inequality holds for $\theta(y)\ge61$, which is valid in the range $y\ge101$).

Now consider integers of the form $q(y) = \prod_{p\le y} p$, so that $\omega(q(y)) = \pi(y)$ and $\log q(y) = \theta(y)$. Equation~\eqref{RS bounds} becomes $\omega(q(y)) \le 1.5(\log q(y))/\log\log q(y)$; while the derivation was valid for $y\ge101$, one can calculate that the inequality holds for $3\le y\le 101$ as well. The following standard argument then shows that
\begin{equation}
\omega(q) \le \frac{1.5\log q}{\log\log q}
\label{omega bound}
\end{equation}
holds for all integers $q\ge3$: if $q$ has $k$ distinct prime factors, then choose $y$ to be the $k$th prime. Then the inequality~\eqref {omega bound} has been shown to hold for $q(y)$, and therefore it holds for $q$ as well, since the left-hand side is $k$ in both cases while the right-hand side is at least as large for $q$ as it is for $q(y)$.

(This argument uses the fact that the right-hand side is an increasing function, which holds only for $q\ge e^e$; therefore technically we have proved~\eqref {omega bound} only for numbers with at least three distinct prime factors, since only then does the corresponding $q(y)$ exceed $e^e$. However, the right-hand side of \eqref {omega bound} is always at least 4 in the range $q\ge3$, and so numbers with one or two distinct prime factors easily satisfy the inequality.)

Finally, the inequality $\rho(q) \le 2^{\omega(q)+1}$ that was noted in Definition~\ref {rho def} allows us to conclude that $\rho(q) \le 2^{1+1.5(\log q)/\log\log q} < 2q^{{1.04}/{\log\log q}}$ for all $q\ge3$, as desired.
\end{proof}

\begin{theorem}
Assume GRH and LI. If $q>480$ and $q\notin\{840,1320\}$, then
$\delta(q;a,b) < 0.75$ for all reduced residues $a$ and $b\mod q$.
 \label{delta<0.75 theorem}
\end{theorem}

\begin{proof}
Again we may assume that $a$ is a nonsquare\mod q and $b$ is a square\mod q.
First we restrict to the range $q\ge 260000 $; by
Proposition~\ref{composite V bound prop} we have $V(q;a,b)>531$.
Using Lemma~\ref {explicit delta in terms of erf lemma}, together
with the upper bound for $\rho(q)$ from Lemma~\ref {rho bound lemma}
and the lower bound for $V(q;a,b)$ from Proposition~\ref {composite
V bound prop}, we have
\begin{multline}
\delta(q;a,b) \le \frac12 + \frac12 \Erf\bigg( \frac{2q^{1.04/\log\log q}}{2\sqrt{\phi(q) \big( \log q -1.02\log\log q - 7.34 \big)}}\bigg) \\
+ \frac{33.7q^{1.04/\log\log q}}{\phi(q)^{3/2} \big( \log q -1.02\log\log q - 7.34 \big)^{3/2}} + 0.03506 \frac{e^{-9.08\phi(q)}}{\phi(q)} \\
+ 127.36q^{1.04/\log\log q} \exp \bigg( {-}\sqrt{ \frac{\phi(q)}2
\bigg( \log q -1.02\log\log q - 7.34 \bigg)}\bigg). \label{getting
ugly}
\end{multline}

Rosser and Schoenfeld~\cite[Theorem 15]{rosser} have given the bound
\begin{equation}
\label{phi lower bound}
\phi(q)>\frac{q}{e^{\gamma_0}\log\log q + 2.50637/\log\log q}
\end{equation}
for $q\ge 3$. When this lower bound
is substituted for $\phi(q)$ in the upper bound~\eqref{getting
ugly}, the result is a smooth function of $q$ that is well-defined
and decreasing for $q\ge260000$, and its value at $q=260000$ is less
than 0.75.

We now turn to the range $1000\le q \le 260000$. We first compute
explicitly, for each such modulus $q$, the lower bound for $V(q;a,b)$ in equation~\eqref {need
this intermediate step}; the value of this sharper lower bound turns out always to exceed $531$ in this range. Consequently, we may use Lemma~\ref {explicit delta in terms
of erf lemma} together with the lower bound for $V(q;a,b)$ from
equation~\eqref {need this intermediate step}, obtaining
\begin{align*}
\delta(q;{}&a,b) \le \frac12 + \frac12 \Erf\bigg( \frac{\rho(q)}{2\sqrt{\phi(q) \big( \log q - \log 2\pi e^{\gamma_0} - \sum_{p\mid q} \frac{\log p}{p-1} - \frac{4\log q}q - \frac{23.62\log^2 q}q \big)}} \bigg) \\
&+ \frac{17.85 \rho(q)}{\phi(q)^{3/2} \big( \log q - \log 2\pi e^{\gamma_0} - \sum_{p\mid q} \frac{\log p}{p-1} - \frac{4\log q}q - \frac{23.62\log^2 q}q \big)^{3/2}} + 0.03506 \frac{e^{-9.08\phi(q)}}{\phi(q)} \\
&+ 63.68 \rho(q) \exp \bigg( {-} \sqrt{ \frac{\phi(q)}2 \bigg( \log q
- \log 2\pi e^{\gamma_0} - \sum_{p\mid q} \frac{\log p}{p-1} -
\frac{4\log q}q - \frac{23.62\log^2 q}q \bigg)} \bigg).
\end{align*}
This upper bound can be
computed exactly for each $q$ in the range $1000\le q\le260000$; the
only five moduli for which the upper bound exceeds 0.75 are 1020,
1320, 1560, 1680, and 1848.

Finally, we use the methods described in Section~\ref{explicit computation of the delta}, computing directly every value of $\delta(q;a,b)$ for the moduli $480 < q \le 1000$ and $q \in \{ 1020, 1320, 1560, 1680, 1848\}$ and verifying the inequality $\delta(q;a,b) < 0.75$ holds except for $q=840$ and $q=1320$, to complete the proof of the theorem.
\end{proof}

\subsection{Explicit computation of the densities}
\label{explicit computation of the delta}

Throughout this section, we assume GRH and LI, and we let $a$ denote a nonsquare\mod q and $b$ a square\mod q. In this section we describe the process by which we computed actual values of the densities $\delta(q;a,b)$, resulting for example in the data given in the tables and figures of this paper. In fact, we used two different methods for these computations, one that works for ``small $q$'' and one that works for ``large~$q$''. For ease of discussion, we define the sets
\begin{align*}
S_1 &= \{ 3\le q\le1000\colon q\not\equiv2\mod4 \text{ and } \phi(q)<80 \} \\
S_2 &= \{101, 103, 107, 109, 113, 115, 119, 121, 123, 125, 129, 133, 141, 143, 145, 147, 153, 155, 159, \\
&\hspace{.76cm} 164, 165, 171, 172, 175, 176, 177, 183, 184, 188, 189, 195, 196, 200, 208, 212, 220, 224, 225,\\
&\hspace{.76cm}   231,  232, 236, 255, 260, 264, 276, 280, 288, 300, 308, 312, 324, 336,  348, 360, 372, 396, 420\} \\
S_3 &= \{ 3\le q\le1000\colon q\not\equiv2\mod4 \text{ and } \phi(q)\ge80 \} \setminus S_2 \\
S_4 &= \{1020, 1320, 1560, 1680, 1848\}.
\end{align*}
We omit integers congruent to 2\mod4 from these sets, since for odd $q$ the prime number race\mod{2q} is identical to the prime number race\mod q.

For the moduli $q$ in the set $S_1\cup S_2$, we numerically evaluated the integral in equation~\eqref{deus ex machina} directly; this method was used by Feuerverger and Martin~\cite{biases} and is analogous to, and indeed based upon, the method used by Rubinstein and Sarnak~\cite{RS}. We first used Rubinstein's computational package {\tt lcalc} to calculate, for each character $\chi\mod q$, the first $N(q)$ nontrivial zeros of $L(s,\chi)$ lying above the real axis. The term $\Phi_{q;a,b}$ in the integrand is a product of functions of the form $F(z,\chi)$, which is indexed by infinitely many zeros of $L(s,\chi)$; we approximated $F(z,\chi)$ by its truncation at $N(q)$ zeros, multiplied by a compensating quadratic polynomial as in~\cite[Section 4.3]{RS}. With this approximation to the integral~\eqref{deus ex machina}, we truncated the range of integration to an interval $[-C(q),C(q)]$ and then discretized the truncated integral, replacing it by a sum over points spaced by $\ep(q)$ as in~\cite[Section 4.1]{RS}. The result is an approximation to $\delta(q;a,b)$ that is valid up to at least 8 decimal places, provided we choose $N(q)$, $C(q)$, and $\ep(q)$ carefully to get small errors. (All of these computations were performed using the computational software {\tt Mathematica}.) Explicitly bounding the error in this process is not the goal of the present paper; we refer the interested reader to~\cite{RS} for rigorous error bounds of this kind, corresponding to their calculation of $\delta(q;N,R)$ for $q\in\{3,4,5,7,11,13\}$.

For the moduli $q$ in the set $S_3\cup S_4$ (and for any other moduli larger than 1000 we wished to address), we used an approach based on our asymptotic formulas for $\delta(q;a,b)$. We now outline a variant of the asymptotic formulas described earlier in this paper, one that was optimized somewhat for the the actual computations rather than streamlined for theoretical purposes.

We first note that a slight modification of the proof of Proposition~\ref{delta comes down to small interval prop} yields the estimate, for any $0\le\kappa\le\frac5{24}$,
\begin{multline}
\delta(q;a,b) = \frac12 + \frac1{2\pi} \int_{-\kappa}^\kappa \frac{\sin \rho(q) x}x \Phi_{q;a,b}(x) \,dx
\\ + \Obar\bigg(\frac1{\pi} \int_\kappa^{5/24} \rho(q) |\Phi_{q;a,b}(x)| \,dx + 0.03506 \frac{e^{-9.08\phi(q)}}{\phi(q)} + 63.67 \rho(q) \big| \Phi_{q;a,b}\big( \tfrac5{24} \big) \big| \bigg)
\label{explicit delta eqn 1}
\end{multline}
as long as $V(q;a,b)\ge 531$. In addition we have, for $|x| < \frac3{10}$, the inequalities
\begin{multline*}
-\tfrac12 V(q;a,b) x^2 - U(q;a,b) x^4 - 15.816 U(q;a,b) x^6 \\
\le \log \Phi_{q;a,b}(x) \le -\tfrac12 V(q;a,b) x^2 - U(q;a,b) x^4,
\end{multline*}
where for convenience we have defined $U(q;a,b) = W_2(q;a,b) V(q;a,b)$; these inequalities can be proved using an argument similar to the calculation in equation~\eqref{saratoga}, but employing the more precise estimate $\log J_0(z) = -z^2/4-z^4/64+\Obar(0.00386z^6)$ for $|z|\le2$.
Using the methods of Section~\ref{arithmetic only}, we also obtain the formula
\begin{multline}
U(q;a,b) = \frac{\phi(q)}2 (3+\iota_q(a^2b^{-2})) \bigg( \log \frac{q}{2\pi e^{-\gamma_0}} - \sum_{p\mid q} \frac{\log p}{p-1} - \frac{\zeta(2)}2 \bigg) \\
+ \frac{\phi(q)}2 \bigg( 4\frac{\Lambda(q/(q,a-b))}{\phi(q/(q,a-b))} - \frac{\Lambda(q/(q,a^2-b^2))}{\phi(q/(q,a^2-b^2)} - \big(\iota_q(-a^2b^{-2})-4\iota_q(-ab^{-1})\big) \Big(\log 2 + \frac{\zeta(2)}4 \Big) \bigg) \\
+ \frac14 \allchisum |\chi(a)-\chi(b)|^4\bigg( 2\frac{L'(1,\chi)}{L(1,\chi)}-\frac{L''(1,\chi)}{L(1,\chi)}+\Big(\frac{L'(1,\chi)}{L(1,\chi)} \Big)^2  \bigg).
\label{U formula}
\end{multline}

If we define $\kappa(q;a,b)=\min(\frac{\pi}{\rho(q)},V(q;a,b)^{-1/4})$, then we know that $\kappa(q;a,b) \le \frac5{24}$ because of the lower bound $V(q;a,b)\ge 531$, and also that $(\sin \rho(q) x)/x$ is nonnegative for $|x|\le \kappa(q;a,b)$. Hence, equation~\eqref{explicit delta eqn 1} and the subsequent discussion establishes the following proposition:

\begin{prop}
\label{explicit bounds delta q about 1000}
Assume GRH and LI, and let $a$ be a nonsquare\mod q and $b$ a square\mod q. If $V(q;a,b)\ge 531$, then
\begin{multline*}
\frac12 + \frac1{2\pi} \int_{-\kappa(q;a,b)}^{\kappa(q;a,b)} \frac{\sin \rho (q) x}x e^{ -V(q;a,b)x^2/2 - U(q;a,b)x^4-15.816 U(q;a,b)x^6 } \,dx - Y(q;a,b) \\
\le \delta(q;a,b) \le \frac12 + \frac1{2\pi} \int_{-\kappa(q;a,b)}^{\kappa(q;a,b)} \frac{\sin \rho (q) x}x e^{ -V(q;a,b)x^2/2 - U(q;a,b)x^4} \,dx + Y(q;a,b),
\end{multline*}
where
\begin{multline*}
Y(q;a,b) = \frac{\rho(q)}{\pi} \int_{\kappa(q;a,b)}^{5/24} e^{ -V(q;a,b)x^2/2 - U(q;a,b)x^4} \,dx \\
+ 0.03506 \frac{e^{-9.08\phi(q)}}{\phi(q)} + 63.67 \rho(q) e^ {-25 V(q;a,b)/1152 - (5/{24})^4 U(q;a,b)}
\end{multline*}
and formulas for $V(q;a,b)$ and $U(q;a,b)$ are given in Theorem~\ref {variance evaluation theorem} and equation~\eqref{U formula}, respectively.
\end{prop}

The inequalities in Proposition~\ref{explicit bounds delta q about 1000} give accurate evaluations of $\delta(q;a,b)$ when $\phi(q)$ is large; we chose the inequality $\phi(q)\ge80$ to be our working definition of ``large''. For each of the moduli $q$ in the set $S_3 \cup S_4$, we computed every possible value of $V(q;a,b)$ and verified that they all exceed $531$, so that Proposition~\ref{explicit bounds delta q about 1000} can be used. (The reason that the moduli in $S_2$ were calculated using the first method, rather than this one, is because at least one variance $V(q;a,b)$ was less than $531$ for each of the moduli in $S_2$.) We then calculated the upper and lower bounds of Proposition~\ref{explicit bounds delta q about 1000}, using numerical integration in {\tt pari/gp}, to obtain all values of $\delta(q;a,b)$. The calculation of $V(q;a,b)$ and $U(q;a,b)$ involve the analytic terms $L(1,\chi)$, $L'(1,\chi)$, and $L''(1,\chi)$; we used the {\tt pari/gp} package {\tt computeL} (see \cite{dokchister}) to obtain these values accurate to 16 decimal places.

\begin{table}[b]
\caption{The 20 smallest values of $\delta(244;a,1)$ and of $\delta(997;a,1)$, calculated using Proposition~\ref{explicit bounds delta q about 1000}}
\smaller\smaller
\begin{tabular}{|c|c|c|c|c|}
\hline
$q$ & $a$ & $a^{-1}$ & $\delta(q;a,1)$ & Error bound \\
\hline
244 & 243 & 243 &        0.558910 &        0.000022\\
244 & 123 & 123  &       0.559000  &       0.000018\\
244 & 3 & 163     &      0.562304     &    0.000020\\
244 & 7 & 35       &     0.563216       &  0.000022\\
244 & 31 & 63       &    0.563543       &  0.000022\\
244 & 153 & 185      &   0.563804      &   0.000021\\
244 & 11 & 111 &         0.564069 &        0.000024\\
244 & 29 & 101  &        0.564124  &       0.000024\\
244 & 17 & 201   &       0.564321   &      0.000023\\
244 & 33 & 37     &      0.564436     &    0.000024\\
244 & 19 & 167     &     0.564741     &    0.000024\\
244 & 23 & 191      &    0.564786      &   0.000023\\
244 & 107 & 187      &   0.565310      &   0.000024\\
244 & 69 & 145  &        0.565319  &       0.000022\\
244 & 53 & 221   &       0.565376   &      0.000022\\
244 & 85 & 89     &      0.565606     &    0.000022\\
244 & 129 & 157    &     0.565683    &     0.000021\\
244 & 173 & 189     &    0.565707     &    0.000023\\
244 & 177 & 193      &   0.565859      &   0.000023\\
244 & 103 & 199       &  0.565861       &  0.000024\\
\hline
\end{tabular}
\hfil
\begin{tabular}{|c|c|c|c|c|}
\hline
$q$ & $a$ & $a^{-1}$ & $\delta(q;a,1)$ & Error bound \\
\hline
997 &   2   &   499 &   0.508116457 &   0.000000014 \\
997 &   5   &   399 &   0.508142372 &   0.000000015 \\
997 &   7   &   285 &   0.508184978 &   0.000000015 \\
997 &   11  &   272 &   0.508238549 &   0.000000016 \\
997 &   17  &   176 &   0.508279881 &   0.000000016 \\
997 &   29  &   722 &   0.508329803 &   0.000000016 \\
997 &   37  &   512 &   0.508345726 &   0.000000016 \\
997 &   41  &   535 &   0.508351018 &   0.000000016 \\
997 &   8   &   374 &   0.508353451 &   0.000000016 \\
997 &   43  &   371 &   0.508355411 &   0.000000016 \\
997 &   47  &   297 &   0.508358709 &   0.000000016 \\
997 &   61  &   474 &   0.508368790 &   0.000000016 \\
997 &   163 &   367 &   0.508392448 &   0.000000016 \\
997 &   103 &   242 &   0.508392587 &   0.000000016 \\
997 &   113 &   150 &   0.508395577 &   0.000000016 \\
997 &   181 &   661 &   0.508397690 &   0.000000016 \\
997 &   127 &   840 &   0.508402416 &   0.000000016 \\
997 &   157 &   870 &   0.508404812 &   0.000000016 \\
997 &   283 &   613 &   0.508406794 &   0.000000016 \\
997 &   179 &   518 &   0.508406994 &   0.000000016 \\
\hline
\end{tabular}
\label{sample data}
\end{table}

Table~\ref{sample data} gives a sample of the data we calculated with this second method, including the error bounds obtained. The error bounds are stronger for when $q$ and $\phi(q)$ are large, explaining why the error bounds for the large prime $q=997$ are so much better than for the smaller composite number $q=244$.

We also take this opportunity to reinforce the patterns described in Section~\ref {impact section}. 
For $q=244$, the residue class $a=123$ has the property that $q/(q,a-1)=2$; thus the contribution of $K_{244}(122)$ to $\Delta(244;123,1)$ reduces the density $\delta(244;123,1)$. 
We see also the familiar small densities corresponding to $a=243\equiv-1\mod{244}$ and to small prime values of $a$. For $q=997$, the small prime values of $a$ (among those that are nonsquares modulo 997) appear in perfect order. We point out that the residue class $a=8$ is almost in its correct limiting position, since the contribution to $\Delta(997;a,1)$ is inversely correlated to $\frac{\Lambda(a)}a$, and $\frac{\Lambda(41)}{41} > \frac{\Lambda(43)}{43} > \frac{\Lambda(8)}8 > \frac{\Lambda(47)}{47}$.

We mention that we undertook the exercise of calculating values $\delta(q;a,b)$ by both methods, for several intermediate values of $q$, as a way to verify our computations. For example, the calculations of $\delta(163;a,b)$ (see Table~\ref {163 data}) were done using the integral formula~\eqref{deus ex machina} as described above. We calculated these same densities using Proposition~\ref{explicit bounds delta q about 1000}; the error bounds obtained were all at most $4.6\times10^{-6}$, and the results of the first calculation all lay comfortably within the intervals defined by the second calculation.

\begin{table}[h]
\caption{The top 120 most unfair prime number races}
\smaller\smaller
\begin{tabular}{|c|c|c|c|}
\hline
\hphantom{0}$q$\hphantom{0} & \hphantom{0}$a$\hphantom{0} & $a^{-1}$ & $\delta(q;a,1)$ \\
\hline
24  &   5   &   5   &   0.999988   \\
24  &   11  &   11  &   0.999983   \\
12  &   11  &   11  &   0.999977   \\
24  &   23  &   23  &   0.999889   \\
24  &   7   &   7   &   0.999834   \\
24  &   19  &   19  &   0.999719   \\
8   &   3   &   3   &   0.999569   \\
12  &   5   &   5   &   0.999206   \\
24  &   17  &   17  &   0.999125   \\
3   &   2   &   2   &   0.999063   \\

\hline
\end{tabular}
\begin{tabular}{|c|c|c|c|}
\hline
\hphantom{0}$q$\hphantom{0} & \hphantom{0}$a$\hphantom{0} & $a^{-1}$ & $\delta(q;a,1)$ \\
\hline
8   &   7   &   7   &   0.998939   \\
24  &   13  &   13  &   0.998722   \\
12  &   7   &   7   &   0.998606   \\
8   &   5   &   5   &   0.997395   \\
4   &   3   &   3   &   0.995928   \\
120 &   71  &   71  &   0.988747   \\
120 &   59  &   59  &   0.988477   \\
60  &   11  &   11  &   0.987917   \\
60  &   29  &   29  &   0.986855   \\
120 &   109 &   109 &   0.986835   \\

\hline
\end{tabular}
\begin{tabular}{|c|c|c|c|}
\hline
\hphantom{0}$q$\hphantom{0} & \hphantom{0}$a$\hphantom{0} & $a^{-1}$ & $\delta(q;a,1)$ \\
\hline
60  &   19  &   19  &   0.986459   \\
120 &   89  &   89  &   0.986364   \\
120 &   79  &   79  &   0.986309   \\
120 &   101 &   101 &   0.984792   \\
15  &   2   &   8   &   0.983853   \\
120 &   13  &   37  &   0.980673   \\
40  &   19  &   19  &   0.980455   \\
60  &   7   &   43  &   0.979323   \\
120 &   23  &   47  &   0.979142   \\
15  &   14  &   14  &   0.979043   \\

\hline
\end{tabular}
\begin{tabular}{|c|c|c|c|}
\hline
\hphantom{0}$q$\hphantom{0} & \hphantom{0}$a$\hphantom{0} & $a^{-1}$ & $\delta(q;a,1)$ \\
\hline
120 &   17  &   113 &   0.978762   \\
120 &   7   &   103 &   0.978247   \\
48  &   23  &   23  &   0.978096   \\
120 &   43  &   67  &   0.978013   \\
60  &   17  &   53  &   0.977433   \\
48  &   41  &   41  &   0.977183   \\
40  &   29  &   29  &   0.977161   \\
20  &   3   &   7   &   0.976713   \\
120 &   53  &   77  &   0.976527   \\
60  &   23  &   47  &   0.975216   \\

\hline
\end{tabular}
\begin{tabular}{|c|c|c|c|}
\hline
\hphantom{0}$q$\hphantom{0} & \hphantom{0}$a$\hphantom{0} & $a^{-1}$ & $\delta(q;a,1)$ \\
\hline
120 &   91  &   91  &   0.975051   \\
120 &   83  &   107 &   0.975001   \\
120 &   29  &   29  &   0.974634   \\
120 &   19  &   19  &   0.974408   \\
120 &   11  &   11  &   0.971988   \\
48  &   31  &   31  &   0.970470   \\
40  &   7   &   23  &   0.969427   \\
40  &   13  &   37  &   0.969114   \\
120 &   73  &   97  &   0.967355   \\
20  &   19  &   19  &   0.966662   \\

\hline
\end{tabular}
\begin{tabular}{|c|c|c|c|}
\hline
\hphantom{0}$q$\hphantom{0} & \hphantom{0}$a$\hphantom{0} & $a^{-1}$ & $\delta(q;a,1)$ \\
\hline
15  &   7   &   13  &   0.964719   \\
120 &   31  &   31  &   0.963190   \\
60  &   13  &   37  &   0.963058   \\
60  &   59  &   59  &   0.962016   \\
40  &   31  &   31  &   0.960718   \\
48  &   5   &   29  &   0.960195   \\
40  &   3   &   27  &   0.960099   \\
16  &   7   &   7   &   0.959790   \\
48  &   11  &   35  &   0.959245   \\
120 &   119 &   119 &   0.957182   \\

\hline
\end{tabular}
\begin{tabular}{|c|c|c|c|}
\hline
\hphantom{0}$q$\hphantom{0} & \hphantom{0}$a$\hphantom{0} & $a^{-1}$ & $\delta(q;a,1)$ \\
\hline
15  &   11  &   11  &   0.955226   \\
120 &   41  &   41  &   0.955189   \\
48  &   19  &   43  &   0.952194   \\
5   &   2   &   3   &   0.952175   \\
20  &   13  &   17  &   0.948637   \\
120 &   61  &   61  &   0.948586   \\
60  &   41  &   41  &   0.947870   \\
16  &   3   &   11  &   0.947721   \\
48  &   13  &   37  &   0.946479   \\
40  &   17  &   33  &   0.946002   \\

\hline
\end{tabular}
\begin{tabular}{|c|c|c|c|}
\hline
\hphantom{0}$q$\hphantom{0} & \hphantom{0}$a$\hphantom{0} & $a^{-1}$ & $\delta(q;a,1)$ \\
\hline
40  &   11  &   11  &   0.945757   \\
40  &   39  &   39  &   0.942554   \\
60  &   31  &   31  &   0.941802   \\
48  &   7   &   7   &   0.939000   \\
16  &   5   &   13  &   0.938369   \\
168 &   125 &   125 &   0.936773   \\
168 &   155 &   155 &   0.935843   \\
168 &   47  &   143 &   0.932099   \\
168 &   61  &   157 &   0.931981   \\
84  &   41  &   41  &   0.931702   \\

\hline
\end{tabular}
\begin{tabular}{|c|c|c|c|}
\hline
\hphantom{0}$q$\hphantom{0} & \hphantom{0}$a$\hphantom{0} & $a^{-1}$ & $\delta(q;a,1)$ \\
\hline
20  &   11  &   11  &   0.931367   \\
168 &   139 &   139 &   0.931362   \\
168 &   55  &   55  &   0.931346   \\
48  &   47  &   47  &   0.929478   \\
168 &   67  &   163 &   0.928944   \\
84  &   71  &   71  &   0.928657   \\
168 &   41  &   41  &   0.927933   \\
84  &   55  &   55  &   0.927755   \\
168 &   71  &   71  &   0.927349   \\
16  &   15  &   15  &   0.926101   \\

\hline
\end{tabular}

\begin{tabular}{|c|c|c|c|}
\hline
\hphantom{0}$q$\hphantom{0} & \hphantom{0}$a$\hphantom{0} & $a^{-1}$ & $\delta(q;a,1)$ \\
\hline
168 &   65  &   137 &   0.923960   \\
168 &   53  &   149 &   0.923937   \\
168 &   83  &   83  &   0.923868   \\
21  &   5   &   17  &   0.923779   \\
168 &   79  &   151 &   0.922597   \\
40  &   21  &   21  &   0.922567   \\
168 &   37  &   109 &   0.922359   \\
168 &   17  &   89  &   0.920542   \\
48  &   17  &   17  &   0.918910   \\
56  &   27  &   27  &   0.918015   \\

\hline
\end{tabular}
\begin{tabular}{|c|c|c|c|}
\hline
\hphantom{0}$q$\hphantom{0} & \hphantom{0}$a$\hphantom{0} & $a^{-1}$ & $\delta(q;a,1)$ \\
\hline
168 &   59  &   131 &   0.917874   \\
168 &   23  &   95  &   0.917718   \\
168 &   31  &   103 &   0.917278   \\
168 &   29  &   29  &   0.915514   \\
72  &   53  &   53  &   0.913533   \\
21  &   2   &   11  &   0.911872   \\
168 &   19  &   115 &   0.911412   \\
168 &   11  &   107 &   0.909850   \\
168 &   73  &   145 &   0.908239   \\
168 &   5   &   101 &   0.908206   \\

\hline
\end{tabular}
\begin{tabular}{|c|c|c|c|}
\hline
\hphantom{0}$q$\hphantom{0} & \hphantom{0}$a$\hphantom{0} & $a^{-1}$ & $\delta(q;a,1)$ \\
\hline
56  &   31  &   47  &   0.906135   \\
84  &   67  &   79  &   0.905578   \\
168 &   13  &   13  &   0.904525   \\
168 &   97  &   97  &   0.904162   \\
72  &   35  &   35  &   0.903755   \\
84  &   47  &   59  &   0.902413   \\
56  &   37  &   53  &   0.900863   \\
84  &   53  &   65  &   0.899063   \\
28  &   11  &   23  &   0.898807   \\
168 &   127 &   127 &   0.898647   \\
\hline
\end{tabular}

\label{120 most biased table}
\end{table}

Finally, the upper bounds for $\delta(q;a,b)$ in Theorems~\ref{delta bound prime q theorem} and~\ref{delta<0.75 theorem}, together with the explicit calculation of the densities $\delta(q;a,b)$ for $q\in S_1\cup S_2\cup S_3\cup S_4$, allow us to determine the most biased possible two-way races, that is, the largest values of $\delta(q;a,b)$ among all possible choices of $q$, $a$, and~$b$. In particular, we verified Theorem~\ref{top ten thm} in this way, and we list the 120 largest densities in Table~\ref{120 most biased table}; there are precisely 117 distinct densities above $\frac9{10}$. (It is helpful to recall here that $\delta(q;a,1) = \delta(q;a^{-1},1)$ and that $\delta(q;a,1) = \delta(q;ab,b)$ for any nonsquare $a$ and square $b$ modulo~$q$.)

\providecommand{\bysame}{\leavevmode\hbox to3em{\hrulefill}\thinspace}
\providecommand{\MR}{\relax\ifhmode\unskip\space\fi MR }
\providecommand{\MRhref}[2]{%
  \href{http://www.ams.org/mathscinet-getitem?mr=#1}{#2}
}
\providecommand{\href}[2]{#2}

\end{document}